\documentclass[a4paper,11pt,reqno]{amsart}

\usepackage{amsmath,amssymb,amscd,amsthm,amsxtra, esint,mathtools}
\usepackage{mathabx}

\usepackage{mathrsfs} 
\usepackage{enumerate}
\usepackage{times}
\usepackage{bbold}
\allowdisplaybreaks

\usepackage[implicit=true]{hyperref}

\usepackage{color}

\let\emptyset \undefined
\newsymbol\emptyset    203F

\theoremstyle{plain}
\newtheorem{theorem}{Theorem}[section]
\newtheorem{corollary}[theorem]{Corollary}
\newtheorem{lemma}[theorem]{Lemma}
\newtheorem{proposition}[theorem]{Proposition}

\theoremstyle{remark}

\numberwithin{equation}{section}

\newcommand\numberthis{\stepcounter{equation}\tag{\theequation}}

\newcommand{\E}{\mathbb{E}}
\newcommand{\N}{\mathbb{N}}
\newcommand{\C}{\mathbb{C}}
\renewcommand{\P}{{\mathbb P}}
\newcommand{\R}{\mathbb{R}}

\newcommand{\Z}{\mathbb{Z}}
\newcommand{\1}{\mathbb 1}
\newcommand{\wt}{\widetilde}
\newcommand{\ta}{\theta}
\newcommand{\conj}{\overline}
\newcommand{\CGN}{C_{\mathrm{GNS}}}

\newcommand{\norm}[1]{\left\|#1\right\|}

\newcommand{\T}{\mathbb{T}}

\newcommand{\D}{\mathcal D}

\newcommand{\NN}{\mathbf{N}}

\newcommand{\eps}{\varepsilon}
\newcommand{\ph}{\varphi}
\renewcommand{\leq}{\leqslant}
\renewcommand{\geq}{\geqslant}
\newcommand{\les}{\lesssim}

\newcommand{\dl}{\delta}
\renewcommand{\H}{{\mathbb H}^1_a}

\DeclareMathOperator{\loc}{loc}
\DeclareMathOperator{\supp}{supp}
\DeclareMathOperator{\Law}{Law}

\newcommand{\jb}[1]{\langle #1 \rangle}

\let\Re\relax
\DeclareMathOperator{\Re}{Re}

\let\Im\relax
\DeclareMathOperator{\Im}{Im}

\usepackage{marginnote}

\begin{document}

\title[Phase transitions for focusing NLS]
{Phase transition for invariant measures of the focusing Schr\"odinger equation}

\author{Leonardo Tolomeo}
\address{Leonardo Tolomeo, University of Edinburgh}
\email{l.tolomeo@ed.ac.uk}

\author{Hendrik Weber}
\address{Hendrik Weber, University of M\"unster
}
\email{hendrik.weber@uni-muenster.de}

%Additional authors

%\author{}
%\address{}
%\email{}

%\keywords{}

%\subjclass[2000]{Primary ; Secondary:}

 \begin{abstract}
We consider the Gibbs measure for the focusing nonlinear Schr\"odinger equation on the one-dimensional torus $\mathbb T$, that was introduced in a seminal paper by Lebowitz, Rose and Speer (1988). We show that in the large torus limit, the measure exhibits a phase transition, depending on the size of the nonlinearity. This phase transition was originally conjectured on the basis of numerical simulation by Lebowitz, Rose and Speer (1988). Its existence is however striking in view of a series of negative results by McKean (1995) and Rider (2002). 
\end{abstract}

%\dedicatory{}

\date\today

\maketitle

%%%%%%%%%%%%%%%%%%%%%%%%
\section{Introduction}
%%%%%%%%%%%%%%%%%%%%%%%%
\label{s:Intro}
%%%%%%%%%%%%%%%%%%%%%%%%

In this paper, we continue the study of the focusing Gibbs measures
for the nonlinear Schr\"odinger equations (NLS), 
initiated in the seminal papers by Lebowitz, Rose, and Speer~\cite{Lebowitz1988}.
%and Bourgain \cite{BO94}.
A Gibbs measure $\rho$ %on a given spatial domain $\M$ 
denotes a probability measure on 
functions\,/\,distributions with formal density:
\begin{align*}
d\rho = Z^{-1} e^{- H(u)} d u , 
%\label{Gibbs1}
\end{align*}
\noindent
where  $H(u)$ is a Hamiltonian functional and $Z$ is a normalisation constant, called partition function. 
We will focus our attention to the focusing nonlinear Schr\"odinger equation (NLS), 
\begin{align}
i u_t +  \Delta u + \beta |u|^{p-2} u = 0, 
\label{NLS1}
\end{align}
for which the Hamiltonian $H(u)$ is given by 
\begin{align}
H(u) = \frac 12\int |\nabla u|^2 d x - \frac \beta p\int |u|^p d x.
\label{Gibbs2}
\end{align}
%The NLS equation  \hw{is this really necessary/good?}
% has been studied extensively as model for  
% various physical phenomena ranging from Langmuir waves in plasmas to signal propagation in optical fibers \cite{SS, KMM, Agrawal}. 
 The study of the equation~\eqref{NLS1}
 from the point of view of the (non-)equilibrium statistical mechanics 
 has received wide attention; see for example \cite{Lebowitz1988, Bo94, BO96, BO97,Tzv1,  Tzv2, LMW, BBulut, 
CFL, DNY2, Bring3}. See also \cite{BOP4} for a survey on the subject, more from the dynamical point of view.

The main difficulty in constructing the focusing Gibbs measures
comes from the high degree $p>2$ of the negative term in the Hamiltonian \eqref{Gibbs2}.
This makes the problem extremely different from the defocusing case, which would correspond to the well studied $\Phi^p$ model of quantum field theory. In \cite{Lebowitz1988}, Lebowitz, Rose, and Speer suggested
to consider the Gibbs measure with an extra $L^2$-cutoff 
\begin{align}
d\rho = Z_\NN^{-1} e^{- H(u)} \1_{\{M(u) \le \NN\}} d u,
\label{Gibbs3}
\end{align}
where $M(u)$ denotes the mass functional 
\begin{equation}
M(u) = \int |u|^2. \label{mass}
\end{equation}
This choice is suitable for the study of the statistical mechanics of NLS, since the mass $M(u)$ is conserved by the flow of NLS, and represents the ``generalised number of particles". 
%Hence, the measure \eqref{Gibbs3} represents a generalised grand-canonical ensemble for NLS. \hw{I find this wording odd}

Our main goal in this paper is to study the Gibbs measure \eqref{Gibbs3}
on the one-dimensional torus with size $L$
$$L\T  = \R/(L\Z) \cong [-L/2, L/2],$$
and describe its behaviour in the large-torus limit $L \to \infty$, depending on the parameter $\beta$. More precisely, the main result of this paper is the following. 
\begin{theorem}\label{thm:list}
Let $2 < p < 6$, $\alpha, \NN > 0$, and consider the measure $\rho_L$ formally given by
\begin{align}
d\rho_L(u) = \frac{1}{Z_{\NN,L}} \exp\Big(\frac{\beta}{L^\gamma} \int_{L\T} |u|^p - \frac12 \int_{L\T} |\partial_x u|^2 - \frac\alpha2 \int_{L\T} |u|^2 \Big) \1_{\{M(u) \le \NN L\}} d u.
\label{Gibbs}
\end{align}
Then the following statements hold. 
\begin{itemize}
\item[\textup(i)] \textup{(supercritical case)} Let $\gamma < \frac p2-1$. Then the measure $\rho_L$ converges to the probability measure $\delta_{u=0}$. More precisely, one has the following concentration phenomenon. There exists a function $Q_p \in S(\R)$, depending only on the parameters $\beta, \NN$, and a rescaling $Q_p^L$ of $Q_p$, such that for an appropriate norm, one has
$$ \rho_L\Big( \Big\{ \inf_{x_0\in \R, \ta \in [0,2\pi]} \big \| u - e^{i\ta}L^\frac 12 Q_p^L(\cdot - x_0)\big\| > o(1)\Big) = o(1), $$
and $ L^\frac 12 Q_p^L $ converges to $0$ in distribution.
\item[\textup{(ii)}] \textup{(critical case, strong nonlinearity)} Let $\gamma = \frac p2-1$. Then there exists $\beta_1 = \beta_1(\alpha, p, \NN) \gg 1$ such that for every $\beta > \beta_1$, no limit point of $\rho_L$ as $L\to \infty$ is equal 
to the Ornstein-Uhlenbeck measure $\mu_{OU}$ on $\R$ with mass parameter $\alpha$, formally given by 
\begin{equation}
d\mu_{OU}(u) = \frac{1}{Z} \exp\Big(- \frac12 \int_{\R} |\partial_x u|^2 - \frac\alpha2 \int_{\R} |u|^2 \Big) \Big) du.
\end{equation}
%(in the topology of weak convergence of probability measures on the space of distributions).
\item[\textup{(iii)}] \textup{(critical case, weak nonlinearity)} Suppose that $\NN > \frac{1}{2\sqrt\alpha}$, and $\gamma = \frac p2 -1$. Then there exists $\beta_0 = \beta_0(\alpha, p, \NN) \ll 1$ such that for each $\beta < \beta_0$, $\rho_L$ weakly converges 
%(as a probability measure on the space of distributions) 
to  $\mu_{OU}$. 
%
% the Ornstein-Uhlenbeck measure $\mu_{OU}$ on $\R$ with mass parameter $\alpha$, formally given by 
%\begin{equation}
%d\mu_{OU}(u) = \frac{1}{Z} \exp\Big(- \frac12 \int_{\R} |\partial_x u|^2 - \frac\alpha2 \int_{\R} |u|^2 \Big) \Big) du.
%\end{equation}
\item
[\textup{(iv)}] \textup{(subcritical case)} Suppose that $\NN > \frac{1}{2\sqrt\alpha}$. For $\gamma > \frac p2 -1$, the measure $\rho_L$ weakly converges to same measure $\mu_{OU}$.
\end{itemize}
In all situations above, the convergence is to be understood  as weak convergence in the space of probability measures over distributions over $\R$, where we see distributions over $L\T \cong [-L/2,L/2]$ as distributions over $\R$ via their periodic extension.
\end{theorem}

A more precise version of statements (i)-(iv) will be provided in what follows. In particular 
\begin{itemize}
\item Theorem \ref{thm:list}, (i) follows from Theorem \ref{thm:2} and Corollary \ref{triviality}, (1).
\item Theorem \ref{thm:list}, (ii) is expanded upon in Theorem \ref{thm:2b} and follows from Corollary \ref{triviality}, (2).
\item Theorem \ref{thm:list}, (iii) and (iv) are the content of Theorem \ref{thm:3}.
\end{itemize}

In this context, the choice of imposing a mass threshold proportional to the size $L$ of the torus is fairly natural, due to the translation-invariance of the measure. %Indeed, we do not expect it to be possible to obtain a non-trivial translation-invariant limit without imposing this condition.  \hw{does this sentence have a deeper meaning I don't understand?}
Moreover, regarding the convergence to the Ornstein-Uhlenbeck process in (iii) and (iv), the restriction $\NN > \frac{1}{2\sqrt\alpha}$ seems to be substantial, since one has that 
$$ \mu_{OU}\Big(\Big\{ \Big|\int_{-L/2}^{L/2} |u|^2 - \frac{L}{2\sqrt\alpha}\Big| \gtrsim L \Big)\Big\} \to 0$$
exponentially fast as $L \to \infty$, see \eqref{LDE}. Therefore, for $\NN < \frac{1}{2\sqrt\alpha}$, we do not expect convergence to the Ornstein-Uhlenbeck measure (with the same $\alpha$), and that the behaviour of the measure $\rho_L$ to be dominated by mass cutoff, even (slightly) above the parameter $\beta_0$.  

The most striking result in Theorem \ref{thm:list} is the phase transition happening at the critical case. The existence and description of this phase transition has a somewhat troubled history. On the basis of numerical simulations, in their original work, Lebowitz, Rose, and Speer suggested that such a transition might occur, depending on the particular value of a certain quantity, see \cite[Section 3.3]{Lebowitz1988}. 
However, in our context, this quantity corresponds exactly to $(\NN L)^{-1}$, which in view of Theorem \ref{thm:list}, it is not the critical quantity that determines the phase transition.\footnote{In order to relate the quantities of the two papers, one just needs to make the change of variable $u = T^\frac12 v$.} 
This problem was then tackled by McKean in \cite{McKean1995}, who considered the case $\gamma = 0$, $p=4$. However, his proof contained a mistake, and provided an incorrect result. 
Finally, the problem with $\gamma=0, p=4$ was solved by Rider \cite{doi:10.1002/cpa.10043}, who showed that no phase transition occurs and that the measure converges to the trivial measure $\delta_{u=0}$, by showing a slightly weaker version of concentration phenomenon provided by  Theorem \ref{thm:list}, (i). 

\subsection{Construction of the measure and variational formulation}\label{Sec:GaussianMeasure}
Consider the torus $L\T$, and recall that the Gaussian measure with inverse covariance $\alpha-\Delta$, i.e.\ the measure formally given by 
\begin{equation*}
\mu_L =  \frac{1}{Z} \exp\Big( -\frac12 \int_{L\T} |\partial_x u|^2 - \frac{\alpha}2 \int |u|^2\Big) du d\overline u,
\end{equation*}
can be seen as the law of the random series 
\begin{equation}\label{random_fourier0}
u(x) = \sum_{k\in \Z} \frac{g_k}{\sqrt{\alpha + \lambda_k^2}} e_k(x),
\end{equation}
where $g_k$ are i.i.d., standard complex-valued Gaussian random variables, and $\{e_k\}_{k\in \Z}$ is an $L^2$-orthonormal basis for the Laplacian $-\Delta$ with eigenvalues $\lambda_k^2$. By identifying $L\T$ with the interval $[-L/2,L/2]$ 
and imposing periodic boundary conditions, we see that 
\begin{gather*}
e_k(x) = \frac{1}{\sqrt L} \exp\big(2\pi i \frac{kx}{L} \big),\\
\lambda_k^2 = 4\pi^2 \frac{k^2}{L^2}.
\end{gather*}
Therefore, \eqref{random_fourier0} becomes 
\begin{equation}\label{random_fourier}
u(x)  = \sum_{k\in\Z} \frac{g_k}{\sqrt{\alpha + 4\pi^2\big(\frac{k}{L}\big)^2}} \frac{e^{i\frac{2\pi kx}{L}}}{\sqrt L}.
\end{equation}
Therefore, the measure \eqref{Gibbs} can be expressed as 
\begin{equation}\label{rhoLdef}
\rho_L = \frac{1}{Z}\exp\Big(\frac{\beta}{L^\gamma} \int_{L\T} |u|^p\Big) \1_{\{M(u) \le L\NN\}} \mu_L,
\end{equation}
whenever the partition function 
\begin{equation} \label{partition_function}
 Z= Z(\beta,L,\NN) = \int \exp\Big(\frac{\beta}{L^\gamma} \int_{L\T} |u|^p\Big) \1_{\{M(u) \le L\NN\}} d\mu_L(u)
\end{equation}
is finite. With this definitions, the following result holds. 
\begin{theorem}[\cite{Lebowitz1988, Bo94, OST}]\label{thm:1}
Let $\gamma = 0$, and assume either
\begin{itemize}
\item[(i)] $p<6$ and $\NN >0$, or
\item[(ii)] $p=6$ and $\NN \le \NN_0$,
\end{itemize}
where 
\begin{equation}\label{N0}
\NN_0^2 = \frac{6}{2\beta \CGN}, 
\end{equation}
and $\CGN=\CGN(6)$ denotes the optimal constant in the Gagliardo-Nirenberg-Sobolev inequality  
$$ \int_{\R} |u|^6 \le \CGN M(u)^2 \int |\partial_x u|^2. $$
Then 
\begin{equation} \label{eqn: Zfinite}
Z(\beta,L,\NN) < \infty,
\end{equation}
otherwise 
\[
Z(\beta,L,\NN) = \infty.
\]
\end{theorem}
This result was first shown by Lebowitz, Rose and Speer in \cite{Lebowitz1988} with the exception of the threshold case $\NN = \NN_0$. However, the proof of the finiteness of the partition function contained a mistake, that was later corrected by Bourgain in \cite{Bo94} in the case $p<6$ or $p=6, \NN \ll 1$, 
and finally the case $p= 6$, $\NN \le \NN_0$ was proved by Oh, Sosoe and the first author in \cite{OST}. 
We point out that the threshold case $\NN = \NN_0$ is surprising due to the existence of finite-time blowup solutions of \eqref{NLS1} when $p=6$ and $M(u) = \NN_0$, and the proof of the boundedness of the partition function required a very precise decomposition of the Gaussian measure in a neighbourhood of the ``soliton manifold" $\{ Q_{\dl,x_0,\ta} \}_{\dl >0, x_0 \in \R, \ta >0}.$ We will not discuss the issue further here, and refer the interested readers to \cite{OST}.

The main technical novelty of this paper, is providing a different characterisation of the partition function \eqref{partition_function}, based on the Bou\'e-Dupuis variational formula introduced in \cite{BD}. 
Namely, this formula states that if $\mu$ is a function-valued Gaussian measure and $\ph(t)$ is a Brownian motion with $\Law(\ph(1)) = \mu$, and $F$ is a real-valued functional, then under appropriate hypotheses, 
\begin{align*}
&\log\Big(\int \exp\big(F(u)\big) d\mu_L(u)\Big) \\
&= \sup_{h:[0,1] \to H, \atop h \text{ progressively measurable }} \E\Bigg[ F\Big(\ph(1) +  \int_0^1 h(s) ds\Big) - \frac 12 \int_0^1 \|h(s)\|_{H}^2 ds\Bigg],
\end{align*}
where here $H$ denotes the Cameron-Martin space of the Gaussian measure $\mu$, and the filtration with respect to which $h$ needs to be progressively measurable is the natural filtration induced by $\ph$.
Via the change of variable 
$$ V(t) = \int_0^t h(s) ds, $$
and specialising the formula to the Gaussian measure $\mu_L$, this formula becomes 
\begin{equation} \label{BDintro}
\begin{aligned}
&\log\Big(\int \exp\big(F(u)\big) d\mu(u)\Big) \\
&= \sup_{V:[0,1] \to H, \atop V \text{ progressively measurable }} \E\Bigg[ F\Big(\ph(1) +  V(1)\Big) - \frac 12 \int_0^1 \| \partial_x \dot V(s)\|_{L^2}^2 + \alpha \| \dot V(s) \|_{L^2}^2 ds\Bigg].
\end{aligned}
\end{equation}
We will present the particular version of the Bou\'e-Dupuis formula that we are going to use in this paper in Proposition \ref{LEM:BD} and Proposition \ref{prop: BDinf}.

This formula has been first introduced in the context of constructing Gibbs-type measures in the work by Barashkov and Gubinelli \cite{barashkov2018variational}, who used it to give a new and very elegant proof for ultraviolet stability of the $\Phi^4_3$ model of quantum field theory. 
After this breakthrough, the formula immediately gained popularity, and it has since been used extensively to develop further understanding of these measures, see \cite{GOTW, CGW, OOT1, OOT2, OST2, FoTo, RSTW, barashkov2022variational, barashkov2022stochastic, barashkov2022multiscale, barashkov2023invariance, Bring, bringmann2022gibbs}. 

As a first step of this paper, we will show that a formula similar to \eqref{BDintro} holds in the context of the measures $\rho_L$. As a byproduct, this will provide an extremely simple proof of Theorem \ref{thm:1},\footnote{With the exception of the threshold case $\NN = \NN_0$.} see Proposition \ref{prop:1a} and Proposition \ref{prop:1b}. This analysis is performed in Section 3. 

\subsection{Supercritical case and strongly nonlinear critical case: concentration around a single soliton}

We now explain the source of the triviality 
in Theorem \ref{thm:list}, (i) and give a better description of Theorem \ref{thm:list}, (ii). Define the quantity 
\begin{equation} \label{min:eqn}
A(\beta,\NN) := \inf_{\norm{u}_{L^2(\R)}^2 \le \NN} - \frac \beta p \int |u|^p + \frac12 \int |\partial_xu|^2.
\end{equation}

We apply the formula \eqref{BDintro} to the functional  \footnote{The argument at the beginning of Section \ref{ss:MAIN_PROOF}  shows that the indicator function can be moved into the exponential. }
$$ F(u) = \1_{\{M(u)\le \NN L\}} \frac \beta {pL^\gamma} \int_{L\T} |u|^p.$$
This regime is characterised by the fact that  the term $\ph(1)$ can be treated as a perturbation.
Therefore, we obtain that  
\begin{align*}
&\log Z(\beta,L,\NN) \\
&\approx \sup_V \E\Bigg[ \1_{\{M(V(1))\le \NN L\}} \frac \beta {pL^\gamma} \int_{L\T} |V(1)|^p - \frac 12 \int_0^1 \| \partial_x \dot V(s)\|_{L^2}^2 + \alpha \| \dot V(s) \|_{L^2}^2 ds\Bigg] \\
&\approx \sup_V \Bigg( \1_{\{M(V)\le \NN L\}} \frac \beta {pL^\gamma} \int_{L\T} |V|^p -\frac12 \|\partial_x V\|_{L^2}^2 - \frac\alpha2 \|V\|_{L^2}^2\Bigg),
\intertext{and after applying the scaling $ V = L^\frac12 \lambda^{-\frac12} W(\lambda^{-1} \cdot)$ for $\lambda^{-1} = L^{\frac{p-2-2\gamma}{6-p}}$, we obtain for $L\gg1$,}
&\approx L^{\frac{p+2-4\gamma}{6-p}} \sup_W \Big(  \1_{\{M(W)\le \NN \}} \frac \beta p \int_\R |W|^p - \frac12 \|\partial_x W\|_{L^2}^2 \Big)\\
&= - L^{\frac{p+2-4\gamma}{6-p}} A(\beta,\NN).
\end{align*}
This suggests that, as long as the perturbation coming from $\ph$ can be seen as a small order perturbation, then the quantity $Z(\beta,L,\NN)$ has a very precise asymptotic, 
and the profile of the perturbation $V$ in \eqref{BDintro} should be close to a rescaled version of an optimiser for \eqref{min:eqn}. By making this argument rigorous, in the supercritical case, we obtain the following statement, that provides both a good asymptotic for the partition function $Z(\beta,L,\NN)$, and shows a concentration phenomenon for the measure $\rho_L$ around optimisers for \eqref{min:eqn}.
\begin{theorem}[Supercritical case]\label{thm:2}
Let $2 < p < 6$, and let $\alpha, \NN > 0$.  Let 
\begin{equation}
Z_L := \E\left[\exp\Big(\frac\beta {pL^\gamma} \int |\ph|^p\Big) \1_{\{M(\ph) \le \NN L\}}\right],
\end{equation}
on the torus of size $L$, with $\gamma < \frac p2 -1$. 
Here $\E$ denotes expectation with respect to the measure $\mu_L$.
Then 
\begin{equation}\label{thm2: asymptotic}
\lim_{L \to \infty} \frac{\log Z_L}{L^{\frac{p+2-4\gamma}{6-p}}} = -A(\beta,\NN).
\end{equation}
Moreover, let $Q_p \in H^1(\R)$ be the unique\footnote{Up to translations and multiplications by a unitary complex number.} optimiser for $A(\beta,\NN)$, and let $\delta > 0$. For fixed $q < \infty$, consider the strip
\begin{align*}
S_\delta := \Big\{ u: &\exists x_0 \in [-L/2, L/2] \text{ s.t. }\\
&\norm{L^{-\frac12} \lambda^\frac 12 u(\lambda (\cdot- x_0)) -Q_p}_{L^2 \cap L^q([-L/2,L/2])} < \dl, \lambda^{-1} = L^{\frac{p-2-2\gamma}{6-p}}\Big\},
\end{align*}
where $x-x_0$ is to be understood modulo $L$.
Then the measure concentrates on the strip $S_\delta$; more precisely,
\begin{equation} \label{concentration}
\lim_{L\to \infty} \frac{1}{Z_L} \E\left[\exp\big(\frac\beta {pL^\gamma} \int |\ph|^p\big) \1_{\{M(\ph) \le \NN L\}} \1_{\{\ph \not \in S_\delta \}}\right] = 0.
\end{equation}
\end{theorem}
The uniqueness of the function $Q_p$ is shown in Lemma \ref{minimisers}, \eqref{minimisers1}.

The situation in the critical case is more delicate, the reason being that one cannot see the term $\ph(1)$ in \eqref{BDintro} as a small order perturbation anymore. This can be observed using the following heuristic. 
If we want the term $\ph(1) + V(1)$ to be close to a rescaled version of $Q_p$ with mass $\approx L\NN$, 
because of the rapid decay of $Q_p$ at infinity, $V$ must be close to $-\ph$ on much of the domain, and therefore one must have that 
$$ \Big|\int \ph(1)\overline{V(1)}\Big| \gtrsim L. $$
In turn, this implies that 
$$ \alpha \int |V|^2 \gtrsim L $$
as well. 
Moreover, we have that $L^{\frac{p+2-4\gamma}{6-p}} A(\beta,\NN) = LA(\beta,\NN)$, so the two terms have the same order. 
We add here that the role of $\alpha>0$ is not substantial in this heuristic, and one could repeat similar considerations by using only the term $\int |\partial_x V|^2$. 
This suggests that the asymptotic \eqref{thm2: asymptotic} and the concentration \eqref{concentration} do not hold when $\gamma = \frac p2-1$. However, we can recover a similar result in the asymptotic $\beta \to \infty$, which is the content of the following theorem.   
\begin{theorem}[Critical case, strong nonlinearity]\label{thm:2b}
Let $2 < p < 6$, and let $\alpha, \NN > 0$.  Let
\begin{equation}
Z_L := \E\left[\exp\Big(\frac\beta {pL^{\frac p2-1}} \int |\ph|^p\Big) \1_{\{M(\ph) \le \NN L\}}\right],
\end{equation}
on the torus of size $L$. Then 
\begin{equation}\label{thm2b: asymptotic}
\liminf_{L \to \infty} \frac{\log Z_L}{L} = -A(\beta, \NN)(1 + o(1)).
\end{equation}
as $\beta \to \infty$.
Moreover, let $Q_p \in H^1(\R)$ realise the maximum in the RHS of \eqref{thm2: asymptotic}, and let $\delta > 0$. For fixed $q < \infty$, consider the strip
\begin{align*}
S_\delta := \Big\{ u: &\exists x_0 \in [-L/2, L/2] \text{ s.t. }\\
&\norm{L^{-\frac12} \lambda^\frac 12 u(\lambda (\cdot- x_0)) -Q_p}_{L^2 \cap L^q(\T)} < \dl, \lambda = \beta^{\frac{4}{6-p}}\Big\},
\end{align*}
where $x-x_0$ is to be understood modulo $L$.
Then the measure concentrates on the strip $S_\delta$ as $\beta \to \infty$; more precisely, for every $\delta > 0$, there exists $\beta_0 = \beta_0(\dl) \gg 1$ such that for every $\beta \ge \beta_0$,
\begin{equation} \label{concentration2b}
\limsup_{L\to \infty} \frac{1}{Z_L} \E\left[\exp\Big(\frac\beta {pL^\gamma} \int |\ph|^p\Big) \1_{\{M(\ph) \le \NN L\}} \1_{\{\ph \not \in S_\delta \}}\right] = 0.
\end{equation}
\end{theorem}
As it turns out, the quantitive concentration statements \eqref{concentration} and \eqref{concentration2b}, together with translation invariance of the measure $\rho_L$, are enough to show the qualitative parts of the statements in Theorem \ref{thm:list} (i) and (ii). Indeed, we have the following.
\begin{corollary} \label{triviality}
Let $2 < p < 6$, and let $\alpha, \NN > 0$, and $\gamma \le \frac p2 -1$. Let $M \gg 1$, $\eps >0$. Then
\begin{enumerate}
\item If $\gamma < \frac p2-1$,  
\begin{equation}
\lim_{L \to \infty} \frac 1 {Z_L} \E\left[\exp\Big(\frac\beta {pL^\gamma} \int |\ph|^p\Big) \1_{\{M(\ph) \le \NN L\}} \1_{\{\int_{-M}^{M} |\ph| > \eps \}}\right] = 0.
\end{equation}
\item If $\gamma = \frac p2-1$, 
\begin{equation}
\lim_{\beta \to \infty} \limsup_{L \to \infty} \frac 1 {Z_L} \E\left[\exp\Big(\frac\beta {pL^\gamma} \int |\ph|^p\Big) \1_{\{M(\ph) \le \NN L\}} \1_{\{\int_{-M}^{M} |\ph| > \eps \}}\right] = 0.
\end{equation}
\end{enumerate}
In particular, 
\begin{enumerate}
\item For $\gamma < \frac p2 -1$, the measure on distributions $\frac 1 {Z_L} \exp \Big( \frac\beta {pL^\gamma} \int |\ph|^p \Big) \mu $ converges weakly to $\dl_0$ as $L \to \infty$.
\item For $\gamma = \frac p2 -1$, $\beta \gg 1$, we have that
$$ \lim_{L\to \infty} \frac 1 {Z_L} \exp \big( \frac\beta {pL^\gamma} \int |\ph|^p \big) \mu \neq \mu,$$
if any such limit exists. 
\end{enumerate}
\end{corollary}
The proofs of Theorem \ref{thm:2}, Theorem \ref{thm:2b} and Corollary \ref{triviality} will be presented in Section 4.

\subsection{Weakly nonlinear critical case and subcritical case} 
In this context, in Section \ref{sec:GaussianLimit} we will show the following.  
\begin{theorem}[Weakly nonlinear critical case and subcritical case]\label{thm:3}
Let $\gamma = \frac p 2 -1$ and let $\NN > \frac 1 {2 \sqrt{\alpha}}$. Then, for every $0\le\beta < \beta_0 = \beta_0(p,\alpha,\NN)$, the measure on $C(\R)$
\begin{equation}\label{rhoL}
\rho_L = \frac{1}{Z_L} \exp\Big(\frac\beta {pL^\gamma} \int_{-L/2}^{L/2} |\ph|^p\Big) \1_{\{M(\ph) \le \NN L\}}\mu_L,
\end{equation}
converges weakly to the Ornstein-Uhlenbeck measure with mass parameter $\alpha$ as $L \to \infty$. Here, by Ornstein-Uhlenbeck with mass parameter $\alpha$, we mean the centred, complex-valued Gaussian measure on $C(\R)$ with covariance operator
\begin{gather*}
 \E[\langle{\ph},{f}\rangle\overline{\langle{\ph},{g}\rangle}] = \overline{\langle{(\alpha - \Delta)f},{g}\rangle},\\
 \E[\langle{\ph},{f}\rangle{\langle{\ph},{g}\rangle}] =0
\end{gather*}
for each $f,g \in C^\infty_c(\R)$.
The same result holds for $\gamma > \frac p 2 -1$,  $\NN > \frac 1 {2 \sqrt \alpha}$, and for every $0 \le \beta < +\infty$.
\end{theorem}
The proof of Theorem \ref{thm:3} is rather delicate, and we will discuss here its main steps. First of all, we notice that for $\gamma \le 1$, one cannot hope that the (unnormalised) density $\exp\Big(\frac\beta {pL^\gamma} \int_{-L/2}^{L/2} |\ph|^p\Big) \1_{\{M(\ph) \le \NN L\}}$ converges to the constant function $\1$ in $L^1(\mu_L)$. Indeed, by a simple application of Jensen's inequality, we obtain that 
\begin{align*}
&\E\Big[ \exp\Big(\frac\beta {pL^\gamma} \int_{-L/2}^{L/2} |\ph|^p\Big) \1_{\{M(\ph) \le \NN L\}}\Big]\\
& \ge  \exp\Bigg( \frac{\E\Big[ \1_{\{M(\ph) \le \NN L\}} \frac\beta {pL^\gamma} \int_{-L/2}^{L/2} |\ph|^p \Big]}{\P(\{M(\ph) \le \NN L\}}) \Bigg) \P(\{M(\ph) \le \NN L\}) \\
&\ge \exp(cL^{1-\gamma})(1+o(1)),
\end{align*}
where we simply used translation invariance to estimate the integral inside the expectation, together with the fact that $\P(\{M(\ph) \le \NN L\}) \to 1$ exponentially fast (see \eqref{LDE}). 
Instead, what we do is to fix a parameter $0< \theta < \gamma$, and split the density of $\frac{d \rho_L}{d \mu_L}$:
\begin{align*}
\begin{multlined}
\rho_L =\frac{1}{Z_L} \exp\Big(\frac\beta {pL^\gamma} \int_{-L^\ta/2}^{L^\ta/2} |\ph|^p\Big) \1_{\{M(\ph) \le \NN L\}} \\
\times \exp\Big(\frac\beta {pL^\gamma} \int_{[-L^\ta/2,L^\ta/2]^c} |\ph|^p\Big) \1_{\{M(\ph) \le \NN L\}} \mu_L.
\end{multlined}
\end{align*}
Then the proof Theorem \ref{thm:3} follows from the following ideas. 
\begin{enumerate}
\item The density $\exp\Big(\frac\beta {pL^\gamma} \int_{-L^\ta/2}^{L^\ta/2} |\ph|^p\Big) \1_{\{M(\ph) \le \NN L\}}$ converges to $\1$ on $L^1(\mu_L)$. 
\item \label{RandomLabel}For fixed $K > 0$, the random variables $\ph|_{[-K,K]}$ and $\exp\Big(\frac\beta {pL^\gamma} \int_{[-L^\ta/2,L^\ta/2]^c} |\ph|^p\Big)$ are ``almost independent"  as $L \to \infty$ when $\ph$ is distributed according to $\mu_L$. 
\end{enumerate}
Property $(1)$ strongly depends on the smallness of the parameter $\beta$, or subcriticality. Indeed, by exploiting the variational formula \eqref{BDintro} again, with $F= \1_{\{M(\ph) \le \NN L\}} \frac\beta {pL^\gamma} \int_{-L^\ta/2}^{L^\ta/2} |\ph|^p$, and the fact that $\P(\{M(\ph) \le \NN L\}) \to 1$ exponentially fast, the convergence in $L^1$ essentially coincides with the statement that $V\equiv 0$ is an almost optimiser for \eqref{BDintro}. 
This in turn essentially requires (a variation of) the following inequality 
\begin{align*}
&\1_{ \{ \norm{\ph(1) + V}_{L^2}^2 \le L\NN \}} \frac{\beta}{pL^\gamma} \int |V|^p \\
&\le  \frac 12  \norm{\partial_xV}_{\dot H^1}^2 + \frac \alpha 2 \norm{ V}_{L^2}^2 + o(1),
\end{align*}
which is what requires us to be in the small nonlinearity regime (or in the subcritical case). A rigorous version of (1) above and this heuristic can be found in Lemma \ref{lemma:small_scale1}. 

We now move to discussing how one can obtain a suitable version of (2). This essentially boils down to showing that if $F$  depends only on $\ph|_{[-K,K]}$, then
\begin{equation} \label{ind_hope}
\begin{aligned}
 &\E\Big[\exp\Big( F(\ph)\Big) \times \exp\Big(\frac\beta {pL^\gamma} \int_{[-L^\ta/2,L^\ta/2]^c} |\ph|^p\Big)  \1_{\{M(\ph) \le \NN L\}} \Big] \\
 & \approx \E\Big[\exp\Big( F(\ph)\Big) \Big] \times \E\Big[\exp\Big(\frac\beta {pL^\gamma} \int_{[-L^\ta/2,L^\ta/2]^c} |\ph|^p\Big)  \1_{\{M(\ph) \le \NN L\}} \Big].
\end{aligned}
\end{equation}
If we replace the quantity $M(u)$ with 
$$ M_\ta(u) = \int_{[-L^\ta/2,L^\ta/2]^c} |u|^2, $$
then we can obtain \eqref{ind_hope} by writing the three expectations in \eqref{ind_hope} via the variational formula \eqref{BDintro} once again, and then studying the properties of (almost) optimisers of the three expressions. This is performed in Lemma \ref{lemma: large_scale}, Lemma \ref{lemma: small_scale2} and Lemma \ref{lemma: small_scale2b}, and relies heavily on the fact that solutions of the equation 
$$ (\alpha - \Delta) v = 0 \text{ on }  [-L^\ta/2,L^\ta/2]\setminus [-K,K] $$
depend only on the boundary values $v(-L^\ta/2), v(L^\ta/2), v(-K), v(K)$, and decay exponentially fast in the distance from the boundary. 

The final step is replacing $M_\ta(\ph)$ with $M(\ph)$. This is rather delicate, and it is performed in Lemma \ref{lemma: eps_removal}, Lemma \ref{lemma: cutoff_removal} and Lemma \ref{lemma: theta_removal2}. 
While Lemma \ref{lemma: eps_removal} and Lemma \ref{lemma: theta_removal2} can be proven using the various ideas  presented in this introduction, the proof of Lemma \ref{lemma: cutoff_removal} heavily relies on translation invariance.
\subsection{Structure of the paper}
This paper is structured as follows. 

In Section 2, we introduce various notations, and prove a series of properties related to optimisers for the quantity $A(\beta, \NN)$ defined in \eqref{min:eqn}, and the analogous relevant properties for similar quantities defined over function on the torus $L\T$. We will also show the existence of an approximation of the endpoint of a Brownian motion $\ph(1)$ with another Gaussian process $X_\eta(1)$, where $X_\eta$ is differentiable in the time variable $t$. This approximation will be instrumental in the proofs of Proposition \ref{prop:1b}, and the lower bounds in \eqref{thm2: asymptotic} and \eqref{thm2b: asymptotic}.

In Section 3, we prove Theorem \ref{thm:1} for $\NN \neq \NN_0$, and show that the quantity $Z(\beta,L,\NN)$ satisfies the Bou\'e-Dupuis formula \eqref{BD} whenever it is finite.

In Section 4, we will focus on the supercritical case and the strongly nonlinear critical case, and show Theorem \ref{thm:2}, Theorem \ref{thm:2b} and Corollary \ref{triviality}. In particular, we will first show \eqref{thm2: asymptotic} and \eqref{thm2b: asymptotic} by showing independently lower and upper bounds for $Z_L$, and then exploiting those to show \eqref{concentration} and \eqref{concentration2b}, in order to then proceed to the proof of Corollary \ref{triviality}.

In Section 5, we consider the subcritical case and the weakly nonlinear critical case, and show Theorem \ref{thm:3} following the strategy delineated in subsection 1.3.

In Appendix A we will present the particular version of the Bou\'e-Dupuis formula that we are going to use throughout the paper. In Appendix B, we show a large deviation estimate for $M(\ph)$ when $\ph$ is distributed according to $\mu_L$, and then show convergence of $\mu_L$ to $\mu_{OU}$ when $L \to \infty$ in the correct topology. 

\subsection*{Acknowledgements:}
We thank Fabian H\"ofer for his detailed reading of an earlier version of this article. 
HW was supported by the Royal Society through the University Research Fellowship UF140187, by the Leverhulme Trust through a Philip Leverhulme Prize and by the European Union (ERC, GE4SPDE, 101045082). HW acknowledges funding by the Deutsche Forschungsgemeinschaft under Germany’s Excellence Strategy EXC 2044 390685587, Mathematics M\"unster: Dynamics – Geometry – Structure.

\section{Preliminaries} 

For a function $f: [-L/2,L/2] \to \C$, we define the Fourier series $\hat f: \Z \to \C$ by the formula
\begin{equation}
\hat f(n) := L^{-\frac 12} \int_{-L/2}^{L/2} f(x) e^{-\frac{2\pi i}{L} n x} d x.    \label{ft}
\end{equation}
This lead to Parseval's identity 
\begin{equation}
\int_{-L/2}^{L/2} |f|^2(x) ds = \sum_{n \in Z} |\hat f(n)|^2 \label{parseval}
\end{equation}
and the inversion formula  
\begin{equation}
f(x) = L^{-\frac 12} \sum_{n \in Z} \hat f(n) e^{\frac{2\pi i}{L} n x}. \label{inversion}
\end{equation}
In view of the random Fourier series representation  \eqref{random_fourier0} of samples from $\mu_L$, we define the function valued Brownian motion
\begin{align}\label{def:phi}
\ph(t,x) =  \sum_{k\in \Z} \frac{B_k(t)}{\sqrt{\alpha + 4\pi^2 (\frac{k}{L})^2}} \frac{e^{i\frac{2\pi kx}{L}}}{\sqrt L},
\end{align}
for i.i.d. standard complex valued Brownian motions $B_k$. Notice that this corresponds exactly to a cylindrical Brownian motion with Hilbert space $H$ given by  $\|u\|_{H}^2 = \alpha \int |u|^2 + \int |\partial_x u|^2$.

We denote by $\D'(K)$ the dual of $C^\infty_c((-K,K))$. For a distribution $u \in \D'(\R)$, we can define its restriction $u|_{(-K,K)}$ via the duality 
$$ \jb{u|_{(-K,K)}, \psi} = \langle u,\psi \rangle  \text{ for every }\psi \in C^\infty_c((-K,K)). $$
Given a functional $F: \D'(K) \to \C$ and $u \in \D'(\R)$, we will abuse of notation and denote 
$$ F(u) := F(u|_{(-K,K)}). $$ 

We write  $A \lesssim B$ to denote that there exists a constant $C>0$ such that $A \leq C B$ and $A \sim B$ for $A \lesssim B$ and $B \lesssim A$. 

\subsection{Approximate process}

The following construction will be used in the construction of controls $V$ in the proof of lower bounds in the proofs of Theorems~\ref{thm:1} \ref{thm:2}, \ref{thm:2b}. In order to justify that $\ph(1)$ 
can be treated as a perturbation in the variational principle~\ref{BDintro} we will construct $V(t) = \rho - X_\eta$ (see e.g. proof of Proposition~\ref{prop:1b} below) where $X_\eta$, constructed in the following Lemma, is such that 
$X_\eta(1) - \ph(1)$ is small while $ \frac 12 \int_0^1 \| \partial_x \dot V(s)\|_{L^2}^2 + \alpha \| \dot V(s) \|_{L^2}^2 ds$ stays under control, and the process $X_\eta(t)$ is adapted to the filtration 
generated by $\ph$.

\begin{lemma} \label{OUapprox}
Let $\ph = \ph(t,x)$ be defined by \eqref{def:phi}.
For $\eta>0$ let $X_\eta$ be the stochastic process defined by 
\begin{equation}
\begin{cases}
\partial_t \hat X_{\eta} (t,n) = \frac {L} {\eta \jb{n}} (\hat \ph(t,n) - \hat{X}_\eta (t,n))\\
\hat X_\eta (0,n)  = 0
\end{cases}
\end{equation}
for $\jb{n} \le L \eta^{-1}$, and $X_{\eta} = 0$ for $\jb{n} \ge L \eta^{-1}$. Here,  we use the notation $\jb{n} = \sqrt{\alpha +n^2}$.
Then $X_\eta$ is a Gaussian process, such that for any $x \in [-L/2, L/2]$
%for $	\alpha = 0$,
%\begin{align}
%\E |X_\eta(1)-\ph(1)|^2 &\les \eta (|\log \eta|+ \log L), \\
%\E \int_0^1 \norm{\dot X_\eta (t)}_{\dot H^1}^2 dt &\lesssim L \frac{|\log \eta| + \log L}\eta.
%\end{align}
%and for $\alpha > 0$,
%\hw{define this norm. Insist that $\alpha>0$}
\begin{align}
\E |X_\eta(1,x)-\ph(1,x)|^2 &\les \eta| \log \eta|, \\
\E \int_0^1 \norm{\partial_x \dot X_\eta (t)}_{L^2}^2 dt &\lesssim \frac{L |\log \eta|}{\eta}.
\end{align}
\end{lemma}
\begin{proof}
Fix $n$ with $\jb{n} \le L \eta^{-1}$, and let $Y_n:= \hat \ph(n) - \hat X_\eta(n)$. Then  $Y_n$ satisfies the stochastic differential equation
\begin{equation}
\begin{cases}
 d Y_n = - \frac L {\jb{n}\eta} Y_n  dt+ d  \hat \ph(n)\\
 \hat Y_\eta (0)  = 0,
\end{cases}
\end{equation}
and therefore 
$$Y_n(t) = \int_0^t \exp\left(-\frac{L}{\jb{n}\eta}(t-s)\right)  d_s \hat \ph(s,n) ,$$
and so, by Ito's isometry, and the definition \eqref{def:phi}  of the Brownian motion $\ph$ 
\begin{align*}
\E[|Y_n(t)|^2] &=  \E[|\hat \ph(1,n) |^2] \int_0^t \exp\left(-\frac {2Ls}{\jb{n}\eta} \right) d s\\
&=\E[|\hat \ph(1,n) |^2] \left [1 - \exp\left(-\frac {2Lt}{\jb{n}\eta} \right)\right] \frac{\jb{n}\eta}{2L}. 
\end{align*}
Moreover,  
$$\E \int_0^t |\partial_t \hat X_{\eta} (n)|^2 = \frac {L^2} {\jb{n}^2\eta^2} \int_0^1 \E[|Y_n(t)|^2] \sim\frac{L}{\jb{n}\eta }\E[|\hat \ph(1,n) |^2]. $$
Therefore, by \eqref{parseval} and by stationarity in the spatial variable we have  for any $x $
\begin{align*}
\E |X_\eta(1,x)-\ph(1,x)|^2 &= L^{-1} \E \|X_\eta(1)-\ph(1)\|_{L^2}^2  \\
%&= L^{-1} \Big(\sum_{\jb{n} \le L\eta^{-1}} \E |Y_n(1)|^2 + \sum_{|n| > L\eta^{-1}} \E |\hat \ph (1,n)|^2\Big) \\
%& \les L^{-1} \Big(\sum_{\jb{n} \le L\eta^{-1}} \E[|\hat \ph(1,n) |^2]\frac{\jb{n}\eta}{L} + \sum_{\jb{n} > L\eta^{-1}} \E |\hat \ph (1,n)|^2\Big)\\
%& \les \sum_{\jb{n} \le L\eta^{-1}} \eta\frac{1}{\jb{n}} +  \sum_{\jb{n} > L\eta^{-1}} \frac L {|n|^2}\\
%& \les \eta (|\log \eta| + \log L) + \eta \les \eta (|\log \eta| + \log L),
%\end{align*}
%and therefore
%\begin{align*}
%\E |X_\eta(1)-\ph(1)|^2 &=
&= L^{-1} \Big(\sum_{|n| \lesssim L\eta^{-1}} \E |Y_n(1)|^2 + \sum_{|n| \gg L\eta^{-1}} \E |\hat \ph (n)(1)|^2\Big) \\
& \les L^{-1} \Big(\sum_{|n| \lesssim L\eta^{-1}} \jb{n}^{-1} \eta + \sum_{|n| > L\eta^{-1}} \jb{n}^{-2}\Big)\\
& \les \eta |\log \eta| + \eta \les \eta |\log \eta|.
\end{align*}
Moreover, 
%by \eqref{parseval}, for $\alpha = 0$,
%\begin{align*}
%\E \int_0^1 \norm{\dot X_\eta (t)}_{\dot H^1}^2 dt &\sim \sum_{|n| \le L\eta^{-1}} \frac{n^2}{L^2} \E \int_0^t |\partial_t \hat X_{\eta} (n)|^2\\
%& \sim \sum_{|n| \le L\eta^{-1}} \frac{n^2}{L}\cdot \frac{\E[|\hat \ph(n) (1)|^2]}{|n|\eta} \\
%& \lesssim \sum_{|n| \le L\eta^{-1}} \frac{L}{|n|\eta} \lesssim L \frac{|\log \eta| + \log L}{\eta},
%\end{align*}
%and for $\alpha > 0$,
\begin{align*}
\E \int_0^1 \norm{\dot X_\eta (t)}_{\dot H^1}^2 dt &\sim \sum_{\jb{n} \le \eta^{-1}} \frac{n^2}{L^2} \E \int_0^t |\partial_t \hat X_{\eta} (n)|^2\\
& \sim \sum_{|n| \les L\eta^{-1}} \frac{n^2}{L}\frac{1}{\jb{n}^3 \eta} \\
& \lesssim \frac{L |\log \eta|}{\eta}.
\end{align*}
\end{proof}

\subsection{Analysis of the deterministic variational problem}
We establish some facts about the deterministic variational problem that are important throughout analysis. We start by recalling 
 the Gagliardo-Nirenberg-Sobolev inequality on the real line
\begin{equation}
\int_{\mathbb{R}} |\ph|^p \le \CGN^p
%C_{GNS,p}^p 
\left(\int_{\mathbb{R}} |\partial_x \ph|^2\right)^{\frac p4 - \frac12} \NN^{\frac p4 + \frac12}, \label{gns}
\end{equation}
where $\NN = \int_{\mathbb{R}} |\phi|^2 $. The following Lemma gives a version of this inequality on the torus.

\begin{lemma}[Gagliardo-Nirnenberg-Sobolev on the torus]
Let $ f \in H^1(-L/2,L/2)$ with periodic boundary conditions. Let 
\begin{gather*}
P_0 f = \frac 1 L \int_{-L/2}^{L/2} f(x) dx, \\
P_{\neq 0} f = f - P_0 f.
\end{gather*}
Then for $2 < p \le 6$, we have
\begin{equation}\label{GNStorus}
\int |P_{\neq 0} f|^p(x) dx \le \CGN^p \Big(\int |P_{\neq 0} f|^2(x)\Big)^\frac {p+2}{4} \Big(\int |\partial_x f|^2(x)\Big)^{\frac{p-2}4}, 
\end{equation}
where $\CGN = \CGN(p)$ is the optimal constant for the analogous Gagliardo-Nirenberg-Sobolev inequality on $\R$. Moreover, for every $\dl > 0$, there exists $C_\dl > 0$ such that
\begin{equation}\label{GNStorus2}
\|f\|_{L^6}^6 \le C_\dl L^{-\frac p2 + 1} \|f\|_{L^2}^p + \CGN^p(1+\dl) \|f\|_{L^2}^\frac{p+2}2  \Big(\int |\partial_x f|^2(x)\Big)^{\frac{p-2}{4}}.
\end{equation}
\end{lemma}
\begin{proof}
The estimate \eqref{GNStorus2} follows directly from decomposing $f = P_0 f + P_{\neq 0} f$, followed by Young's inequality, \eqref{GNStorus} and the observation
$$ \int |P_0 f|^p = L |P_0 f|^p = L^{-\frac p2 + 1} \big(L |P_0 f|^2\big)^\frac p2 = L^{-\frac p2 + 1} \| P_0 f\|_{L^2}^\frac p2 \le L^{-\frac p2 + 1} \|f\|_{L^2}^\frac p2. $$
The proof of the estimate \eqref{GNStorus} can be found in \cite[Lemma 3.3]{OST}. 
\end{proof}

For fixed $p \ge 2$, recall the quantity 
\begin{equation} %\label{min:eqn}
A(\beta,\NN) = \inf_{\norm{\ph}_{L^2(\R)}^2 \le \NN} - \frac \beta p \int |\ph|^p + \frac12 \int | \partial_x \ph|^2.
\end{equation}
defined above in \eqref{min:eqn} and define
\begin{equation} \label{min:eqntorus}
B(\beta,\NN) := \inf_{\ph \in H^1([-L/2,L/2]), \int \ph = 0,\atop  \norm{\ph}_{L^2([-L/2,L/2])}^2 \le \NN} - \frac \beta p \int_{-L/2}^{L/2} |\ph|^p + \frac12 \int_{-L/2}^{L/2} |\partial_x \ph|^2,
\end{equation}
where here $H^1([-L/2,L/2])$ is understood with periodic boundary conditions.
By performing the scaling $\ph = \mu \lambda^\frac12 \ph_{\lambda,\mu}(\lambda x)$, we obtain that 
\begin{equation} \label{Ascaling}
A(\beta,\NN) =  \mu^2 \lambda^{2} A(\lambda^{-\frac{6-p}2} \mu^{p-2} \beta, \mu^2 \NN).
\end{equation}
\begin{lemma} \label{gns2}
$A(\beta, \NN) > -\infty$  if and only if
\begin{enumerate}
\item $p < 6$, and in this case $A(\beta, \NN) < 0$, or 
\item $p = 6$ and $\NN \le \NN_0(\beta)$, and $A(\beta, \NN) = 0$.
\end{enumerate}
Moreover, 
\begin{enumerate}
\setcounter{enumi}{2}
\item If $A(\beta, \NN) > -\infty$, then $B(\beta, \NN) \ge A(\beta, \NN)$,
\item If $A(\beta, \NN) = -\infty$, then $B(\beta, \NN) = A(\beta, \NN) = - \infty$.
\end{enumerate}

\end{lemma}
\begin{proof}
%Recall the Gagliardo-Nirenberg-Sobolev inequality  \hw{good place?}
%\begin{equation}
%\int |\ph|^p \le \CGN^p \left(\int|\partial_x \ph|^2\right)^{\frac p4 - \frac12} \NN^{\frac p4 + \frac12}. \label{gnsa}
%\end{equation}
In case (1) the Gagliardo-Nirenberg-Sobolev inequality \eqref{gns} yields
$$ - \frac \beta p \int |\ph|^p + \frac12 \int |\partial_x \ph|^2 \ge -  \frac \beta p\CGN^p \left(\int|\partial_x \ph|^2\right)^{\frac p4 - \frac12} \NN^{\frac p4 + \frac12} + \frac12 \int |\partial_x \ph|^2. $$
Since $\frac p4 - \frac 12 < 1$, this expression viewed as a function of $\int|\partial_x \ph|^2$ is bounded from below, so  $A(\beta, \NN)$ is finite. Fix a nonzero function $\rho \in H^1(\R)$ with $\| \rho \|_{L^2} \le \NN$. Define 
\begin{equation} \label{rhold}
 \rho_\lambda(x) := \lambda^{-\frac12} \rho(\lambda^{-1} x). 
\end{equation}
We have that 
\begin{align*}
A(\beta, \NN) &\le - \frac \beta p \int |\rho_\lambda|^p + \frac12 \int |\rho_\lambda'|^2 \\
&= - \frac {\lambda^{-(\frac p2 - 1)}  \beta} p  \int |\rho|^p + \frac {\lambda^{-2}}2 \int |\rho'|^2.
\end{align*}
Recalling that $\frac p2 - 1 < 2$, by choosing $\lambda \gg 1$, we can make the RHS of this expression strictly negative. Therefore, $A(\beta, \NN) < 0$.

In case (2), by definition of $\NN_0$ (see \eqref{N0}), we have that 
$$ \frac{\beta}6 \CGN^6 \NN_0^{2} = \frac 12. $$
Therefore, by \eqref{gns}, 
\begin{align*}
- \frac \beta 6 \int |\ph|^6 + \frac12 \int |\partial_x \ph|^2 &\ge \Big(\frac 12 - \frac{\beta}6 \CGN^6 \NN^{2}\Big)   \int |\partial_x \ph|^2\\
&= \frac{\beta}6 \CGN^6 (\NN_0^2 - \NN^{2}) \int |\partial_x \ph|^2 \ge 0,
\end{align*}
so $A(\beta, \NN) \ge 0$. Note that we can achieve the value $0$ simply by choosing $ \ph = 0$.

If $p= 6$, but $\NN > \NN_0$, by the optimality of $\CGN$, there exists a function $\rho \in H^1$ such that $\| \rho \|_{L^2}^2 < \NN$, and 
$$ \frac \beta 6 \int \rho^6 > \frac 12 \int |\rho'|^2. $$
Since the set 
$$\Big\{ f \in C^1: \supp(f) \text{ is compact}, \int f = 0\Big\} $$ 
is dense in $ H^1(\R)$, we can assume that $\rho$ has compact support and $\int \rho = 0$. Therefore, if $\rho_\lambda$ is as in \eqref{rhold}, for $\lambda \ll 1$ we have that 
\begin{align*}
B(\beta, \NN) &\le - \frac \beta 6 \int |\rho_\lambda|^6 + \frac12 \int | \partial_x \rho_\lambda|^2 \\
&= \lambda^{-2} \Big(- \frac { \beta} 6  \int |\rho|^6 + \frac {1}2 \int |\partial_x \rho|^2\Big) \to -\infty
\end{align*}
as $\lambda \to 0$, and the same holds for $A(\beta, \NN)$. If $p > 6$, we can take any nonzero $\rho$ with compact support and $\int \rho =0$, and for $\lambda \ll 1$, 
\begin{align*}
B(\beta, \NN) &\le - \frac \beta p \int |\rho_\lambda|^p + \frac12 \int |\partial_x \rho_\lambda|^2 \\
&= - \frac {\lambda^{-(\frac p2 - 1)}  \beta} p  \int |\rho|^p + \frac {\lambda^{-2}}2 \int | \partial_x \rho|^2 \to -\infty 
\end{align*}
as $\lambda \to 0$, and the same holds for $A(\beta, \NN)$. 

We are just left with proving (3). We start with the real-valued case. Given a function $f \in H^1([-L/2,L/2])$ with $\int f = 0$, 
by the intermediate value theorem, there must exist a point $x_0 \in [-L/2,L/2]$ with $f(x_0) = 0$. Define $F_f: \R \to \R$  as 
\begin{equation} \label{F_f}
F_f(x) =
\begin{cases}
0 & \text{ if } x \le x_0 \\
f(x) &  \text{ if }  x_0 \le x \le L/2 \\
f(x - L) &  \text{ if }  L/2 \le x \le L+ x_0 \\
0 & \text{ if } x \ge L + x_0.
\end{cases}
\end{equation}
We have that $F_f \in H^1(\R)$, $\int |F_f|^p = \int |f|^p$, and $\int |\partial_x F_f|^2 = \int |\partial_x f|^2$. Therefore, when $\ph$ is real-valued, we have that 
\begin{align*}
&- \frac \beta p \int_{-L/2}^{L/2} |\ph|^p + \frac12 \int_{-L/2}^{L/2} |\partial_x \ph|^2 \\
&=  - \frac \beta p \int |F_\ph|^p + \frac12 \int | \partial_x F_\ph|^2\\
& \ge \inf_{F \in H^1(\R), \norm{F}_{L^2(\R)}^2 \le \NN} - \frac \beta p \int |F|^p + \frac12 \int |\partial_x F|^2\\
& = A(\beta, \NN).
\end{align*}
We now move to the complex-valued case. Let $\ph_1 = \Re \ph$, $\ph_2 = \Im \ph$. By Jensen's inequality, recalling that $p > 2$, we have that 
\begin{align*}
\int_{-L/2}^{L/2} |\ph|^p &= \int_{-L/2}^{L/2} \big(|\ph_1|^2 + |\ph_2|^2\big)^\frac p2 \\
& = \int_{-L/2}^{L/2} \Big(\frac{M(\ph_1)}{\NN} \big(\frac{\NN}{M(\ph_1)}|\ph_1|^2\big) + \frac{M(\ph_2)}{\NN} \big(\frac{\NN}{M(\ph_2)}|\ph_2|^2\big)\Big)^\frac p2 \\
& \le \big(\frac{M(\ph_1)}{\NN}\big) \int_{-L/2}^{L/2} \left| \big(\frac{\NN}{M(\ph_1)}\big)^\frac 12\ph_1\right|^p + \big(\frac{M(\ph_2)}{\NN}\big) \int_{-L/2}^{L/2} \left| \big(\frac{\NN}{M(\ph_2)}\big)^\frac 12\ph_2\right|^p. 
\end{align*}
Therefore, by the real case and the fact that $M(\big(\frac{\NN}{M(\ph_j)}\big)^\frac 12\ph_j) = \NN$, we obtain that
\begin{align*}
&- \frac \beta p \int_{-L/2}^{L/2} |\ph|^p + \frac12 \int_{-L/2}^{L/2} |\partial_x \ph|^2 \\
&\ge \big(\frac{M(\ph_1)}{\NN}\big) A(\beta, \NN) - \frac 1 2 \cdot \frac{M(\ph_1)}{\NN} \int_{-L/2}^{L/2} \left|\big(\frac{\NN}{M(\ph_1)}\big)^\frac 12 \partial_x \ph_1 \right|^2 \\
&\phantom{\ge\ } + \big(\frac{M(\ph_2)}{\NN}\big) A(\beta, \NN) - \frac 1 2 \cdot \frac{M(\ph_2)}{\NN} \int_{-L/2}^{L/2} \left|\big(\frac{\NN}{M(\ph_2)}\big)^\frac 12 \partial_x \ph_2\right|^2 \\
&\phantom{\ge\ } + \frac12 \int_{-L/2}^{L/2} |\partial_x \ph|^2 \\
& = A(\beta, \NN).
\end{align*}
%Since $C^2$ functions are dense in $H^1$, we can just consider the case $\ph \in C^2$. By expressing $\ph$ in polar coordinate (on the set $\{\ph \neq 0\}$), it is straightforward to check that pointwise
%$$\big| |\ph|'\big| \le |\partial_x \ph|. $$
%Therefore, by the real case discussed above,
%\begin{align*}
%B(\beta,\NN) &= \inf_{\ph \in H^1([-L/2,L/2]), \int \ph = 0,\atop  \norm{\ph}_{L^2([-L/2,L/2])}^2 \le \NN} - \frac \beta p \int_{-L/2}^{L/2} |\ph|^p + \frac12 \int_{-L/2}^{L/2} |\partial_x \ph|^2 \\
%&\ge \inf_{\ph \in H^1([-L/2,L/2]), \int \ph = 0,\atop  \norm{\ph}_{L^2([-L/2,L/2])}^2 \le \NN} - \frac \beta p \int_{-L/2}^{L/2} |\ph|^p + \frac12 \int_{-L/2}^{L/2} \big||\ph|'\big|^2 \\
%& = A(\beta, \NN).
%\end{align*}

\end{proof}

\begin{lemma} \label{minimisers}
Let $p < 6$. 
Then for every $\beta, \NN > 0$, there exists $Q_p \in H^1(\R)$, $\norm{Q_p}_{L^2}^2 = \NN$, such that 
$$ A (\beta,\NN)=  - \frac \beta p \int |Q_p|^p + \frac12 \int | \partial_x Q_p|^2.$$
Moreover,
\begin{enumerate}
\item If $R \in H^1(\R)$ satisfies  $\norm{R}_{L^2}^2 \le \NN$, $ A(\beta,\NN) =  - \frac \beta p \int |R|^p + \frac12 \int |\partial_x R|^2$, then there exists 
$x_R \in \R$ and $\ta \in [0,2\pi]$ such that $R(x) = e^{i\ta} Q_p(x-x_R)$. \label{minimisers1}
\item For every $\delta >0$, there exists $\eps_0 = \eps_0(\dl) > 0$ such that if $\ph \in H^1(\R)$ satisfies $\inf_{x_0 \in \R,\ta\in[0,2\pi]} \norm{\ph -e^{i\ta} Q_p(\cdot - x_0)}_{L^2 \cap L^\infty} > \delta$, $\norm{\ph}_{L^2}^2 \le \NN$, then \label{minimisers2}
\begin{equation}
\label{stability}
 - \frac \beta p \int |\ph|^p + \frac12 \int |\partial_x \ph|^2 \ge A(\beta,\NN) + \eps_0.
\end{equation}
Similarly, we have that  
\begin{equation}\label{stabilityT}
\inf_{\ph \in H^1([-L/2,L/2]),\  \norm{\ph}_{L^2}^2 \le \NN, \int W =0, \atop
\sup_{x_0 \in \R,\ta\in[0,2\pi]} \norm{\ph -e^{i\ta} Q_p(\cdot - x_0)}_{L^2 \cap L^\infty(\R)} > 2\dl} -\frac{\beta}{p} \int |\ph|^p + \frac{1}2 \int |\partial_x \ph|^2 \ge A(\beta,\NN) + \eps_0(\dl),
\end{equation}
where $H^1([-L/2,L/2])$ is understood with periodic boundary conditions.
\item $A$ is a continuous function of $\beta$ and $\NN$. \label{minimisers3}
\end{enumerate}
\end{lemma}
\begin{proof}
Let $\ph_n$ be a minimising sequence, i.\ e. $\norm{\ph_n}_{L^2}^2 \le \NN$, and 
$$ A(\beta, \NN) = \lim_n  - \frac \beta p \int |\ph_n|^p + \frac12 \int |\partial_x \ph_n|^2.$$
By Lemma \ref{gns2}, (1), we have that $A(\beta, \NN) < 0$. 
Moreover, from Gagliardo-Nirenberg's inequality 
$$\int |\ph|^p \le \CGN^p \left(\int|\partial_x \ph|^2\right)^{\frac p4 - \frac12} \NN^{\frac p4 + \frac12}$$
and the fact that $\frac p4 - \frac 12 < 1$,
%(and $\NN \le \NN_0$ when $p=6$)
 we deduce that $\int |\partial_x \ph_n|^2$ is uniformly bounded. Hence, we can invoke
 the profile decomposition \cite[Proposition 3.1]{HK}
for the (subcritical) Sobolev embedding: $H^1(\R)\hookrightarrow L^2 \cap L^\infty(\R)$. There exist $J^* \in \Z_{\geq 0} \cup\{\infty\}$, 
a sequence $\{\psi^j \}_{j \in \N}$ of non-trivial $H^1(\R)$-functions, 
and a sequence $\{x^j_n\}_{j  \in \N}$ for each $n \in \N$
such that 
up to a subsequence, we have
 \begin{align}
 \ph_n(x)=\sum_{j=1}^J \psi_j(x-x_n^j)+r^J_n(x).
 \label{PD0}
 \end{align}
 
\noindent
for every finite $0 \leq J \leq J^*$, 
where the remainder term $r^J_n$
satisfies
\begin{align}
 \lim_{J \to \infty} \limsup_{n \to \infty} \| r^J_n\|_{L^2 \cap L^\infty(\R)} = 0.
 \label{PD1}
\end{align}

\noindent
Here,  $\lim_{J \to \infty} f(J) := f(J^*)$ if $J^* < \infty$.
Moreover, we have
\begin{align}
\|\ph_n\|_{L^2(\R)}^2&  =  \sum_{j = 1}^J \|\psi_j \|_{L^2(\R)}^2 + \|r_n^J\|_{L^2(\R)}^2
+ o(1),\label{PD2}\\
\|\ph_n\|_{\dot H^1(\R)}^2 & =  \sum_{j = 1}^J \|\psi_j \|_{\dot H^1(\R)}^2 + \|r_n^J\|_{\dot H^1(\R)}^2
+ o(1),
\label{PD3}
\end{align}

\noindent
as $n \to \infty$, and 
\begin{align}
\limsup_{n \to \infty}
\|\ph_n\|_{L^p(\R)}^p  = &  \sum_{j = 1}^{J^*} \|\psi_j \|_{L^p(\R)}^p .
\label{PD4}
\end{align}

Let $\tilde \psi_j := \frac{\NN^\frac12}{\norm{\psi_j}_{L^2}} \psi_j$, so that $\norm{\tilde \psi_j}_{L^2}^2 = \NN$. From the inequality 
$$\frac12 \int |\partial_x \tilde \psi_j|^2 \ge A + \frac \beta p \int |\tilde \psi_j|^p,$$ 
\eqref{PD3}, \eqref{PD1}, \eqref{PD2}, we obtain, up to subsequences, 
\begin{align*}
A &= \lim_n  - \frac \beta p \int |\ph_n|^p + \frac12 \int |\partial_x \ph_n|^2 \\
&\ge \sum_{j=1}^{J^*} - \frac \beta p \int |\psi_j|^p + \frac12 \int | \partial_x \psi_j|^2\\
&= \sum_{j=1}^{J^*} - \frac \beta p \frac{\norm{\psi_j}_{L^2}^p}{\NN^\frac p2} \int |\tilde\psi_j|^p + \frac12 \frac{\norm{\psi_j}_{L^2}^2}{\NN}  \int | \partial_x \tilde \psi_j|^2\\
&\ge \sum_{j=1}^{J^*} - \frac \beta p \frac{\norm{\psi_j}_{L^2}^p}{\NN^\frac p2} \int |\tilde\psi_j|^p + \frac \beta p \frac{\norm{\psi_j}_{L^2}^2}{\NN} \int |\tilde\psi_j|^p + \frac{\norm{\psi_j}_{L^2}^2}{\NN}A\\
&= A \sum_{j=1}^{J^*} \frac{\norm{\psi_j}_{L^2}^2}{\NN} + \sum_{j=1}^{J^*} \frac \beta p \frac{\norm{\psi_j}_{L^2}^2}{\NN} \int |\tilde\psi_j|^p \left(1 - \frac{\norm{\psi_j}_{L^2}^{p-2}}{\NN^{\frac p2 - 1}}  \right) \\
& \ge A + \sum_{j=1}^{J^*} \frac \beta p \frac{\norm{\psi_j}_{L^2}^2}{\NN} \int |\tilde\psi_j|^p \left(1 - \frac{\norm{\psi_j}_{L^2}^{p-2}}{\NN^{\frac p2 - 1}}  \right) \\
& \ge A.
\end{align*}
In order for the last inequality to be an equality, we need either $J^* =0$, or $J^*= 1$ and $\norm{\psi_1}_{L^2}^2 = \NN$. If $J^* = 0$, by going through the same steps, we obtain $A \ge 0$, which is a contradiction with Lemma \ref{gns2}, (1). Therefore, $J^* = 1$, and so 
\begin{equation}\label{convergenceMinimising}
\lim_{n\to\infty} \norm{\varphi_n(x+x_n^1) - \psi_1}_{L^2\cap L^\infty} = 0,
\end{equation}
and $\psi_1$ is 
a minimiser for \eqref{min:eqn}. In particular, there exists a minimiser $Q_p \in H^1(\R)$.

Let $R_1,R_2$ be minimisers for \eqref{min:eqn}. Since the sequence $\ph_n := R_j$ is a minimising sequence, by the previous argument we obtain that $\norm{R_j}_{L^2}^2 = \NN$. Moreover, by a simple variational argument, we obtain that for some $\lambda_j \in \R$, $R_j$ satisfies 
\begin{equation}\label{Eulero-Lagrange}
-\beta |R_j|^{p-2}R_j - \partial_x^2 R_j = \lambda_j R_j
\end{equation}
in distribution. Moreover, since $R_j \in H^1$, we obtain that $\partial_x^2 R_j$ is a continuous function, so $R_j$ is a classical solution of \eqref{Eulero-Lagrange}. Multiplying \eqref{Eulero-Lagrange} by $\conj R_j$ and integrating, recalling that $\int |R_j|^2 = \NN$, we obtain the identity 
\begin{equation}\label{M1}
-\beta \int |R_j|^p + \int |\partial_x R_j|^2 = \lambda_j \NN.
\end{equation}
Multiplying \eqref{Eulero-Lagrange} by $\conj{\partial_x R_j}$, we obtain that 
$$ -\frac \beta p |R_j|^p - \frac 12 |\partial_x R_j|^2 - \frac {\lambda_j} 2 |R_j|^2  $$
is a constant function. Since $R_j \in H^1(\R)$, by evaluating the expression on a sequence $x_n \to \infty$, we obtain that 
\begin{equation}\label{M2}
-\frac \beta p |R_j|^p - \frac 12 |\partial_x R_j|^2 - \frac {\lambda_j}2 |R_j|^2 = 0,
\end{equation}
which integrated gives the identity 
\begin{equation}\label{M3}
-\frac \beta p \int |R_j|^p - \frac 12 \int |\partial_x R_j|^2 = \frac {\lambda_j}2 \NN.
\end{equation}
Recalling that $-\frac \beta p \int |R_j|^p + \frac 12 \int |\partial_x R_j|^2 = A(\beta, \NN)$, from \eqref{M1} and \eqref{M3} we obtain that $\lambda_1 = \lambda_2 = \lambda(\beta, \NN)$.

Since $R_j \in H^1 \cap C^2_{\loc}$, there exists a point $x_j$ such that $\partial_x R_j(x_j) = 0$. From \eqref{M2}, we obtain that for some $\ta_1,\ta_2 \in [0,2\pi]$,
$$R_1(x_1) = e^{i\ta_1} R_2(x_2) = e^{i\ta_2} \left(-\frac{p\lambda}{2\beta}\right)^{\frac 1 {p-2}}.$$
Therefore, $R_3 = R_2(\cdot - (x_1 - x_2))$ satisfies $\partial_x R_3(x_1) = 0 = \partial_x R_1(x_1)$ and $R_3(x_1) = e^{-i\ta_1} R_1(x_1)$, and both $R_1$ and $R_3$ solve \eqref{Eulero-Lagrange} with the same $\lambda$. Therefore, $R_1 = e^{i\ta_1} R_3$, or more precisely, 
$$R_1(x) = e^{i\ta_1} R_2(x-(x_1-x_2)), $$
which proves \eqref{minimisers1}.

In order to prove \eqref{minimisers2}, we first prove \eqref{stability}, proceeding by contradiction, i.e.\ we suppose that there exists $\delta >0$ such that for every $n\in \N$, there exists $\ph_n$ with $\norm{\ph_n}_{L^2}^2 \le \NN$ and
\begin{equation}\label{M4}
\sup_{x_0 \in \R, \ta \in [0,2\pi]} \norm{\ph_n - e^{i\ta} Q_p(\cdot - x_0)}_{L^2 \cap L^\infty} > \delta,
\end{equation}
but 
\begin{equation}\label{M5}
 - \frac \beta p \int |\ph_n|^6 + \frac12 \int |\partial_x \ph_n|^2 \le A(\beta,\NN) + \frac 1n.
\end{equation}
This implies that $\ph_n$ is a minimising sequence, so by \eqref{convergenceMinimising}, up to subsequences, there exists a sequence $x_n$ such that 
$$\norm{\ph_n(\cdot - x_n) - \psi_1}_{L^2 \cap L^\infty} \to 0,$$
where $\psi_1$ is a minimiser. Therefore, by \eqref{minimisers1}, there exists $x_0,\ta$ such that $\psi_1 = e^{i\ta}Q_p(\cdot - x_0)$. Therefore, 
$\norm{\ph_n(\cdot - x_n) - e^{i\ta} Q_p(\cdot - x_0)}_{L^2 \cap L^\infty} \to 0$, or 
$$\norm{\ph_n(\cdot - (x_n-x_0)) - e^{i\ta} Q_p}_{L^2 \cap L^\infty} \to 0, $$
which contradicts \eqref{M4}.
The implication from \eqref{stability} to \eqref{stabilityT} can be proven proceeding analogously to the proof of Lemma \ref{gns2}, (3).

In order to prove \eqref{minimisers3}, we first notice that 
\begin{align*}
A(\beta,\NN) = &\inf_{\norm{\ph}_{L^2(\R)}^2 \le \NN} - \frac \beta p \int |\ph|^p + \frac12 \int |\partial_x \ph|^2\\
= &\inf_{\norm{\ph}_{L^2(\R)}^2 \le 1} - \frac {\NN^{\frac p2}\beta} p \int |\ph|^p + \frac\NN2 \int |\partial_x \ph|^2\\
=~ &\NN \inf_{\norm{\ph}_{L^2(\R)}^2 \le 1} - \frac {\NN^{\frac p2 - 1}\beta} p \int |\ph|^p + \frac12 \int |\partial_x \ph|^2\\
= &\NN A(\NN^{\frac p2 - 1}\beta, 1),
\end{align*}
so it is enough to show continuity in $\beta$ for $\NN = 1$. By definition, $A(\beta, 1)$ is nonincreasing and lower semicontinuous in $\beta$. Therefore, it is enough to show that 
\begin{equation} \label{M6}
\liminf_{\gamma \uparrow \beta} A(\gamma, 1) \le A(\beta, 1).
\end{equation}
Let  $Q_\beta$ be an optimiser for \eqref{min:eqn} with $\NN = 1$. Let $
\gamma > 0$. We have that 
\begin{align*}
A(\beta, 1) = &- \frac {\beta} p \int |Q_\beta|^p + \frac12 \int |\partial_x Q_\beta|^2 \\
= & - \frac {\gamma} p \int |Q_\beta|^p + \frac12 \int |\partial_x Q_\beta|^2 - \frac{\beta-\gamma}p \int |Q_\beta|^p\\
\ge & A(\gamma,1) - \frac{\beta-\gamma}p \int |Q_\beta|^p, 
\end{align*}
so
\begin{equation}\label{M7}
A(\gamma, 1) \le A(\beta, 1) + \frac{\beta - \gamma}p \int |Q_\beta|^p.
\end{equation}
Therefore, by taking limits, we obtain \eqref{M6}, so $A$ is continuous. 
\end{proof}

%%%%%%%%%%%%%%%%%%%%%%%%
\section{Construction of the measure}
%%%%%%%%%%%%%%%%%%%%%%%%
\label{ss:MAIN_PROOF}
%%%%%%%%%%%%%%%%%%%%%%%

In order to illustrate our method we first give a proof of Lebowitz-Rose-Speer's  \cite{Lebowitz1988} classical result on the normalisability of the measures $\rho_L$, 
which we stated as Theorem~\ref{thm:1}. The following proposition contains the first part, estimate \eqref{eqn: Zfinite}, of this theorem and provides a (non-optimal) estimate on the partition function $Z(\beta,L,\NN)$.
\begin{proposition}\label{prop:1a}
Let $L\geq 1$ and consider $\mu_L$ defined by the random Fourier series \eqref{random_fourier}.
Suppose that 
\begin{itemize}
\item[(i)] $p<6$ and $\NN >0$, or
\item[(ii)] $p=6$ and $\NN < \NN_0$ (defined in \eqref{N0} above)
\end{itemize}
Then the potential 
$$ F = \1_{ \{ M(\ph) \leq \NN  \}} \frac{\beta}{p} \int_{L \T} |\ph|^p $$
satisfies the hypotheses of Proposition \ref{prop: BDinf} with $H = H^1$. Moreover, in this case, letting 
\[
Z(\beta,L,\NN)  =  \E\Big[ \exp \Big( \frac{\beta}{p} \int_{L \T} |\ph|^p \Big) \1_{ \{ M(\ph) \leq \NN  \}}\Big],
\]
we have that 
\begin{equation} \label{eqn: Zbound}
\log Z \les_{\beta,p} 1 +  (\NN^{\frac p2 +1} + 1) L^{\max(\frac p2 + 1, \frac 52)}  < + \infty. 
\end{equation}

\end{proposition}
\begin{proof}
For $M >0$, we consider $\wt Z, \wt Z_M$ given by 
\begin{gather} \label{eqn: tildeZ}
\wt Z := \E\Big[\exp\Big(\1_{\{\int|\ph|^2 \le \NN\}} \frac\beta p \int |\ph|^p \Big) \Big], \\
\wt Z_M := \E\Big[\exp\Big(\1_{\{\int|\ph|^2 \le \NN\}} \Big(M \wedge \frac\beta p \int |\ph|^p \Big)\Big) \Big]. \label{eqn: tildeZM}
\end{gather}
We have that 
\begin{align*}
\wt Z &= \E\Big[\exp\Big(\frac\beta p \int |\ph|^p \Big) \1_{\{\int|\ph|^2 \le \NN\}} \Big] + \E [\1_{\{\int|\ph|^2 > \NN\}}]\\
& = Z + \P(\{\int |\ph|^2) > \NN\}),
\end{align*}
so 
\begin{equation*}
Z \le \wt Z \le Z+1.
\end{equation*}
Therefore, it is enough to prove the analogous of Theorem \ref{thm:1} for $\wt Z$ instead of $Z$. Moreover, by monotone convergence, we have that 
$$ \wt Z = \sup_{M >0} \wt Z_M.$$ 
\textbf{Case $p < 6$ and $p=6, \NN < \NN_0$.} 
By Proposition \ref{prop: BDinf}, it is enough to show that
\begin{equation} \label{eqn: tildeZvar}
\begin{aligned}
&\begin{multlined}
\sup_{V\in \mathbb H^1} \E\Big[ \1_{\{\int |\ph(1) + V(1)|^2 \le \NN\}} \frac \beta p  \int |\ph(1) + V(1)|^6 \\
\phantom{\sup_{V\in \mathbb H^1} \E\Big[ ]}
- \frac 12 \int_0^1 \int \Big( |\partial_x \dot V(t)|^2 + \alpha |\dot V(t)|^2 \Big) dt\Big]
\end{multlined} \\
\les&\ 1 + (\NN^{\frac p2 + 1} + 1) L^{\max(\frac p2 + 1,\frac52)}. 
\end{aligned}
\end{equation}
In the following, for a given $\eps > 0$, let $\eta^\frac s2 \sim \min(\frac{\eps \NN^\frac12}{1+\|\ph\|_{H^s}},1) $, so that if $\ph_\eps:= e^{\eta\Delta} \ph$, 
$$\| \ph-\ph_\eps\|_{L^2} \le \eps\mathbf N^\frac12.$$
Since $\ph_\eps$ is defined by convolution with the heat kernel, for every choice of $\eps > 0$, $1 \le q \le +\infty$, we have that
$$\norm{\ph -\ph_\eps}_{L^q} \le 2 \norm{\ph}_{L^q}. $$
For a given $V \in \mathbb H^1$, we have that 
\begin{align*}
&~\mathbb{E} \bigg[ \1_{\{\int|\ph(1) + V(1)|^2 \le \NN\}} \frac \beta p \int |\ph(1) + V(1)|^p  \\
 &\phantom{\ \E\bigg[]}-  \frac{1}{2} \int_0^1 \int | \partial_x \dot V(t) |^2 + \alpha |\dot V(t)|^2dt \bigg]\\
\le &~\mathbb{E} \bigg[ \1_{\{\int|\ph_\eps(1) + V(1)|^2 \le (1+\eps)^2 \NN\}} \frac \beta p \int |\ph(1) + V(1)|^p  - \frac{1}{2} \int | \partial_x V(1) |^2\bigg]\\
\le &~\mathbb{E} \Big[ \1_{\{\int|\ph_\eps(1) + V(1)|^2 \le (1+\eps)^2 \NN\}}\Big( \frac \beta p \int |\ph(1) + V(1)|^p  - \frac{1}{2} \int | \partial_x V(1) |^2 \Big)    \Big].
\end{align*}
By Young's inequality, for every $\delta > 0$, there exists $C_\delta$ such that 
\begin{align*}
\frac{\beta}{p} \int |\ph(1) + V(1)|^p \le&~ C_\delta \int |\ph(1)-\ph_\eps(1)|^ p+ C_\delta L^{1 - \frac p2} \bigg(\int |\ph_\eps(1) + V(1)|^2\bigg)^\frac p2 \\
&+  \frac{(1+\delta) \beta} p \int |P_{\neq 0}(V(1)+\ph_\eps(1))|^p
\end{align*}
and similarly,
\begin{align*}
\frac{1}{2} \int | \partial_x V(1) |^2 \ge \frac{1}{2(1 + \dl)} \int | \partial_x (\ph_\eps(1) + V(1)) |^2 - C_\dl \int | \partial_x \ph_\eps(1) |^2.
\end{align*}
Therefore, calling $W = P_{\neq 0}(V(1)+\ph_\eps(1))$, we have that
\begin{align*}
&~\mathbb{E} \bigg[ \1_{\{\int|\ph(1) + V(1)|^2 \le \NN\}} \frac1p \int |\ph(1) + V(1)|^p  \\
 &\phantom{\ \E\bigg[]} - \frac{1}{2} \int_0^1 \int | \partial_x \dot V(t) |^2 + \alpha |\dot V(t)|^2dt \bigg]\\
\le &~\mathbb{E} \Big[ \1_{\{\int|\ph_\eps(1) + V(1)|^2 \le (1+\eps)^2 \NN\}} \\
&~\phantom{\mathbb{E} \Big[]}\times \Big( C_\delta \int |\ph(1)-\ph_\eps(1)|^ p + C_\delta (1+\eps)^p \NN^\frac p2 + C_\dl \int | \partial_x \ph_\eps(1) |^2 \\
&~ \phantom{\mathbb{E} \Big[]\times \Big()} + \frac 1 {1 + \dl}\big(   \frac{(1+\delta)^2 \beta} p \int |W|^p - \frac 12 \int |\partial_x W |^2\big)\Big)\Big]\\
\le &~\mathbb{E} \Big[ - \frac{B((1+\dl)^2 \beta, (1 +\eps)^2 \NN)}{1 + \dl} + C_\delta' \Big( \int |\ph(1)|^p + \NN^\frac p2 + \int | \partial_x \ph_\eps(1) |^2 \Big)\Big].
\end{align*}
For every q, we have that $\E \int |\varphi(1)|^q \lesssim L $, 
 and $\int | \partial_x \ph_\eps(1) |^2 \les \eta^{-\frac 34} \| \ph(1) \|_{H^{\frac 14}}^2$. Therefore, by choosing $s = \frac 14$, we have 
$$\mathbb E[\int | \partial_x \ph_\eps(1) |^2] \les \mathbb E[ \eps^{-6} \NN^{-3} (1 + \| \ph(1)\|_{H^\frac 14}^5)] \les \eps^{-6} \NN^{-3} L^\frac 52.$$
Therefore, 
\begin{align*}
 \log \tilde Z \le - \frac{B((1+\dl)^2 \beta, (1 + \eps)^2 \NN)}{1 + \dl} + C_\dl'' ( L^{\frac p2 + 1} + \NN^\frac p2 + \eps^{-6} \NN^{-3} L^\frac 52 ) <  \infty,
\end{align*}
as long as $\eps, \dl$ are small enough (in the case $p = 6$).
\end{proof}

The next proposition contains the second half of Theorem~\ref{thm:1}.
\begin{proposition}\label{prop:1b}
Suppose that 
\begin{itemize}
\item[(i)] $p>6$ and $\NN >0$, or
\item[(ii)] $p=6$ and $\NN > \NN_0$.
\end{itemize}
Let 
\[
Z =  \E\Big[ \exp \Big( \frac{\beta}{p} \int |\ph|^p \Big) \1_{ \{ M(\ph) \leq \NN  \}}\Big].
\]
Then 
\[
Z =\infty. 
\]
\end{proposition}
\begin{proof}
We will show that 
$ \wt Z = \infty.$
For $\eta > 0$, let $X_\eta$ be the process given by Lemma \ref{OUapprox}. Given a function $\rho \in H^1([-L/2, L/2])$ with $\int \rho = 0$ and $\| \rho \|_{L^2}^2 \le (1-\dl)^2\NN$, let 
$V_{\eta, \rho}(t) := - X_\eta + t \rho$. Recall the defintion of $\tilde{Z}_M$ in \eqref{eqn: tildeZM}.
For readability, in the following display we write $E$ for the event $\{\int|\ph(1) - X_\eta(1)|^2 \le \dl^2 \NN\}$ and $\| u \|^2_{H^1_\alpha} = \int |\partial_x u|^2 + \alpha \int |u|^2$. 
By the formula \eqref{P3},  for $M \ge M_0(\eta, \rho)$ big enough, we have
\begin{align*}
\log \wt Z 
\ge & \log \wt Z_M \\
\ge &\ \mathbb{E} \bigg[ \1_{\{\int|\ph(1) + V_{\eta, \rho}(1)|^2 \le \NN\}} \Big( M \wedge \Big( \frac{\beta}{p} \int |\ph(1) + V_{\eta, \rho}(1)|^p \Big)\Big)  \\
 &\phantom{\ \E\bigg[]}  - \frac{1}{2} \int_0^1 \int | \partial_x \dot V_{\eta, \rho}(t) |^2 + \alpha|\dot V_{\eta, \rho}|^2 dt \bigg] \\
 \ge &\ \mathbb{E} \bigg[ \1_{\{\int|\ph(1) + V_{\eta, \rho}(1)|^2 \le \NN\}} \frac{(1-\dl)\beta} p \int |\ph(1) + V_{\eta, \rho}(1)|^p \\
 &\phantom{\ \E\bigg[]}  - \frac{1}{2} \int_0^1 \int | \partial_x \dot V_{\eta, \rho}(t) |^2 + \alpha|\dot V_{\eta, \rho}|^2 dt \bigg] \\
 \ge &\ \E\bigg[\1_{{E}} \frac{(1-\dl)\beta} p \int |\ph(1) - X_\eta(1) + \rho|^p \\
 &\phantom{\ \E\bigg[]} -\frac{1}{2} \int_0^1 \norm{-\partial_t X_\eta(1) + \rho}_{H^1_\alpha}^2 \bigg]\\
 \ge &\ \E\Big[\1_{{E}} \frac{(1 - \dl)^2\beta}{p} \int \rho^p\Big] - C_\dl \E\Big[\int |\ph(1) - X_\eta|^p\Big] \\
 &\ - \frac 12 \E\bigg[ \int_0^1 \norm{-\partial_t X_\eta}_{H^1_\alpha}^2\bigg] - \frac 12 \norm{\rho}_{H^1_\alpha}^2 \\
\ge&\ \frac{\P(E) (1-\dl)^2 \beta}{p} \int \rho^p - \frac 12 \int |\partial_x \rho|^2 \\
&\ - \frac{\alpha (1-\dl)^2 \NN}2 - C_\dl \E\Big[\int |\ph(1) - X_\eta|^p\Big] - \frac 12 \E\bigg[ \int_0^1 \norm{-\partial_t X_\eta}_{H^1_\alpha}^2\bigg].
\end{align*}
By taking the supremum over $\rho$ in this expression, we obtain 
\begin{align}
\log \wt Z \ge&\ -B(\P(E) (1-\dl)^2 \beta, (1-\dl)^2 \NN)  \label{div1} \\
&- \frac{\alpha (1-\dl)^2 \NN}2 - C_\dl \E\Big[\int |\ph(1) - X_\eta|^p\Big] - \frac 12 \E\left[ \int_0^1 \norm{-\partial_t X_\eta}_{H^1_\alpha}^2\right], \label{div2}
\end{align}
where $B$ is as in \eqref{min:eqntorus}. By Lemma \ref{OUapprox}, the quantities in \eqref{div2} are finite for every choice of $\delta, \eta$. Therefore, in view of Lemma \ref{gns2}, (4), if we can show that for every $\dl > 0$,
$$ \lim_{\eta \to 0} \P(E) = 1,$$
we obtain that $B(\P(E) (1-\dl)^2 \beta, (1-\dl)^2 \NN) = -\infty$ as long as $\eta$ is small enough. 
By Lemma \ref{OUapprox}, $\ph(1) - X_\eta(1)$ is Gaussian, and $\E| \ph(1) - X_\eta(1) |^2 \les \eta (|\log \eta| + \log L)$. Therefore, by Markov's inequality,
\begin{align*}
P(E^c) \le \frac{L \eta (|\log \eta| + \log L)}{\dl^2 \NN} \to 0
\end{align*}
as $\eta \to 0$.
\end{proof}

\section{Supercritical case and strongly nonlinear critical case}
In this section we show the concentration of measure phenomenon around the optimisers for \eqref{min:eqn} provided by Theorem~\ref{thm:2}.
Similarly to Section~\ref{ss:MAIN_PROOF}, we define 
$$\wt Z_L := \E\left[\exp\left(\big(\frac\beta {pL^\gamma} \int |\ph|^p\big) \1_{\{M(\ph) \le \NN L\}}\right)\right],$$
Since $ \wt Z_L - 1 \le Z_L \le \wt Z_L$, proving \eqref{thm2: asymptotic} and \eqref{thm2b: asymptotic} is equivalent to proving the same statements with $Z_L$ substituted by $\wt Z_L$.

\begin{proposition} \label{prop: lb}
Let $p, \alpha, \NN,\gamma$ be as in Theorem \ref{thm:2}. Then we have
\begin{equation}
\liminf_{L \to \infty} \frac{\log \wt Z_L}{L^{\frac{p+2-4\gamma}{6-p}}} \ge -A(\beta, \NN).
\end{equation}
Similarly, if $p, \alpha, \NN$ are as in Theorem \ref{thm:2b}, we have that 
\begin{equation}
\liminf_{L \to \infty} \frac{\log \wt Z_L}{L} \ge -A(\beta, \NN)(1 + o(1))
\end{equation}
as $\beta \to \infty$.
\end{proposition}
\begin{proof}
%By the formula \eqref{eqn:varform}, the statement is equivalent to proving that, for any $\eps > 0$ and for any $L$ big enough, 
%\begin{align}
%&\inf_{V \in \mathbb H^1_a} \E\left[-\Big(\frac\beta {pL^\gamma} \int |\ph(1) + V(1)|^p\Big) \1_{\{M(\ph(1)+V(1)) \le \NN L\}} + \frac 12 \int_0^1 \int |\partial_x \dot V(t)|^2 + \alpha |\dot V|^2 dt\right] \notag \\
%\le&({L^{\frac{p+2-4\gamma}{6-p}}} + \eps) \min_{\int_\R |\ph|^2 = \NN} \left(-\frac \beta p\int_\R |\ph|^p + \frac12 \int_\R |\partial_x \ph|^2\right). \label{minimum}
%\end{align}
Notice that by Proposition \ref{prop:1a}, we can use the formula \eqref{BD} for $\wt Z_L$.
Let $Q_p \in H^1(\R)$ be a minimiser for \eqref{min:eqn}, the existence of which is guaranteed by Lemma \ref{minimisers}, and let $\delta > 0$. Define
\begin{equation}
w_\delta(x) = L^\frac12 \delta^{-\frac12} Q_p\big(\frac{x}{\delta}\big) - \frac1L\int_{-\frac L2}^\frac L2 L^\frac12 \delta^{-\frac12} Q_p\big(\frac{x}{\delta}\big) d x,
\end{equation}
and let $X_\eta$ be the process defined in Lemma \ref{OUapprox}. Take
$$ V= -X_\eta + t (1-\eps) w_\delta. $$
Then 
\begin{align*}
& \E \left[\Big(\int |\ph(1) + V(1)|^p\Big) \1_{\{M(\ph(1)+V(1)) \le \NN L\}} \right] \\
\ge &~ \E\left[ -C_\eps \int |\ph(1) - X_\eta|^p + (1-\eps)^{2p} \Big(\int |w_\delta|^p\Big) \1_{\{M(\ph(1)-X_\eta) \le \eps^2 \NN L\}} \right]\\
= &~ \E\left[ -C_\eps \int |\ph(1) - X_\eta|^p\right] \\
&\ + (1-\eps)^{2p} \Big(\int |w_\delta|^p\Big) \P({\{M(\ph(1)-X_\eta) \le \eps^2 \NN L\}}) \\
\ge & - C_\eps' \eta^p L |\log \eta| + \exp\Big(-c\frac{\eps^2\NN}{\eta |\log \eta|}\Big) (1-\eps)^{2p} \Big(\int |w_\delta|^p\Big),
\end{align*}
and 
\begin{align*}
&\frac 12 \int_0^1 \int |\partial_x \dot V(t)|^2 + \alpha |\dot V(t)|^2 dt \\
\le&~\E\left[C_\eps \int_0^1 \int |\partial_x \dot X_\eta|^2 + \alpha |\dot X_\eta|^2 d t\right] + \frac12 \int |\partial_x  w_\delta|^2 \\
\le&~ C_\eps' L\frac{|\log \eta|}{\eta} + \frac12 \int | \partial_x w_\delta |^2.
\end{align*}
Therefore, by choosing $\eta \ll 1$ (depending only on $\eps$), and $\delta = L^{-\frac {p-2-2\gamma}{6-p}}$ (with $\gamma = \frac p2 - 1$ under the hypotheses of Theorem \ref{thm:2b}),  
we get for $L$ big enough and $\eps \ll 1$,
\begin{align*}
&\sup_{V \in \mathbb H^1_a} \E\bigg[\Big(\frac\beta {pL^\gamma} \int |\ph(1) + V(1)|^p\Big) \1_{\{M(\ph(1)+V(1)) \le \NN L\}} \\
 &\phantom{\ \sup_{V \in \mathbb H^1_a} \E\bigg[]} - \frac 12 \int_0^1 \int |\partial_x \dot V(t)|^2 + \alpha |\dot V(t)|^2 dt\bigg] \\
\ge& \frac{1}{L^\gamma}\ (1-\eps)^{3p}\frac \beta p \Big(\int |w_\delta|^p\Big) - \frac12 \int | \partial_x w_\delta|^2 - C_\eps'' L \\
\ge&\ {L^{\frac{p+2-4\gamma}{6-p}}} \bigg((1-\eps)^{3p} \frac \beta p \int_\R \bigg|Q_p - \delta^\frac 12 \int_\R Q_p\bigg|^p - \frac 12 \int_\R |\partial_x Q_p|^2 - \eps \bigg) - C''_\eps  L\\
= &\ L^{\frac{p+2-4\gamma}{6-p}} (-A((1-\eps)^{3p}\beta, \NN) - \eps) +O(L).
\end{align*}
Therefore, by \eqref{BD}, we have that 
$$  \frac{\log \wt Z_L}{L^{\frac{p+2-4\gamma}{6-p}}} \ge - A((1-\eps)^{3p}\beta, \NN) + \eps + O(L^{1 - \frac{p+2-4\gamma}{6-p}}). $$
Since $\eps$ is arbitrary, by sending $\eps \to 0$, from Lemma \ref{minimisers}, (3) we obtain 
$$ \liminf_{L \to \infty} \frac{\log \wt Z_L}{L^{\frac{p+2-4\gamma}{6-p}}} \ge - A(\beta, \NN)$$
when $\gamma < \frac p2 -1$, i.e.\ under the hypotheses of Theorem \ref{thm:2}, and from \eqref{Ascaling}, by choosing $\eps \to 0$ appropriately as $\beta \to \infty$, we have
$$ \liminf_{L \to \infty} \frac{\log \wt Z_L}{L^{\frac{p+2-4\gamma}{6-p}}} \ge - A(\beta, \NN)(1 + O(\eps)) + O(C_\eps'') = - A(\beta, \NN)(1+ o(1))$$
as $\beta \to \infty$ when $\gamma = \frac p2 -1$, i.e.\ under the hypotheses of Theorem \ref{thm:2b}.
\end{proof}
\begin{proposition} \label{prop: ub}
Let $p, \alpha, \NN,\gamma$ be as in Theorem \ref{thm:2}. Then we have
\begin{equation}\label{eqn: ub}
\limsup_{L \to \infty} \frac{\log \wt Z_L}{L^{\frac{p+2-4\gamma}{6-p}}} \le -A(\beta, \NN).
\end{equation}
Similarly, if $p, \alpha, \NN$ are as in Theorem \ref{thm:2b}, we have that 
\begin{equation}\label{eqn: ub2}
\limsup_{L \to \infty} \frac{\log \wt Z_L}{L} \le -A(\beta, \NN)(1 + o(1))
\end{equation}
as $\beta \to \infty$.
\end{proposition}
\begin{proof}
%By the formula \eqref{eqn:varform}, the statement is equivalent to proving that, for any $\eta > 0$ and for any $L$ big enough, 
%\begin{align}
%&\inf_{V \in \mathbb H^1_a} \E\left[-\Big(\frac\beta {pL^\gamma} \int |\ph(1) + V(1)|^p\Big) \1_{\{M(\ph(1)+V(1)) \le \NN L\}} + \frac 12 \int_0^1 \int |\partial_x \dot V(t)|^2 + \alpha |\dot V|^2 dt\right] \notag \\
%\ge&L^{\frac{p+2-4\gamma}{6-p}} (1 - \eta) \min_{\int_\R |\ph|^2 = \NN} \left(-\frac \beta p\int_\R |\ph|^p + \frac12 \int_\R |\partial_x \ph|^2\right). \label{minimum2}
%\end{align}
As argued in the proof of Proposition \ref{prop: lb}, we can apply the formula \eqref{BD} to $\wt Z$. Therefore,
\begin{align}
&\ \log \tilde Z_L \notag \\
\le&\ \sup_{V \in \mathbb H^1} \E\left[\Big(\frac\beta {pL^\gamma} \int |\ph(1) + V(1)|^p\Big) \1_{\{M(\ph(1)+V(1)) \le \NN L\}} - \frac 12 \int_0^1 \int |\partial_x \dot V(t)|^2 \right] \notag  \\
= &~  \E\left[\sup_{V \in H^1} \Big(\frac\beta {pL^\gamma} \int |\ph(1) + V|^p\Big) \1_{\{M(\ph(1)+V) \le \NN L\}} - \frac 12 \int |\partial_x  V|^2 \right] \label{eqn: E1}
\end{align}
Let $ \ph_\eps = e^{\eta \Delta} \varphi, $ with $\eta$ being the biggest value such that
$$\norm{\ph-\ph_\eps}_{L^2} \le \eps \NN^\frac12 L^\frac12. $$
We make the change of variable $V = W - \ph_\eps(1)$ in \eqref{eqn: E1}. We have that 
\begin{align}
&~\eqref{eqn: E1} \notag \\
= &~ \E\Big[\sup_{W \in H^1}\Big(\frac\beta {pL^\gamma} \int |\ph(1) - \ph_\eps (1) + W|^p\Big) \1_{\{M(\ph(1)-\ph_\eps(1)+W) \le \NN L\}} \notag\\
&~\phantom{\sup_{W \in H^1} \E\Big[]} - \frac 12 \int |\partial_x  (-\ph_\eps(1) + W)|^2 \Big] \notag\\
\le &~ \E\Big[\sup_{W \in H^1}\Big(\frac\beta {pL^\gamma} \int |\ph(1) - \ph_\eps (1) + W|^p\Big) \1_{\{M(W) \le (1+\eps)\NN L\}}\notag\\
&~\phantom{\sup_{W \in H^1} \E\Big[]} - \frac 12 \int |\partial_x  (-\ph_\eps(1) + W)|^2 \Big] \notag\\
\le&~ C_\eps \E\Big[\frac\beta {pL^\gamma} \int |\ph(1)|^p + \int |\partial_x \ph_\eps(1)|^2 \Big] \notag\\
&~- \inf_{W \in H^1,\atop \norm{W}_{L^2}^2 \le \NN(1+\eps)L}-\frac{(1+\eps)\beta}{pL^\gamma}  \int |W|^p + \frac{(1-\eps)}{2} \int |\partial_x W|^2  \notag \\
\le&~ \tilde C_\eps \Big(L^{1-\gamma}(\log L)^\frac p2  + L \E\Big[ \eta^{-2} \Big]^\frac12 \Big)- (1 -\eps) B\Big(\frac{(1+\eps)\beta}{(1-\eps) L^\gamma},(1+\eps)L\NN\Big) \notag \\
\le&~ \tilde C_\eps \Big(L^{1-\gamma}(\log L)^\frac p2  + L \E\Big[ \eta^{-2} \Big]^\frac12 \Big)- (1 -\eps) A\Big(\frac{(1+\eps)\beta}{(1-\eps) L^\gamma}, (1+\eps) L \NN\Big), \label{eqn: E3}
\end{align}
where the last inequality follows from Lemma \ref{gns2}, (3).
We have that
\begin{align*}
\P(\eta^{-1} > M) =&~ \P\left(\norm{e^{\frac 1 M \Delta} \ph - \ph}_{L^2}^2 > \eps^2 \NN L\right) ,
\end{align*}
and $e^{\frac 1 M \Delta} \ph - \ph$ is Gaussian with 
$$\E\Big[|e^{\frac 1 M \Delta} \ph - \ph|^2\Big] \les  M^{-\frac12}, $$
so 
$$\P\left(\norm{e^{\frac 1 M \Delta} \ph - \ph}_{L^2}^2 > \eps^2 \NN L\right)  \le \exp\Big(-c {\eps^2 \NN M^\frac12}\Big),$$
hence 
$\E[\eta^{-2}] \les_{\eps, \NN} 1$, so
\begin{equation} \label{eqn: E4}
\tilde C_\eps \Big(L^{1-\gamma}  + L \E\Big[ \eta^{-2} \Big]^\frac12 \Big) \gtrsim_{\eps,\NN} - L.
\end{equation}
Moreover, by \eqref{Ascaling}, we have that 
\begin{equation}\label{eqn: E5}
A\Big(\frac{(1+\eps)\beta}{(1-\eps) L^\gamma}, (1+\eps) L\NN \Big) = L^{\frac{p+2-4\gamma}{6-p}} A\Big(\frac{(1+\eps)\beta}{(1-\eps)}, (1+\eps)\NN \Big). 
\end{equation}
Putting \eqref{eqn: E4} and \eqref{eqn: E5} into \eqref{eqn: E3}, we obtain
\begin{equation}
\log \wt Z_L \le C_{\eps,\NN} L - L^{\frac{p+2-4\gamma}{6-p}} A\Big(\frac{(1+\eps)\beta}{(1-\eps)}, (1+\eps)\NN \Big),
\end{equation}
so 
$$ \limsup_{L \to \infty} \frac{\log \wt Z_L}{L^{\frac{p+2-4\gamma}{6-p}}} \le -A\Big(\frac{(1+\eps)\beta}{(1-\eps)}, (1+\eps)\NN \Big). $$
by continuity of $A$ (Lemma \ref{minimisers}, (3) and \eqref{Ascaling}), \eqref{eqn: ub} and \eqref{eqn: ub2} follow. 
\end{proof}
With the following lemmas, we discuss the concentration phenomenon of the measure around the strip $S_\dl$.
\begin{lemma} \label{lemma: ub2}
Let $\dl > 0$. 

Let $p, \alpha, \NN,\gamma$ be as in Theorem \ref{thm:2}. Then there exists $\eps_1 = \eps_1(\dl)$ such that 
\begin{align*}
&~\limsup_{L\to \infty} \frac{ \log  \E\left[\exp\big(\frac\beta {pL^\gamma} \int |\ph|^p\big) \1_{\{M(\ph) \le \NN L\}} \1_{\{\ph \not \in S_\delta \}}\right] }{L^{\frac{p+2-4\gamma}{6-p}}} \le -A(\beta, \NN) - \eps_1.
\end{align*}
Similarly, if $p, \alpha, \NN$ are as in Theorem \ref{thm:2b}, we have that 
\begin{align*}
&~\limsup_{L\to \infty} \frac{ \log  \E\left[\exp\big(\frac\beta {pL^\gamma} \int |\ph|^p\big) \1_{\{M(\ph) \le \NN L\}} \1_{\{\ph \not \in S_\delta \}}\right] }{L} \le -(1-\eps_1)A(\beta, \NN) 
\end{align*}
as $\beta \to \infty$.
\end{lemma}
\begin{proof}
We have that 
\begin{align*}
&~\E\left[\exp\big(\frac\beta {pL^\gamma} \int |\ph|^p\big) \1_{\{M(\ph) \le \NN L\}} \1_{\{\ph \not \in S_\delta \}}\right] \\
\le &~ \E\left[\exp\big(\frac\beta {pL^\gamma} \int |\ph|^p \1_{\{M(\ph) \le \NN L\}} \1_{\{\ph \not \in S_\delta \}}\big)\right],
\end{align*}
hence, by the Bou\'e-Dupuis formula \eqref{BD}, it is enough to prove that there exists $\eps_1 = \eps_1(\dl) > 0$ such that 
\begin{align}
&~\sup_{V \in \mathbb H^1} \E\Big[\Big(\frac\beta {pL^\gamma} \int |\ph(1) + V(1)|^p\Big) \1_{\{M(\ph(1)+V(1)) \le \NN L\}}\1_{\{M(\ph(1)+V(1)) \not\in S_\delta\}}\notag \\
&\phantom{\inf_{V \in \mathbb H^1_a} \E\Big[]}- \frac 12 \int_0^1 \int |\partial_x \dot V(t)|^2 + \alpha |\dot V|^2 dt\Big] \notag \\
\le&\ L^{\frac{p+2-4\gamma}{6-p}} (1 - \eps_1) A(\beta, \NN) + O(L). \label{minimum3}
\end{align}
for $L$ big enough, where $\gamma = \frac p2 -1$ if we work under the hypotheses of Theorem \ref{thm:2b}. We have that 
\begin{align}
&~\sup_{V \in \mathbb H^1} \E\Big[\Big(\frac\beta {pL^\gamma} \int |\ph(1) + V(1)|^p\Big) \1_{\{M(\ph(1)+V(1)) \le \NN L\}}\1_{\{M(\ph(1)+V(1)) \not\in S_\delta\}}\notag \\
&\phantom{\sup_{V \in \mathbb H^1_a} \E\Big[]}- \frac 12 \int_0^1 \int |\partial_x \dot V(t)|^2 + \alpha |\dot V|^2 dt\Big] \notag \\
\le &~\E\Big[\sup_{V \in \mathbb H^1}\Big(\frac\beta {pL^\gamma} \int |\ph(1) + V|^p\Big) \1_{\{M(\ph(1)+V) \le \NN L\}}\1_{\{M(\ph(1)+V) \not\in S_\delta\}}\notag \\
&\phantom{\sup_{V \in \mathbb H^1_a} \E\Big[]}- \frac 12 \int |\partial_x V(t)|^2\Big] \label{F1} 
\end{align}
Let $\eps \ll 1$ to be determined later, and let $ \ph_\eps = e^{\eta \Delta} \phi, $ with $\eta$ being the biggest value such that
$$\norm{\ph-\ph_\eps}_{L^2} \le \eps \NN^\frac12 L^\frac12, \norm{\ph-\ph_\eps}_{L^q} \le 4^{-1} \delta L^{\frac12}\lambda^{\frac1q-\frac12}.$$
Notice that the second condition implies that, for every $x_0 \in [-L/2,L/2]$,
$$\norm{L^{-\frac12} \lambda^\frac12 (\ph - \ph_\eps)(\lambda (\cdot - x_0))}_{L^q} < 4^{-1} \dl. $$
Therefore, proceeding as in the proof of Proposition \ref{prop: ub}, we obtain (for $\eps$ small enough, depending on $\dl$)
\begin{align}
&\eqref{F1} \notag \\
\le&~ \tilde C_\eps \Big(L^{1-\gamma}  + L \E\Big[ \eta^{-2} \Big]^\frac12 \Big) \label{F2}\\
&~- \inf_{W \in H^1([-L/2,L/2]), W \not \in S_{\delta/2}, \atop \norm{W}_{L^2}^2 \le \NN(1+\eps)L, \int W =0} -\frac{(1+\eps)\beta}{pL^\gamma} \int |W|^p + \frac{(1-\eps)}2 \int |\partial_x W|^2. \label{F3}
\end{align}
We have that 
$$
\begin{aligned}
\P(\eta^{-1} > M) \le&\ \P\left(\norm{e^{\frac 1 M \Delta} \ph - \ph}_{L^2}^2 > \eps^2 \NN L\right) \\
&\ + \P\left(\norm{e^{\frac 1 M \Delta} \ph - \ph}_{L^p}^2 > 4^{-1}\delta L^{\frac12}\lambda^{\frac1q-\frac12} \right) .
\end{aligned}
$$
Moreover, $e^{\frac 1 M \Delta} \ph - \ph$ is Gaussian, and 
$$\E\big|e^{\frac 1 M \Delta} \ph - \ph\big|^2 \les M^{- \frac 12},$$
so 
$$ \P\left(\norm{e^{\frac 1 M \Delta} \ph - \ph}_{L^2}^2 > \eps^2 \NN L\right) \le \exp\Big(-c \eps^2 \NN M\Big), $$
and 
$$ \P\left(\norm{e^{\frac 1 M \Delta} \ph - \ph}_{L^q} > 4^{-1}\delta L^{\frac12}\lambda^{\frac1q-\frac12}  \right) \le \exp\Big(-c \dl^2 L^{1- \frac 2q} \lambda^{\frac 2q - 1} M\Big).$$
Therefore, we obtain that $\E[\eta^{-2}] \les_{\dl,\eps,\NN} \max(1, (L^{-1}\lambda)^{1 - \frac 2q})$ and so 
\begin{equation}\label{F4}
\eqref{F2} \les_{\eps} L\max(1, (L^{-1}\lambda)^{1 - \frac 2q}).
\end{equation}
For convenience of notation, let 
$$ \tilde \eps:= (1-\eps^2) \Big( \frac{1-\eps}{(1+\eps)^\frac p2} \Big)^{-\frac{4}{6-p}} - 1. $$
Making the change of variable $W(x) = L^\frac12(1+\eps)^\frac12 \mu^{-\frac12} U(\frac \cdot \mu)$, 
with 
$$ \mu = \Big( \frac{1-\eps}{(1+\eps)^\frac p2} \Big)^\frac{2}{6-p}\lambda,$$ 
by \eqref{stabilityT}, we obtain that 
\begin{align*}
&\eqref{F3}\\
\le&~-{L^{\frac{p+2-4\gamma}{6-p}}}(1+\tilde \eps) \inf_{U \in H^1([-\frac{L}{2\mu}, \frac{L}{2\mu}]), \norm{U}_{L^2}^2 \le \NN, \atop \sup_{x,\ta} \norm{U- e^{i\ta}Q_p(\cdot - x)}_{L^2 \cap L^p} >  \delta/2, \int U =0} -\frac{\beta}{p} \int |U|^p + \frac{1}2 \int |\partial_x U|^2 \\
\le&~-{L^{\frac{p+2-4\gamma}{6-p}}}(1+\tilde \eps) A(\beta, \NN) + \eps_0(\dl/4),
\end{align*}
where $\eps_0$ is defined in \eqref{stability}. 
Therefore, by putting these estimates into \eqref{F2}, \eqref{F3}, we obtain
\begin{align*}
&~\eqref{F1} \\
\le &~  C_{\eps} L\max(1, (L^{-1}\lambda)^{1 - \frac 2q}) - {L^{\frac{p+2-4\gamma}{6-p}}}(1+ \tilde \eps)(A(\beta,\NN) +\eps_0(\delta/4)).
\end{align*}
Therefore, by choosing $\eps$ small enough, for $L$ big enough, we obtain \eqref{minimum3}.
\end{proof}
\begin{proof}[Proof of Theorem \ref{thm:2} and \ref{thm:2b}]
The asymptotics \eqref{thm2: asymptotic} and \eqref{thm2b: asymptotic} follow directly from Proposition \ref{prop: lb} and  \ref{prop: ub}. Therefore, we just need to prove \eqref{concentration} and \eqref{concentration2b}.  

We first focus on \eqref{concentration}. By Lemma \ref{lemma: ub2} and \eqref{thm2: asymptotic}, we have that there exists $a=a(\dl)>0$ such that 
\begin{align*}
&~\limsup_{L\to\infty} \frac{ \log  \E\left[\exp\big(\frac\beta {pL^\gamma} \int |\ph|^p\big) \1_{\{M(\ph) \le \NN L\}} \1_{\{\ph \not \in S_\delta \}}\right] }{L^{\frac{p+2-4\gamma}{6-p}}}\\
 =&~  -A(\beta,\NN) - a\\
 =&~ \lim_{L\to \infty} \frac{\log(Z_L)}{L^{\frac{p+2-4\gamma}{6-p}}} - a.
\end{align*}
Therefore,
\begin{align*}
&~\limsup_{L\to\infty} \log\bigg(\frac{   \E\Big[\exp\big(\frac\beta {pL^\gamma} \int |\ph|^p\big) \1_{\{M(\ph) \le \NN L\}} \1_{\{\ph \not \in S_\delta \}}\Big] }{Z_L}\bigg) \\
= &~\lim_{L\to \infty}{L^{\frac{p+2-4\gamma}{6-p}}}\bigg( \limsup_{L \to \infty} \frac{ \log  \E\Big[\exp\big(\frac\beta {pL^\gamma} \int |\ph|^p\big) \1_{\{M(\ph) \le \NN L\}} \1_{\{\ph \not \in S_\delta \}}\Big] }{L^{\frac{p+2-4\gamma}{6-p}}} \\
&\phantom{~\lim_{L\to \infty}{L^{\frac{p+2-4\gamma}{6-p}}}\bigg()}
- \lim_{L\to \infty} \frac{\log(Z_L)}{L^{\frac{p+2-4\gamma}{6-p}}}\bigg)\\
= &~ \lim_{L\to \infty} -a{L^{\frac{p+2-4\gamma}{6-p}}} = -\infty,
\end{align*}
from which we obtain \eqref{concentration}.

We now show \eqref{concentration2b}. By Lemma \ref{lemma: ub2} and \eqref{thm2b: asymptotic}, we have that as $\beta \to \infty$, 
\begin{align*}
&~\limsup_{L\to\infty}\log\Bigg( \frac{   \E\left[\exp\big(\frac\beta {pL^{\frac p2 -1}} \int |\ph|^p\big) \1_{\{M(\ph) \le \NN L\}} \1_{\{\ph \not \in S_\delta \}}\right] }{Z_L}\Bigg) \\
\le &\ \limsup_{L\to\infty} L \Bigg( \limsup_{L\to\infty} \Bigg( \frac{  \log \E\left[\exp\big(\frac\beta {pL^{\frac p2 -1}} \int |\ph|^p\big) \1_{\{M(\ph) \le \NN L\}} \1_{\{\ph \not \in S_\delta \}}\right] }{L}\Bigg) \\
&\phantom{\ \limsup_{L\to\infty} L \Bigg()}
- \liminf_{L\to\infty} \frac{\log Z_L}{L}\Bigg) \\
= &\ \limsup_{L\to\infty} L (1-\eps_1)(-A(\beta,\NN)) - L (1+o(1))(-A(\beta,\NN)) \\
= &\ \limsup_{L\to\infty} L (\eps_1 + o(1)) A(\beta,\NN)\\
= &\ -\infty
\end{align*}
when $\beta$ is big enough.
\end{proof}

\begin{proof}[Proof of Corollary \ref{triviality}]
In view of \eqref{concentration} and \eqref{concentration2b}, it is enough to show that for every $\eta > 0$,
\begin{gather*}
\limsup_{L \to \infty} \frac 1 {Z_L} \E\left[\exp\big(\frac\beta {pL^\gamma} \int |\ph|^p\big) \1_{\{M(\ph) \le \NN L\}} \1_{\{\int_{-M}^{M} |\ph| > \eps \}} \1_{\{\ph \in S_\delta \}}\right] \le \eta
\end{gather*}
when $\dl$ is small enough (depending on $M,\eps, \eta$).
By translation invariance, for every $-L/2 < y < L/2$, we have that  
\begin{align*}
&~\frac 1 {Z_L} \E\left[\exp\big(\frac\beta {pL^\gamma} \int |\ph|^p\big) \1_{\{M(\ph) \le \NN L\}} \1_{\{\int_{-M}^{M} |\ph| > \eps \}} \1_{\{\ph \in S_\delta \}}\right]\\
&= \frac 1 {Z_L} \E\left[\exp\big(\frac\beta {pL^\gamma} \int |\ph|^p\big) \1_{\{M(\ph) \le \NN L\}} \1_{\{\int_{-M+y}^{M+y} |\ph| > \eps \}} \1_{\{\ph \in S_\delta \}}\right] ,
\end{align*}
where $-M+y, M+y$ are to be interpreted modulo $L$. Therefore,
\begin{align*}
&~\frac 1 {Z_L} \E\left[\exp\big(\frac\beta {pL^\gamma} \int |\ph|^p\big) \1_{\{M(\ph) \le \NN L\}} \1_{\{\int_{-M}^{M} |\ph| > \eps \}} \1_{\{\ph \in S_\delta \}}\right]\\
&=\frac1L\int_{-L/2}^{L/2} \frac1 {Z_L} \E\left[\exp\big(\frac\beta {pL^\gamma} \int |\ph|^p\big) \1_{\{M(\ph) \le \NN L\}} \1_{\{\int_{-M+y}^{M+y} |\ph| > \eps \}} \1_{\{\ph \in S_\delta \}}\right]  d y \\
&=\frac 1 {Z_L} \E\bigg[\exp\big(\frac\beta {pL^\gamma} \int |\ph|^p\big) \1_{\{M(\ph) \le \NN L\}} \1_{\{\ph \in S_\delta \}}\cdot \frac{|\{y: \int_{-M+y}^{M+y} |\ph|> \eps\}|}{L}\bigg] \stepcounter{equation}\tag{\theequation} \label{t0}.
\end{align*}
With $\lambda$ as in the definition of $S_\delta$, we have that 
\begin{align}
&~\int_{-M+y}^{M+y} |\ph(x)| dx \notag\\
&= \int_{\lambda^{-1}(-M+y+x_0)}^{\lambda^{-1}(M+y+x_0)} |\ph|(\lambda(x-x_0)) d x \notag\\
&\le \int_{\lambda^{-1}(-M+y+x_0)}^{\lambda^{-1}(M+y+x_0)} |\ph(\lambda(x-x_0)) - L^\frac12\lambda^{-\frac12} Q_p(x)| d x \label{t1}\\
&\phantom{\le\ }+ \int_{\lambda^{-1}(-M+y+x_0)}^{\lambda^{-1}(M+y+x_0)} L^\frac12\lambda^{-\frac12} |Q_p(x)| d x \label{t2}.
\end{align}
Therefore, we have that 
\begin{equation} \label{t1+t2}
\Big|\Big\{y: \int_{-M+y}^{M+y} |\ph|> \eps\Big\}\Big| \le \{ y : \eqref{t1} \ge \frac\eps2\} +  \{ y : \eqref{t2} \ge \frac\eps2\}.
\end{equation}
By H\"older's inequality and definition of $S_\delta$, we have 
\begin{align*}
&~\int \eqref{t1}(y) dy \\
&\le (2\lambda^{-1} ML)^\frac12 \Big( \int_{-L/2}^{L/2} \int_{\lambda^{-1}(-M+y+x_0)}^{\lambda^{-1}(M+y+x_0)} |\ph(\lambda(x-x_0)) - L^\frac12\lambda^{-\frac12} Q_p(x)|^2 d x dy\Big)^\frac12\\
&= (2\lambda^{-1} ML)^\frac12 \Big(2M\int_{-\lambda^{-1}L/2}^{\lambda^{-1}L/2} |\ph(\lambda(x-x_0)) - L^\frac12\lambda^{-\frac12} Q_p(x)|^2 d x \Big)^\frac12\\
&\le 2M \lambda^{-1} L \delta. 
\end{align*}
Therefore, 
\begin{equation*}
\eps \{ y : \eqref{t1} \ge \eps\} \le \int \eqref{t1}(y) \1_{\{ y : \eqref{t1} \ge \eps\}}dy \le 2M \lambda^{-1} L \delta,
\end{equation*}
so 
\begin{equation} \label{t1est}
 \{ y : \eqref{t1} \ge \eps\} \le 2M \lambda^{-1} L \delta \eps^{-1}.
\end{equation}
We move to estimating the term  $\{ y : \eqref{t2} \ge \eps\}$. It is easy to check that 
$$ \eqref{t2} \les L^\frac12 \lambda^{-\frac32} M \|Q_p\|_{L^\infty}, $$
so in order to have $\{ y : \eqref{t2} \ge \eps\} \neq \emptyset$ as $L \to \infty$, we need $L \gtrsim \lambda^3$. 
Recalling the exponential decay of $Q_p$,\footnote{See for instance \cite[Theorem A.1]{frank}, that provides an explicit formula for $Q_p$.}, we have that, for $L$ big enough (depending on $M,\eps$), for any $\eps_0 > 0$,
\begin{equation} \label{t2est}
 |\{ y : \eqref{t2} \ge \eps\}| \les \lambda L^{\eps_0} \1_{\{L \gtrsim \lambda^3\}}.
\end{equation}
Therefore, by \eqref{t1+t2}, \eqref{t1est}, \eqref{t2est}, we obtain 
\begin{equation*}
\Big|\Big\{y: \int_{-M+y}^{M+y} |\ph|> \eps\Big\}\Big| \les M \lambda^{-1} L \delta \eps^{-1} +  \lambda L^{\eps_0} \1_{\{L \gtrsim \lambda^3\}}.
\end{equation*}
Therefore, by \eqref{t0},
\begin{align*}
&~\frac 1 {Z_L} \E\left[\exp\big(\frac\beta {pL^\gamma} \int |\ph|^p\big) \1_{\{M(\ph) \le \NN L\}} \1_{\{\int_{-M}^{M} |\ph| > \eps \}} \1_{\{\ph \in S_\delta \}}\right]\\
&\les \frac 1 {Z_L} \E\bigg[\exp\big(\frac\beta {pL^\gamma} \int |\ph|^p\big) \1_{\{M(\ph) \le \NN L\}} \1_{\{\ph \in S_\delta \}} \bigg]\cdot \frac{M \lambda^{-1} L \delta \eps^{-1} +  \lambda L^{\eps_0}\1_{\{L \gtrsim \lambda^3\}}}{L}\\
&\ll \eta
\end{align*}
as $L \to \infty$, when $\delta$ is small enough (depending on $M,\eps,\eta$).
\end{proof}
\section{Subcritical case and weakly nonlinear critical case}
\label{sec:GaussianLimit}

This section contains the proof of Theorem \ref{thm:3}, namely the convergence of $\rho_L$ to the Gaussian measure $\mu_{OU}$ if the non-linearity is sufficiently dampened.
One cannot hope to show that the that density converges in $L^1$ in the full regime of $\gamma,p$, because for $\gamma \leq 1$, 
\[
\E  \frac{\beta}{pL^\gamma} \int_{-L/2}^{L/2} |\ph|^p  \sim L^{1 - \gamma} \gtrsim 1.
\]
 The following Lemma shows that in the weak coupling regime, one can recover the $L^1$ convergence if the density is restricted to interval of size $L^\theta \ll L$ for $\theta$ small enough.

\begin{lemma}\label{lemma:small_scale1}
Let $\gamma \ge \frac p2-1$, let $0 < \theta < \gamma$ with $\theta \le 1$, and let $\NN > \frac{1}{2\sqrt{\alpha}}$. Then we have that 
\begin{enumerate}
\item If $\gamma = \frac p2 -1$, there exists $\beta_0 = \beta_0(p,\NN,\alpha)$ such that for every $\beta < \beta_0$,
\begin{equation} \label{eqn:small_scale}
Z_\theta := \E\left[ \exp \Big( \frac{\beta}{pL^\gamma} \int_{-L^\theta/2}^{L^\theta/2} |\ph|^p \Big) \1_{ \{ M(\ph) \leq \NN L  \}} \right] \to 1 \text{ as } L\to\infty.
\end{equation}
\item If $\gamma > \frac p2-1$, then \eqref{eqn:small_scale} holds for every choice of $\beta > 0$.
\end{enumerate}
\end{lemma}
\begin{proof}
Since $\NN > \frac 1{2\sqrt \alpha}$, we have that $\P(\{ M(\ph) > \NN L  \}) \to 0$ as $L \to \infty$. Therefore, by the elementary identity 
$$ \E[e^F \1_{E}] = \E[e^{F \1_{E}}] - \P(E),   $$ 
it is enough to prove that 
$$ \E\left[ \exp \Big(  \1_{ \{ M(\ph) \leq \NN L  \}}\frac{\beta}{pL^\gamma} \int_{-L^\theta/2}^{L^\theta/2} |\ph|^p \Big) \right] \to 1 $$
as $L \to \infty$. Since the expression inside the exponential is always nonnegative, this expectation is always $\ge 1$. Therefore, it is enough to show that 
$$ \limsup_{L\to \infty} \E\left[ \exp \Big(  \1_{ \{ M(\ph) \leq \NN L  \}}\frac{\beta}{pL^\gamma} \int_{-L^\theta/2}^{L^\theta/2} |\ph|^p \Big) \right] \le 1,$$
or equivalently,
$$ \limsup_{L\to \infty} \log \E\left[ \exp \Big(  \1_{ \{ M(\ph) \leq \NN L  \}}\frac{\beta}{pL^\gamma} \int_{-L^\theta/2}^{L^\theta/2} |\ph|^p \Big) \right] \le 0.$$
By the Bou\'e-Dupuis formula \eqref{BD}, together with Proposition \ref{prop:1a}, we have that 
\begin{align*}
&\log \E\bigg[ \exp \Big(  \1_{ \{ M(\ph) \leq \NN L  \}}\frac{\beta}{pL^\gamma} \int_{-L^\theta/2}^{L^\theta/2} |\ph|^p \Big) \bigg] \\
%=& \sup_{V \in \mathbb H^1_a} \E\bigg[- \1_{ \{ N(\ph(1)+V(1)) \leq \NN L  \}}\frac{\beta}{pL^\gamma} \int_{-L^\theta/2}^{L^\theta/2} |\ph(1) + V(1)|^p \\
%& \phantom{\sup_{V \in \mathbb H^1_a} \E\bigg[]}
%+ \frac 12 \int_0^1 \norm{\dot V(s)}_{\dot H^1}^2 ds + \frac \alpha 2 \int_0^1 \norm{\dot V(s)}_{L^2}^2 ds\bigg] \\
\le& \sup_{V \in \mathbb H^1} \E\bigg[\1_{ \{ M(\ph(1)+V(1)) \leq \NN L  \}}\frac{\beta}{pL^\gamma} \int_{-L^\theta/2}^{L^\theta/2} |\ph(1) + V(1)|^p \\
& \phantom{\sup_{V \in \mathbb H^1_a} \E\bigg[]}
- \frac 12 \int_0^1 \norm{\dot V(s)}_{\dot H^1}^2 ds - \frac \alpha 2 \int_0^1 \norm{\dot V(s)}_{L^2}^2 ds\bigg] \\ 
= &\ \E\bigg[ \sup_{V \in H^1} \1_{ \{ M(\ph(1)+V) \leq \NN L  \}}\Big(\frac{\beta}{pL^\gamma} \int_{-L^\theta/2}^{L^\theta/2} |\ph(1) + V|^p - \frac 12  \norm{V}_{\dot H^1}^2 - \frac \alpha 2 \norm{ V}_{L^2}^2\Big)\bigg]
\end{align*}

Let $\eps \ll 1$ to be determined later, and let $\ph_\eps$ be defined as 
\begin{equation} \label{phi_eps}
\ph_\eps:= 
\begin{cases}
0 & \text{ if } M(\ph) \leq \frac{1+\eps}{4\alpha} L,\\
\ph \ast \rho_{\min(C\eps^4 L^\frac12\norm{\ph}_{H^\frac14}^{-4}, 1)} & \text{ otherwise },
\end{cases}
\end{equation}
where $\rho$ is a smooth, compactly supported, mollifier, $\rho_\dl(x) = \dl^{-1} \rho(\dl^{-1} x)$, and $C$ is a small enough constant. We make the change of variable $V = - \ph_\eps + W$. From the inequality 
$\norm{\ph - \ph \ast \rho_\delta}_{L^2} \les \delta^{\frac 14} \norm{\ph}_{H^\frac14}$, 
we get that 
\begin{align*}
\norm{\ph - \ph_\eps}_{L^2}^2 &\le
\begin{cases}
\norm{\ph}_{L^2}^2 \le \frac{1+\eps}{4\alpha} L & \text{ if } M(\ph) \leq \frac{1+\eps}{4\alpha}L, \\
(C\eps^4 L^\frac12\norm{\ph}_{H^\frac14}^{-4})^\frac14 \norm{\ph}_{H^\frac14} \le C \eps L^\frac18 & \text{ otherwise },
\end{cases}
\end{align*}
and so we always have that $\norm{\ph - \ph_\eps}_{L^2}^2 \le \frac{1+\eps}{4\alpha} L$ (for $L \ge 1$ and $\eps$ small enough).
Moreover, since $\norm{\ph - \ph_\eps + W}_{L^2}^2 \ge -C_\eps\norm{\ph - \ph_\eps}_{L^2}^2 + (1- \eps) \norm{W}_{L^2}^2$, by taking $\sigma \gg \NN$, we have that 
$$\{\norm{W}_{L^2}^2 \le L\sigma\} \supset \{\norm{\ph - \ph_\eps + W}_{L^2}^2 \le L \NN\}.$$ 
Therefore,
\begin{align*}
&\log \E\bigg[ \exp \Big(  \1_{ \{ M(\ph) \leq \NN L  \}}\frac{\beta}{pL^\gamma} \int_{-L^\theta/2}^{L^\theta/2} |\ph|^p \Big) \bigg] \\
\le &\ \E\bigg[ \sup_{V \in H^1}  \1_{ \{ M(\ph(1)+V) \leq \NN L  \}}\Big(\frac{\beta}{pL^\gamma} \int_{-L^\theta/2}^{L^\theta/2} |\ph(1) + V|^p \\
& \phantom{\ \E\bigg[ \sup_{V \in H^1}] }
- \frac 12  \norm{V}_{\dot H^1}^2 - \frac \alpha 2 \norm{ V}_{L^2}^2\Big)\bigg] \\
\le &\  \E\bigg[\sup_{W \in H^1} \1_{ \{ \norm{W}_{L^2}^2 \le L\sigma \}}\Big(\frac{\beta}{pL^\gamma} \int_{-L^\theta/2}^{L^\theta/2} |(\ph-\ph_\eps)(1) + W|^p \\
& \phantom{\ \E\bigg[ \sup_{V \in H^1}] }
- \frac 12  \norm{- \ph_\eps + W}_{\dot H^1}^2 - \frac \alpha 2 \norm{ - \ph_\eps + W}_{L^2}^2\Big)\bigg] \\
\le & \ \E\bigg[\sup_{W \in H^1} C_\eps \frac{\beta}{pL^\gamma} \int_{-L^\theta/2}^{L^\theta/2} |(\ph-\ph_\eps)(1)|^p + \1_{ \{ \norm{W}_{L^2}^2 \le L\sigma \}} \Big((1+\eps) \frac{\beta}{pL^\gamma} \int_{-L^\theta/2}^{L^\theta/2} |W|^p \bigg. \\
&\bigg. \phantom{\inf_{W \in H^1} }  - (1-\eps)\big(\frac 12  \norm{W}_{\dot H^1}^2 + \frac \alpha 2 \norm{ W}_{L^2}^2\big)\Big) + C_\eps \bigg(\frac 12  \norm{\ph_\eps}_{\dot H^1}^2 + \frac \alpha 2 \norm{ \ph_\eps}_{L^2}^2\bigg) \bigg]\\
= &\ \E\bigg[C_\eps \frac{\beta}{pL^\gamma} \int_{-L^\theta/2}^{L^\theta/2} |(\ph-\ph_\eps)(1)|^p + C_\eps \bigg(\frac 12  \norm{\ph_\eps}_{\dot H^1}^2 + \frac \alpha 2 \norm{ \ph_\eps}_{L^2}^2\bigg) \bigg] \\
& 
\begin{multlined}
-(1-\eps) \inf_{W\in H^1} \1_{ \{ \norm{W}_{L^2}^2 \le L\sigma \}} \Big( - \frac{(1+3\eps) \beta}{pL^\gamma} \int_{-L^\theta/2}^{L^\theta/2} |W|^p \\
\phantom{XXXXXXXXXXXXXXX} 
+ \big(\frac 12  \norm{W}_{\dot H^1}^2 + \frac \alpha 2 \norm{ W}_{L^2}^2\big)\Big). 
\end{multlined}
\numberthis \label{A1.1}
\end{align*}
Since for $\eps$ small enough, $\norm{\ph_\eps}_{L^p(-L^\theta/2, L^\theta/2)} \le \norm{\ph}_{L^p(-L^\theta/2 -1, L^\theta/2+1)}$ and $\norm{\ph \ast \rho_\delta}_{H^1} \les \delta^{-\frac34} \norm{\ph}_{H^\frac14}$, 
\begin{equation}
\begin{aligned}
&C_\eps\E\left[ \frac{\beta}{pL^\gamma} \int_{-L^\theta/2}^{L^\theta/2} |(\ph-\ph_\eps)(1)|^p +  \left(\frac 12  \norm{\ph_\eps}_{\dot H^1}^2 + \frac \alpha 2 \norm{ \ph_\eps}_{L^2}^2\right) \right] \\
\les &\ \E\left[\frac 1 {L^\gamma}  \int_{-L^\theta/2-1}^{L^\theta/2 +1 } |\ph|^p + \1_{\{\ph_\eps \neq 0\}} \max(\eps^{-6} L^{-\frac 34} \norm{\ph}_{H^\frac 14}^8,1)\right]\\
\les &\ L^{\theta - \gamma} + \P\left(\left\{\norm{\ph}_{L^2}^2 > \frac{1+\eps}{4\alpha} L\right\}\right) \E\left[(1 +\eps^{-24} L^{-3} \norm{\ph}_{H^\frac 14}^{32}) \right]\\
\les &\ L^{\theta - \gamma} + \exp(-c \eps L^\frac 12) \eps^{-24} L^{13} \to 0
\end{aligned}\label{A2}
\end{equation}
as $L \to \infty$. Here we used the large deviation estimate 
\begin{equation} \label{LDE2}
\P\left(\left\{\left(\int|\ph|^2 - \frac L {4 \alpha}\right) > M\right\}\right) \les \exp\Big(-c\frac{M}{L^\frac12}\Big),
\end{equation}
which follows from \eqref{LDE}, 
and the estimate $\E \norm{\ph}_{H^\frac 14}^{32} \les L^{16}$, which follows from \eqref{ph_Hsp}.

For the term \eqref{A1.1}, by the Gagliardo-Nirenberg-Sobolev inequality \eqref{GNStorus2}, the fact that $\| W \|_{L^2} \les_{\sigma} L^\frac 12$, and Young's inequality, for a constant $C_\eps = C_\eps(\sigma)$ that can change line to line, we obtain 
\begin{equation} \label{A3}
\begin{aligned}
& (1+3\eps) \frac{\beta}{pL^\gamma} \int_{-L^\theta/2}^{L^\theta/2} |W|^p\\
 \le &\ \frac{(1+4\eps)\beta \CGN^p}{pL^\gamma} \big(\norm{W}_{\dot H^1}^2\big)^{\frac {p-2}4} \big(\norm{W}_{L^2}^2\big)^{\frac {p+2}4} + C_\eps\beta L^{-\frac p2 + 1 - \gamma}\| W \|_{L^2}^{\frac{p+2}{2}} \\
\le &\ \frac{1-2\eps}2 \norm{W}_{\dot H^1}^2 + C_\eps \big(\frac{\beta}{L^\gamma}\big)^{\frac 4{6-p}} \big(\norm{W}_{L^2}^2\big)^{\frac {p+2}{6-p}} +  C_\eps \beta L^{-\frac p2 + 1 - \gamma}L^{\frac{p-2}{4}}\| W \|_{L^2}^{2} \\
\le&  \frac{1-2\eps}2 \norm{W}_{\dot H^1}^2 + C_\eps\Big[ \big(\frac{\beta}{L^\gamma}\big)^{\frac 4{6-p}}L^{\frac {p+2}{6-p}-1} 
+ \beta L^{-\frac{4\gamma + p -2}{4}}\Big]
\norm{W}_{L^2}^2.
\end{aligned}
\end{equation}
Therefore, 
\begin{equation} \label{A4}
\begin{aligned}
&\inf_{W\in H^1}  \1_{ \{ \norm{W}_{L^2}^2 \le L\sigma \}} \Big(-(1+\eps) \frac{(1+3\eps)\beta}{pL^\gamma} \int_{-L^\theta/2}^{L^\theta/2} |W|^p + \frac 12  \norm{W}_{\dot H^1}^2 + \frac \alpha 2 \norm{ W}_{L^2}^2\Big) \\
&\ge \eps \norm{W}_{\dot H_1}^2 + \delta(\eps) \norm{W}_{L^2}^2
\end{aligned}
\end{equation}
as long as $\eps$ is small enough, and
\begin{equation}\label{A5}
\delta(\eps) = \alpha - C_\eps\Big[ \big(\frac{\beta}{L^\gamma}\big)^{\frac 4{6-p}}L^{\frac {p+2}{6-p}-1} 
+ \beta L^{-\frac{4\gamma + p -2}{4}}\Big] > 0
\end{equation}
for $L$ big enough. It easy to check that this holds true for $\gamma = \frac p2 -1 $ and $\beta$ small enough, and for $\gamma < \frac p2 - 1$ and every $\beta$.
\end{proof}

The following Lemma (using essentially the same argument), gives an upper bound on the full partition function, which doesn't vanish as $L \to \infty $ when $\gamma \leq 1$.

\begin{lemma} \label{lemma: theta_removal}
Let $\gamma \ge \frac p2-1$, let $\NN > \frac 1 {2\sqrt\alpha}$, and when $\gamma = \frac p2-1$, let $\beta < \beta_0$, as defined in Lemma \ref{lemma:small_scale1}. 
\begin{equation}
\log{\E\Big[ \exp\Big( \frac \beta {p L^\gamma} \int |\ph|^p \Big) \1_{\{M(\ph) \le \NN L\}}\Big]} \les L^{1-\gamma}.
\end{equation}
\end{lemma}

\begin{proof}
By the the Bou\'e-Dupuis formula \eqref{BD}, we just need to prove that 
\begin{equation} \label{C1}
\begin{aligned}
&\sup_{V \in \mathbb H^1} \E\Big[ \1_{ \{ M(\ph(1)+V(1)) \leq \NN L  \}}\frac{\beta}{pL^\gamma} \int |\ph(1) + V(1)|^p  \\
&\phantom{\inf_{V \in \mathbb H^1_a} \E\Big[]}- \frac 12 \norm{V(1)}_{\dot H^1}^2  - \frac \alpha 2 \norm{ V(1)}_{L^2}^2 \Big] \les L^{1-\gamma}.
\end{aligned}
\end{equation}
%We notice that the expression \eqref{C1} is very similar to \eqref{B1}, with the only differences being of the indicator function in the integral $\int |\ph(1) + V(1)|^p$, and the term 
%$ L\1_{\{M(\ph(1)+V(1)) \le (\NN-\eps) L}.$
Proceeding exactly as in the proof of Lemma \ref{lemma:small_scale1}, we decompose $V = -\ph_\eps + W$, and we obtain the analogous of \eqref{A1.1} with $\theta =1$.
Then we proceed as in \eqref{A2} and \eqref{A3}--\eqref{A5} with the only difference that in \eqref{A2} $L^{\theta-\gamma}$ becomes $L^{1-\gamma}$.

%\begin{equation} \label{C3}
%\begin{aligned}
%& \inf_{W\in H^1} - \1_{ \{ M_\theta(W) \le L\sigma \}}\Big( (1+3\eps) \frac{\beta}{pL^\gamma} \int |W|^p\1_{[-L^\theta/2,L^\theta/2]^c} \\
%& \phantom{ \inf_{W\in H^1}}+ \big(\frac 12  \norm{W}_{\dot H^1}^2 + \frac \alpha 2 \norm{ W}_{L^2}^2\big)\Big) \\
%\ge &~\eps \norm{W}_{\dot H^1}^2 + \delta(\eps)\norm{W}_{L^2}^2 \ge 0.
%\end{aligned}
%\end{equation}
%Therefore, from \eqref{A3}, we have that 
%\begin{align*}
%& \eqref{C1}\\
%\le &~ \E\left[C_\eps \frac{\beta}{pL^\gamma} \int |(\ph-\ph_\eps)(1)|^p  + C_\eps \left(\frac 12  \norm{\ph_\eps}_{\dot H^1}^2 + \frac \alpha 2 \norm{ \ph_\eps}_{L^2}^2\right) \right] \\
%& -(1-\eps) \inf_{W\in H^1} - \1_{ \{ \norm{W}_{L^2}^2 \le L\sigma \}}\Big( (1+3\eps) \frac{\beta}{pL^\gamma} \int |W|^p \1_{[-L^\theta/2,L^\theta/2]^c} \\
%& \phantom{+ \inf_{W\in H^1}}+ \big(\frac 12  \norm{W}_{\dot H^1}^2 + \frac \alpha 2 \norm{ W}_{L^2}^2\big)\Big) \\
%\le&~ C_1 L^{1-\gamma} - 0 \les L^{1-\gamma}. 
%\end{align*}

\end{proof}

The following three  Lemmas make precise  point (2) in the discussion on page \pageref{RandomLabel}, namely  
the statement that for $K > 0$, the random variables $\ph|_{[-K,K]}$ and $\exp\Big(\frac\beta {pL^\gamma} \int_{[-L^\ta/2,L^\ta/2]^c} |\ph|^p\Big)$ are ``almost independent"  as $L \to \infty$.

The following Lemma establishes that up to an error which is $o(1)$ when $L \to \infty$, the control which arises in the variational representation for expectation of $e^{F_\theta}$, defined below,  can be chosen to be supported away from $[-L^\theta/4, L^\theta/4]$.
\begin{lemma} \label{lemma: large_scale}
Let $\gamma \geq \frac p2-1$, and let $0<\theta<\gamma\wedge 1$. Call 
\begin{equation} \label{Ntheta}
M_\theta(u):=\int |u|(x)^2 \1_{[-L^\theta/2,L^\theta/2]^c}(x) dx 
\end{equation}
and 
\begin{equation*}
F_\theta(u) = \1_{ \{ M_\theta(u) \leq \NN L  \}}\frac{\beta}{pL^\gamma} \int \left(|u|^p \1_{[-L^\theta/2,L^\theta/2]^c} \right). 
\end{equation*}
We point out, that $F_\theta(u)$ only depends on the values of $u$ on the complement of the interval $[-L^\theta/2, L^\theta/2]$.

Then, as $L\to \infty$, we have the following equalities: 
\begin{align}
&\sup_{V \in \mathbb H^1_a} \E\Big[ F_\theta (\ph(1) + V(1) ) 
 -  \frac12 \int \norm{\dot V(s)}_{\dot H^1}^2  + \alpha \norm{\dot V(s)}_{L^2}^2 ds \Big]\label{eqn: scale1} \\
 \notag
=&\ \sup_{V \in \mathbb H^1_a}\E\Big[  F_\theta (\ph(1) + V(1) ) \\
&\phantom{ \sup_{V \in \mathbb H^1_a}}
   - \frac12 \iint \left(| \partial_x \dot V(s)|^2 + \alpha|\dot V(s)|^2\right)\1_{(-L^\theta/4,L^\theta/4)^c} ds dx\Big] + o(1) \label{eqn: scale2}\\
=&\sup_{V \in \mathbb H^1_a,\atop \supp(V) \subseteq(-L^\theta/4,L^\theta/4)^c}\E\Big[ F_\theta (\ph(1) + V(1) )  \notag \\
&\phantom{\inf_{V \in \mathbb H^1_a,\atop \supp(V) \subseteq(-L^\theta/2,L^\theta/2)^c} -\frac{\beta}{p} \int}- \frac12 \int \norm{\dot V(s)}_{\dot H^1}^2 + \alpha \norm{\dot V(s)}_{L^2}^2 ds\Big] + o(1) \label{eqn: scale3}
\end{align}
\end{lemma}

\begin{proof}
We clearly have that $\eqref{eqn: scale1}\le\eqref{eqn: scale2}$, because the expression in the sup is  bigger, and that $\eqref{eqn: scale3} \le \eqref{eqn: scale1}$, since the supremum is taken over a smaller set. Therefore, it is enough to prove that $\eqref{eqn: scale3} \ge \eqref{eqn: scale2}$.

Let $V \in \mathbb H^1_a$, and given $y_0 \ge 0$, define the operator $T_{y_0}$ as 
\begin{equation}
\begin{aligned}
&(T_{y_0} V)(x)\\
&:=\ \begin{cases}
V(x) & \text{ if } |x| \ge y_0 \\
\frac{V(y_0) - e^{-2\sqrt{\alpha} y_0}V(-y_0)}{1-e^{-4\sqrt{\alpha} y_0}} e^{-\sqrt{\alpha}(y_0-x)} + 
\frac{V(-y_0) - e^{-2\sqrt{\alpha} y_0}V(y_0)}{1-e^{-4\sqrt{\alpha} y_0}} e^{-\sqrt{\alpha}(y_0+x)}  & \text{ if } |x| \le y_0.
\end{cases}
\end{aligned}
\end{equation}
It is easy to see that $\partial_s(T_{y_0}V(s)) = T_{y_0}\dot V(s)$ is the only solution to the equation 
\begin{equation}
\begin{cases}
\partial_x^2 Y + \alpha Y = 0 & \text{ in the interval } [-y_0,y_0] \\
Y \equiv \dot V & \text{ on } [-y_0,y_0]^c\\
Y(-y_0) = \dot V(-y_0)\\
Y(y_0) = \dot V(y_0).
\end{cases}
\end{equation}
Therefore, by a simple variational argument, $T_{y_0}V$ is the function in $\mathbb H^1$ which is equal to $V$ outside of the interval $[-y_0,y_0]$ and minimises the expression 
$$\frac12 \int \norm{\dot V(s)}_{\dot H^1}^2  +\alpha \norm{\dot V(s)}_{L^2}^2 ds.$$
Let $\rho$ be a smooth even function with $\rho \equiv 1$ on the interval $[-1,1]$ and $\rho \equiv 0$ out of the interval $[-1-\frac1{100}, 1 +\frac1{100}]$. Let $\rho_M(x):= \rho(\frac x M)$. Given $V \in \mathbb H^1_a$, let 
$$W = W(V):= (1-\rho_{L^\theta/4})T_{\frac38L^\theta}V.$$
We have that $\supp(W) \subseteq (-L^\theta/4, L^\theta/4)^c$, $W \equiv V$ on $[-L^\theta/2, L^\theta/2]^c$, for a constant $C$ that can change line to line, 
\begin{equation}
\begin{aligned}
&
\E\Big[ F_\theta(\ph(1) + W(1) )) 
- \frac12 \int \norm{\dot W(s)}_{\dot H^1}^2 + \alpha \norm{\dot W(s)}_{L^2}^2 ds\Big]
\\
=&\ \E\Big[ F_\theta(\ph(1) + V(1) ) \\
& \hspace{8mm}- \frac12 \int \norm{(1-\rho_{L^\theta/4})T_{\frac38L^\theta}\dot V(s)}_{\dot H^1}^2 + \alpha \norm{(1-\rho_{L^\theta/4})T_{\frac38L^\theta}\dot V(s)}_{L^2}^2 ds\Big]\\
\ge &\ \E\Big[F_\theta(\ph(1) + V(1) ) \\
&\hspace{8mm}- \frac12 (1 + C e^{-\frac 19 \sqrt{\alpha} L^\theta})\left( \int \norm{T_{\frac38L^\theta}\dot V(s)}_{\dot H^1}^2 + \alpha \norm{T_{\frac38L^\theta}\dot V(s)}_{L^2}^2 ds\right)\Big]\\
\ge &\ \E\Big[ F_\theta(\ph(1) + V(1) ) \\
&\hspace{8mm}- \frac12 (1 + C e^{-\frac 19 \sqrt{2 \alpha} L^\theta})\left( \int \norm{ (\partial_x \dot V(s) )\1_{(-L^\theta/4,L^\theta/4)^c}}_{L^2}^2 + \alpha \norm{\dot V(s)\1_{(-L^\theta/4,L^\theta/4)^c}}_{L^2}^2 ds\right)\Big]\\
\end{aligned} \label{eqn: scale4}
\end{equation}
Let now $V_n$ be a maximising sequence for the expression in \eqref{eqn: scale2}, so that 
\begin{equation} \label{eqn: scale5}
\begin{aligned}
\eqref{eqn: scale2} =  \lim_n \E\Big[ & F_\theta(\ph(1) + V_n(1) ) \\
&- \frac12 \int \norm{\dot V_n(s)\1_{(-L^\theta/4,L^\theta/4)^c}}_{\dot H^1}^2  +\alpha \norm{\dot V_n(s)\1_{(-L^\theta/4,L^\theta/4)^c}}_{L^2}^2 ds \Big].
\end{aligned}
\end{equation}
Note that we can assume without loss of generality that $V_n = T_{L^\theta/2} V_n$.
Let $W_n = W(V_n)$. By definition of \eqref{eqn: scale3}, since $\supp W_n \subseteq (-L^\theta/4, L^\theta/4)^c$, we have that 
\begin{equation} \label{eqn: scale6}
\begin{aligned}
&~\eqref{eqn: scale3} \\
& \ge \liminf_n \E\Big[ F_\theta(\ph(1) + W_n(1) ) 
- \frac12 \int \norm{\dot W_n(s)}_{\dot H^1}^2 + \alpha \norm{\dot W_n(s)}_{L^2}^2 ds\Big]\\
& \ge \liminf_n \E\Big[ F_\theta(\ph(1) + V_n(1) ) \\
&\phantom{\inf_{W_n \in \mathbb H^1_a} -\frac{\beta}{p} \int}
- \frac12 (1 + C e^{-\frac 19 \sqrt{2 \alpha} L^\theta})\Big(\int \norm{\dot V_n(s)\1_{(-L^\theta/4,L^\theta/4)^c}}_{\dot H^1}^2  \\
&\phantom{\inf_{W_n \in \mathbb H^1_a} -\frac{\beta}{p} \int
+ \frac12 (1 + C e^{-\frac 14 \sqrt \alpha L^\theta})\Big()}+\alpha \norm{\dot V_n(s)\1_{(-L^\theta/4,L^\theta/4)^c}}_{L^2}^2 ds\Big) \Big]\\
& = \eqref{eqn: scale2} + \liminf_n \frac {C e^{-\frac 19 \sqrt{2 \alpha} L^\theta}}2\E\Big[ \int \norm{\dot V_n(s)\1_{(-L^\theta/4,L^\theta/4)^c}}_{\dot H^1}^2  +\alpha \norm{\dot V_n(s)\1_{(-L^\theta/4,L^\theta/4)^c}}_{L^2}^2 ds \Big].
\end{aligned}
\end{equation}
Since $V_n$ is a minimising sequence, by testing the expression \eqref{eqn: scale2} with $V = 0$, we obtain that 
\begin{align*}
 &\E\Big[ \int \norm{\dot V_n(s)\1_{(-L^\theta/4,L^\theta/4)^c}}_{\dot H^1}^2  +\alpha \norm{\dot V_n(s)\1_{(-L^\theta/4,L^\theta/4)^c}}_{L^2}^2 ds \Big] \\
 &\le2\E\Big[F_\theta(\ph(1) + V_n(1) ) \Big] + C L^{1- \gamma}\\
  &\le\E\Big[ 2 (1+\eps) F_\theta( V_n(1) )\Big] + C_\eps L^{1- \gamma}\\
  &\le \E\Big[-A\big(\frac{2\beta(1+\eps)}{L^\gamma}, (\NN^\frac12L^\frac12 + \|\ph\|_{L^2})^2) + \frac 12 \int \|T_{L^\theta/2} \dot V_n(s)\|_{\dot H^1}^2 ds  \Big]+ C_\eps L^{1- \gamma},
\end{align*}
where we recall that $A$ is the quantity defined in \eqref{min:eqn}. Therefore, recalling that by Lemma \ref{gns2}, $A\le 0$, and by \eqref{Ascaling}, and that by our assumption $V_n = T_{L^\theta/2} V_n$, we obtain that
\begin{align*}
 &\E\Big[ \int \norm{\dot V_n(s)\1_{(-L^\theta/4,L^\theta/4)^c}}_{\dot H^1}^2  +\alpha \norm{\dot V_n(s)\1_{(-L^\theta/4,L^\theta/4)^c}}_{L^2}^2 ds \Big] \\
 &\lesssim C_\eps L^{1-\gamma} + \E\Big[\Big(1 + \frac{\|\ph\|_{L^2}}{L^\frac12}\Big)^\frac{4 + 2p}{6-p}\Big] L^{\frac {p+2-4\gamma} {6-p} }\\
  &\lesssim L^{1-\gamma} + L^{\frac {p+2-4\gamma} {6-p} }. 
 \end{align*} 
 Therefore, from \eqref{eqn: scale6}, we obtain
 $$\eqref{eqn: scale3} \le \eqref{eqn: scale2} + o(1). $$
\end{proof}

The following Lemma complements the previous Lemma~\ref{lemma: large_scale}. This time we are  interested in an observable $F$ that only depends on the 
values of $\varphi$ in the compact interval $[-K,K]$. The Lemma states that up to an error which is $o(1)$ as $L \to \infty$ the control $V$ in the variational problem for the expectation  can be chosen to be compactly 
supported in $[-L^\theta/4, L^\theta/4]$.

\begin{lemma} \label{lemma: small_scale2}
Let 
 $F:\D'(K)\to \R$ be a bounded Borel functional, and let $0 < \theta < 1$. Then, as $L \to \infty$, the following holds.
\begin{align}
& \sup_{V \in \mathbb H^1_a} \E\Big[ F(\ph(1) + V(1)) - \frac12 \int \norm{\dot V(s)}_{\dot H^1}^2  + \alpha \norm{\dot V(s)}_{L^2}^2 ds\Big] \label{eqn: Fscale1} \\
=& \sup_{V \in \mathbb H^1_a} \E\Big[ F(\ph(1) + V(1)) - \frac12  \iint   \left( |\partial_x \dot V(s)|^2  +\alpha |\dot V(s)|^2\right) \1_{[-L^\theta/4, L^\theta/4]} ds\Big] +o(1)\label{eqn: Fscale2} \\
=& \sup_{V \in \mathbb H^1_a, \atop \supp(V) \subseteq [-L^\theta/4, L^\theta/4]} \E\Big[ F(\ph(1) + V(1)) -  \frac12 \int  \norm{\dot V(s)}_{\dot H^1}^2  +\alpha \norm{\dot V(s)}_{L^2}^2 ds\Big] +o(1) \label{eqn: Fscale3}
\end{align}
\end{lemma}
\begin{proof}
Similarly to the proof of Lemma \ref{lemma: large_scale}, we have that $\eqref{eqn: Fscale2} \ge \eqref{eqn: Fscale1} \ge \eqref{eqn: Fscale3}$, 
so we just need to prove that $\eqref{eqn: Fscale3} \ge \eqref{eqn: Fscale2} + o(1)$. 
Let $h$ be the function
\begin{align*}
h_V(x) := & \frac{ V(-K) - e^{-\sqrt{\alpha} (L-2K)} V(K)}{1- e^{-2\sqrt{\alpha} (L-2K)}} e^{-\sqrt{\alpha} (L-K-x)} \\
&+  \frac{ V(K) - e^{-\sqrt{\alpha} (L-2K)} V(-K)}{1- e^{-2\sqrt{\alpha} (L-2K)}} e^{-\sqrt{\alpha} (x-K)},
\end{align*}
and define $T_KV$ as 
\begin{equation}
T_K V (x) :=
\begin{cases}
V(x) & \text{ if } |x| \le K,\\
h_V(x) & \text{ if } x \ge K, \\
h_V(x+L)  & \text{ if } x \ge -K.
\end{cases}
\end{equation}
It is easy to see that such $\partial_s (T_K V(s)) = T_K \dot V(s)$ is the only solution the differential equation 
\begin{equation}
\begin{cases}
\partial_x^2 Y(x) + \alpha Y(x) = 0 & \text{ if }|x|>K, \\
Y(-K) = \dot V(-K)\\
Y(K) = \dot V(K)
\end{cases}
\end{equation}
with periodic boundary conditions. Therefore, by a simple variational argument, $\underline V = T_KV$ minimises the expression 
$$ \int  \norm{\dot {\underline V}(s)}_{\dot H^1}^2  +\alpha \norm{\dot {\underline V}(s)}_{L^2}^2 ds $$
in the family of functions $\underline V$ which are equal to $V$ on $[-K,K]$. Let $\rho \in C^\infty_c$ be a function such that $\rho \equiv 1$ on $[-L^\theta/8,L^\theta/8]$, and $\supp(\rho) \subseteq [-L^\theta/4,L^\theta/4]$. Let $V_n$ be a minimising sequence for \eqref{eqn: Fscale2}, and let 
$W_n:= \rho T_KV_n.$ We have that 
\begin{align*}
& \eqref{eqn: Fscale3}\\
\ge &\ \E\Big[ F(\ph(1) + W_n(1)) - \frac12 \int \norm{\dot W_n(s)}_{\dot H^1}^2  +\alpha \norm{\dot W_n(s)}_{L^2}^2 ds\Big] \\
= &\ \E\Big[ F(\ph(1) + V_n(1)) - \frac12 \int   \norm{\rho T_K\dot V_n(s)}_{\dot H^1}^2  +\alpha \norm{\rho T_K\dot V_n(s)}_{L^2}^2 ds\Big] \\
\ge &\ \E\Big[ F(\ph(1) + V_n(1)) - \frac{1+ 2Ce^{-\sqrt{\alpha} (\frac{L^\theta}8-K)}}2 \\
&\ \phantom{\E\Big[ ]}\times \left( \int  \norm{ T_K\partial_x \dot V_n(s)\1_{[-L^\theta/4, L^\theta/4]} }_{L^2}^2  +\alpha \norm{ T_K\dot V_n(s)\1_{[-L^\theta/4, L^\theta/4]} }_{L^2}^2 ds\right)\Big] \\
\ge &\ \E\Big[ F(\ph(1) + V_n(1)) - \frac{1+ 2Ce^{-\sqrt{\alpha} (\frac{L^\theta}8-K)}}2 \\
&\phantom{\E\Big[]} \times  \iint \left( |\partial_x \dot V_n(s)|^2  +\alpha |\dot V_n(s)|^2\right) \1_{[-L^\theta/4, L^\theta/4]} ds\Big]\\
\ge &\ \E\Big[ F(\ph(1) + V_n(1)) -    \frac12\iint \left(|\partial_x \dot V_n(s)|^2  +\alpha |\dot V_n(s)|^2\right) \1_{[-L^\theta/4, L^\theta/4]} ds\Big]\\
&- Ce^{-\sqrt{\alpha} (\frac{L^\theta}8-K)} \E\Big[\iint \left(\frac 12|\partial_x \dot V_n(s)|^2  +\alpha |\dot V_n(s)|^2\right) \1_{[-L^\theta/4, L^\theta/4]} ds\Big].
\end{align*}
Therefore, by taking limits in $n$, we obtain 
\begin{equation} \label{eqn: Fscale4}
\eqref{eqn: Fscale3} \ge \eqref{eqn: Fscale2} - Ce^{-\sqrt{\alpha} (\frac{L^\theta}8-K)}  \liminf_n   \E\Big[\iint \left(|\partial_x \dot V_n(s)|^2  +\alpha |\dot V_n(s)|^2\right) \1_{[-L^\theta/4, L^\theta/4]} ds\Big].\end{equation}
Moreover, from the trivial estimate $\eqref{eqn: Fscale2} \ge -\norm{F}_{L^\infty}$ (which we obtain by evaluating in $V=0$), we have that eventually
\begin{align*}
\norm{F}_{L^\infty}& \ge  \E\Big[- F(\ph(1) + V_n(1))\Big] +\E \Big[ \frac12 \iint \left(|\partial_x \dot V_n(s)|^2  +\alpha |\dot V_n(s)|^2\right) \1_{[-L^\theta/4, L^\theta/4]} ds\Big] \\
&\ge - \norm{F}_{L^\infty} +\frac12 \E \Big[  \iint \left(|\partial_x \dot V_n(s)|^2  +\alpha |\dot V_n(s)|^2\right) \1_{[-L^\theta/4, L^\theta/4]} ds\Big],
\end{align*}
so 
$$\E \Big[ \int  \norm{\dot V_n(s)}_{\dot H^1}^2  +\alpha \norm{\dot V_n(s)}_{L^2}^2 ds\Big] \le 4 \norm{F}_{L^\infty}. $$
Plugging this into \eqref{eqn: Fscale4}, we obtain 
$$\eqref{eqn: Fscale3} \le \eqref{eqn: Fscale2} + o(1).
 $$
\end{proof}
The following Lemma combines the previous two Lemmas to obtain the "almost independence" of functionals $F: \D'(K) \to \R$  from $F_\theta(u)$.
\begin{lemma} \label{lemma: small_scale2b}
Let $\gamma \ge \frac p2-1$, let $0 < \theta < \gamma \wedge 1$, and let $\NN > \frac{1}{2\sqrt\alpha}$. Let $M_\theta$ be as in \eqref{Ntheta}. 
Then, for every $F: \D'(K) \to \R$ Borel and bounded, we have  
\begin{equation} \label{eqn: truncated_limit}
\begin{aligned}
&\lim_{L \to \infty} \frac{ \E\Big[ \exp\Big(F(\ph) + \frac \beta {p L^\gamma} \int |\ph|^p \1_{[-L^\theta/2,L^\theta/2]^c}\Big) \1_{\{M_\theta(\ph) \le \NN L\}}\Big]}
{\E\Big[ \exp\Big( \frac \beta {p L^\gamma} \int |\ph|^p \1_{[-L^\theta/2,L^\theta/2]^c}\Big) \1_{\{M_\theta(\ph) \le \NN L\}}\Big]}\\
=& \lim_{L\to \infty} \E\Big[ \exp(F(\ph))\Big].
\end{aligned}
\end{equation}
\end{lemma}
\begin{proof}
Recall that $\P(\{M_\theta(\ph) \le \NN L\}) \to 1$ as $L \to \infty$, since $\NN \ge \frac 1 {2\sqrt\alpha}$. Moreover, 
$$\exp\Big(\frac \beta {p L^\gamma} \int |\ph|^p \1_{[-L^\theta/2,L^\theta/2]^c}\Big)  \ge 1,$$
and similarly
$$\exp\Big(F(\ph) + \frac \beta {p L^\gamma} \int |\ph|^p \1_{[-L^\theta/2,L^\theta/2]^c}\Big)  \ge \exp(-\norm{F}_{L^\infty}).$$
Therefore,
\begin{align*}
&\E\Big[ \exp\Big(\frac \beta {p L^\gamma} \int |\ph|^p \1_{[-L^\theta/2,L^\theta/2]^c}\Big) \1_{\{M_\theta(\ph) \le \NN L\}}\Big] \\
=&\ \E\Big[ \exp\Big(\1_{\{M_\theta(\ph) \le \NN L\}} \frac \beta {p L^\gamma} \int |\ph|^p \1_{[-L^\theta/2,L^\theta/2]^c}\Big) \Big](1+o(1))
\end{align*}
and similarly
\begin{align*}
&\E\Big[ \exp\Big(F(\ph) + \frac \beta {p L^\gamma} \int |\ph|^p \1_{[-L^\theta/2,L^\theta/2]^c}\Big) \1_{\{M_\theta(\ph) \le \NN L\}}\Big] \\
=&\ \E\Big[ \exp\Big(F(\ph) + \1_{\{M_\theta(\ph) \le \NN L\}} \frac \beta {p L^\gamma} \int |\ph|^p \1_{[-L^\theta/2,L^\theta/2]^c}\Big) \Big](1+o(1))
\end{align*}
Therefore, \eqref{eqn: truncated_limit} is equivalent to 
\begin{equation} \label{eqn: limit1}
\begin{aligned}
&\lim_{L \to \infty} \frac{ \E\Big[ \exp\Big(F(\ph) + \1_{\{M_\theta(\ph) \le \NN L\}} \frac \beta {p L^\gamma} \int |\ph|^p \1_{[-L^\theta/2,L^\theta/2]^c}\Big)\Big]}
{\E\Big[ \exp\Big( \1_{\{M_\theta(\ph) \le \NN L\}} \frac \beta {p L^\gamma} \int |\ph|^p \1_{[-L^\theta/2,L^\theta/2]^c}\Big)\Big]}\\
=& \lim_{L\to \infty} \E\Big[ \exp(F(\ph))\Big].
\end{aligned}
\end{equation}
By the Bou\'e-Dupuis formula \eqref{BD}, this is equivalent to 
\begin{align}
&\sup_{V \in \mathbb H^1_a} \E\Big[F(\ph(1) + V(1)) \notag\\
&\phantom{\inf_{V \in \mathbb H^1_a} \E\Big[]}+\1_{ \{ M_\theta(\ph(1)+V(1)) \leq \NN L  \}}\frac{\beta}{pL^\gamma} \int \left(|\ph(1)+V(1)|^p \1_{[-L^\theta/2,L^\theta/2]^c} \right) \notag \\
&\phantom{\inf_{V \in \mathbb H^1_a} \E\Big[]} -  \int \frac12 \norm{\dot V(s)}_{\dot H^1}^2  +\alpha \norm{\dot V(s)}_{L^2}^2 ds \Big] \label{eqn: sum} \\
=&\ \sup_{V \in \mathbb H^1_a} \E\Big[ F(\ph(1) + V(1)) - \frac12 \int  \norm{\dot V(s)}_{\dot H^1}^2  +\alpha \norm{\dot V(s)}_{L^2}^2 ds\Big] \label{eqn: sum1}\\
&+ \sup_{V \in \mathbb H^1_a} \E\Big[\1_{ \{ M_\theta(\ph(1)+V(1)) \leq \NN L  \}}\frac{\beta}{pL^\gamma} \int \left(|\ph(1)+V(1)|^p \1_{[-L^\theta/2,L^\theta/2]^c} \right) \notag \\
&\phantom{+\inf_{V \in \mathbb H^1_a} \E\Big[]} -  \frac12\int  \norm{\dot V(s)}_{\dot H^1}^2  +\alpha \norm{\dot V(s)}_{L^2}^2 ds \Big] +o(1)\label{eqn: sum2}
\end{align}
By estimating the $\sup$ in \eqref{eqn: sum} using functions $V$ in the form $V = V_1 + V_2$, where $\supp(V_1) \subseteq [-L^\theta/4,L^\theta/4]$ and $\supp(V_2) \subseteq (-L^\theta/4,L^\theta/4)^c$, we obtain
\begin{align*}
&\eqref{eqn: sum} \\
\ge&\ \sup_{V_1 \in \mathbb H^1_a, \atop \supp(V_1) \subseteq [-L^\theta/4,L^\theta/4]} \sup_{V_2 \in \mathbb H^1_a, \atop \supp(V_2) \subseteq (-L^\theta/4,L^\theta/4)^c} \\
&\Bigg\{\E\Big[ F(\ph(1) + V_1(1)) - \frac12\int  \norm{\dot V_1(s)}_{\dot H^1}^2  +\alpha \norm{\dot V_1(s)}_{L^2}^2 ds\Big] \label{eqn: sum1}\\
&+  \E\Big[\1_{ \{ M_\theta(\ph(1)+V_2(1)) \leq \NN L  \}}\frac{\beta}{pL^\gamma} \int \left(|\ph(1)+V_2(1)|^p \1_{[-L^\theta/2,L^\theta/2]^c} \right) \notag \\
&\phantom{+\E\Big[]} -  \frac12\int \norm{\dot V_2(s)}_{\dot H^1}^2  +\alpha \norm{\dot V_2(s)}_{L^2}^2 ds \Big] \Bigg\}\\
=& \eqref{eqn: Fscale3} + \eqref{eqn: scale3} \\
=& \eqref{eqn: sum1} + \eqref{eqn: sum2} + o(1),
\end{align*}
where we used Lemma \ref{lemma: large_scale} and Lemma \ref{lemma: small_scale2} for the last equality.
Moreover,
\begin{align*}
&\eqref{eqn: sum} \\
= &\ \sup_{V \in \mathbb H^1_a} \E\Big[F(\ph(1) + V(1)) \\
&\phantom{\inf_{V \in \mathbb H^1_a} \E\Big[]}- \frac12\iint \left(|\partial_x \dot V(s)|^2 + \alpha|\dot V(s)|^2\right)\1_{[-L^\theta/4,L^\theta/4]} ds dx\\
&\phantom{\inf_{V \in \mathbb H^1_a} \E\Big[]}+\1_{ \{ M_\theta(\ph(1)+V(1)) \leq \NN L  \}}\frac{\beta}{pL^\gamma} \int \left(|\ph(1)+V(1)|^p \1_{[-L^\theta/2,L^\theta/2]^c} \right)  \\
&\phantom{\inf_{V \in \mathbb H^1_a} \E\Big[]}- \frac12 \iint \left(|\partial_x \dot V(s)|^2 + \alpha|\dot V(s)|^2\right)\1_{(-L^\theta/2,L^\theta/2)^c} ds dx\Big] \\
\le &\ \sup_{V \in \mathbb H^1_a} \E\Big[F(\ph(1) + V(1)) \\
&\phantom{\inf_{V \in \mathbb H^1_a} \E\Big[]}- \frac12 \iint \left( |\partial_x \dot V(s)|^2 + \alpha|\dot V(s)|^2\right)\1_{[-L^\theta/4,L^\theta/4]} ds dx\Big]\\
&\ + \sup_{V \in \mathbb H^1_a} \E\Big[\1_{ \{ M_\theta(\ph(1)+V(1)) \leq \NN L  \}}\frac{\beta}{pL^\gamma} \int \left(|\ph(1)+V(1)|^p \1_{[-L^\theta/2,L^\theta/2]^c} \right)  \\
&\phantom{\inf_{V \in \mathbb H^1_a} \E\Big[]}- \frac12 \iint \left(|\partial_x \dot V(s)|^2 + \alpha|\dot V(s)|^2\right)\1_{(-L^\theta/2,L^\theta/2)^c} ds dx\Big] \\ 
=& \eqref{eqn: Fscale2} + \eqref{eqn: scale2} \\
=& \eqref{eqn: sum1} + \eqref{eqn: sum2} + o(1),
\end{align*}
where we used again Lemmas \ref{lemma: large_scale} and \ref{lemma: small_scale2} for the last equality.
\end{proof}

The following crucial Lemma states that removing an interval of size $L^\theta \ll L$ from the domain does not significantly alter the value of the partition function. Its proof relies on translation invariance. 

\begin{lemma} \label{lemma: cutoff_removal}
Let $\gamma \ge \frac p2-1$, let $0<\theta<\gamma\wedge 1$, and let $\NN > \frac 1 {2\sqrt \alpha}$. Let $\beta < \beta_0$ (defined in Lemma \ref{lemma:small_scale1}).
Then for any bounded Borel functional $F$ 
%, and let $\eta < \NN - \frac 1 {2\sqrt \alpha}$. 
\begin{equation} \label{eqn: cutoff_removal}
\lim_{L \to \infty} \frac{\E\Big[ \exp\Big( F(\ph) + \frac \beta {p L^\gamma} \int |\ph|^p \1_{[-L^\theta/2,L^\theta/2]^c}\Big) \1_{\{M(\ph) \le \NN L\}} \Big]}
{\E\Big[ \exp\Big( F( \ph) + \frac \beta {p L^\gamma} \int |\ph|^p \Big) \1_{\{M(\ph) \le \NN L\}} \Big]} = 1.
\end{equation}
\end{lemma}
\begin{proof}
We claim that it is enough to prove the statement for $F=0$. Indeed, desired statement can be rewritten as 
\begin{align*}
\Big| \frac{\int f_L(\ph) \exp(F(\ph)) \nu_L (d\ph) }{\int \exp(F(\ph)) \nu_L (d\ph)} - 1 \Big| \to 0,
\end{align*}
for $f_L(\ph) = \exp ( - \frac \beta {p L^\gamma} \int |\ph|^p \1_{[-L^\theta/2,L^\theta/2]})$ and $\nu_L$ is the probability measure proportional to $  \exp(  \frac \beta {p L^\gamma} \int |\ph|^p ) \1_{\{M(\ph) \le \NN L\}} \mu_L$.
We note that $f_L \leq 1$ and therefore 
\begin{align*}
&\Big| \frac{\int f_L(\ph) \exp(F(\ph)) \nu_L (d\ph) }{\int \exp(F(\ph)) \nu_L (d\ph)} - 1 \Big|  \\
&=  \frac{\int  (1 - f_L(\ph)) \exp(F(\ph)) \nu_L (d\ph) }{\int \exp(F(\ph)) \nu_L (d\ph)}  \\
& \leq  \| (1 - f_L(\ph)) \|_{L^1(\nu_L)} \| \exp (2 \| F \|_{L^\infty} )\\
&= \Big|  \int f_L(\ph) \nu_L(d\ph) - 1 \Big| \| \exp (2 \| F \|_{L^\infty} ),
\end{align*}
which goes to zero if \eqref{eqn: cutoff_removal} holds for $F = 0$.

We clearly have that 
\begin{align*}
&~\E\Big[ \exp\Big(\frac \beta {p L^\gamma} \int |\ph|^p \1_{[-L^\theta/2,L^\theta/2]^c}\Big) \1_{\{M(\ph) \le \NN L\}} \Big] \\
\le &~{\E\Big[ \exp\Big(\frac \beta {p L^\gamma} \int |\ph|^p \Big) \1_{\{M(\ph) \le \NN L\}} \Big]},
\end{align*}
so it is enough to show that 
\begin{equation}\label{cr0}
\liminf_{L \to \infty} \frac{\E\Big[ \exp\Big(\frac \beta {p L^\gamma} \int |\ph|^p \1_{[-L^\theta/2,L^\theta/2]^c}\Big) \1_{\{M(\ph) \le \NN L\}} \Big]}
{\E\Big[ \exp\Big(\frac \beta {p L^\gamma} \int |\ph|^p \Big) \1_{\{M(\ph) \le \NN L\}} \Big]} \ge 1.
\end{equation}
To show this, we exploit translation invariance of the law $\mu$ of $\ph$. In particular, we have that for every $y \in [-L/2, L/2]$, 
\begin{align*}
&\E\Big[ \exp\Big(\frac \beta {p L^\gamma} \int |\ph|^p \1_{[-L^\theta/2,L^\theta/2]^c}\Big) \1_{\{M(\ph) \le \NN L\}}\Big] \\
&= \E\Big[ \exp\Big(\frac \beta {p L^\gamma} \int |\ph|^p \1_{[y-L^\theta/2,y+L^\theta/2]^c}\Big) \1_{\{M(\ph) \le \NN L\}}\Big], 
\end{align*}
where the interval $[y-L^\theta/2,y+L^\theta/2]$ and its complementary are understood modulo $L$. Therefore, by Jensen's inequality applied to the exponential, 
\begin{align*}
&\E\Big[ \exp\Big(\frac \beta {p L^\gamma} \int |\ph|^p \1_{[-L^\theta/2,L^\theta/2]^c}\Big) \1_{\{M(\ph) \le \NN L\}}\Big] \\
=&\ \frac1L \int_{-L/2}^{L/2} \E\Big[ \exp\Big(\frac \beta {p L^\gamma} \int |\ph|^p \1_{[y-L^\theta/2,y+L^\theta/2]^c}\Big) \1_{\{M(\ph) \le \NN L\}}\Big] dy \\ 
\ge&\ \E\Big[ \exp\Big(\frac \beta {p L^\gamma} \int |\ph|^p \frac 1 L \int_{-L/2}^{L/2} \1_{[y-L^\theta/2,y+L^\theta/2]^c}dy \Big) \1_{\{M(\ph) \le \NN L\}}\Big] \\
=&\ \E\Big[ \exp\Big(\frac {(1-L^{\theta - 1})\beta} {p L^\gamma} \int |\ph|^p\Big) \1_{\{M(\ph) \le \NN L\}}\Big]. \label{cr1}\numberthis
\end{align*}
By H\"older, we have that 
\begin{align*}
& \E\Big[ \exp\Big(\frac {\beta} {p L^\gamma} \int |\ph|^p\Big) \1_{\{M(\ph) \le \NN L\}}\Big] \\
\le&\ \eqref{cr1}^{\frac{1}{1-L^{\theta-1}}} \E\Big[ \exp\Big(\frac {\beta} {p L^\gamma} \int |\ph|^p\Big) \1_{\{M(\ph) \le \NN L\}}\Big]^{L^{\theta-1}}.
\end{align*}
Therefore, by Lemma \ref{lemma: theta_removal}, we obtain that for some constant $C >0$,
\begin{align*}
\eqref{cr1} &\ge \E\Big[ \exp\Big(\frac {\beta} {p L^\gamma} \int |\ph|^p\Big) \1_{\{M(\ph) \le \NN L\}}\Big]^{(1-L^{\ta-1})^2} \\
&\ge \E\Big[ \exp\Big(\frac {\beta} {p L^\gamma} \int |\ph|^p\Big) \1_{\{M(\ph) \le \NN L\}}\Big] \big( \exp(CL^{1-\gamma}) \big)^{-2L^{\ta-1}} \\
&\ge \E\Big[ \exp\Big(\frac {\beta} {p L^\gamma} \int |\ph|^p\Big) \1_{\{M(\ph) \le \NN L\}}\Big]  \exp(-2CL^{\ta-\gamma}) \\
&= \E\Big[ \exp\Big(\frac {\beta} {p L^\gamma} \int |\ph|^p\Big) \1_{\{M(\ph) \le \NN L\}}\Big]  (1 + o(1)) 
\end{align*}
as $L\to \infty$. From this, \eqref{cr0} follows immediately.
\end{proof}

The following two technical Lemmas deal with stability with respect to changing the mass constraint. The first of the two 
 states that changing the mass constraint by a small multiple of $L$ only changes the partition function by $o(1)$.
\begin{lemma} \label{lemma: eps_removal}
Let $\gamma \ge \frac p2-1$, let $0<\theta<\gamma$, and let $\NN > \frac 1 {2\sqrt \alpha}$. Let $\eta < \NN - \frac 1 {2\sqrt\alpha}$, and if $\gamma = \frac p2 -1$, for $\beta_0$ as in Lemma \ref{lemma:small_scale1}, let $\beta < \beta_0$. Then, as $L \to \infty$, 
\begin{equation} \label{eqn: eps_removal}
\begin{aligned}
&~{\E\Big[ \exp\Big( \frac \beta {p L^\gamma} \int |\ph|^p\Big) \1_{\{M(\ph) \le \NN L\}}\Big]} \\
=&~{\E\Big[ \exp\Big( \frac \beta {p L^\gamma} \int |\ph|^p\Big) \1_{\{M(\ph) \le (\NN-\eta) L\}}\Big]} +o(1)
\end{aligned}
\end{equation}
\end{lemma}
\begin{proof}
We have that 
\begin{align*}
&~\E\Big[ \exp\Big( \frac \beta {p L^\gamma} \int |\ph|^p) \1_{\{M(\ph) \le \NN L\}}\Big] \\
= &~ \E\Big[ \exp\Big( \frac \beta {p L^\gamma} \int |\ph|^p\Big) \1_{\{M(\ph) \le (\NN-\eta) L\}}\Big] \\
& + \E\Big[ \exp\Big( \frac \beta {p L^\gamma} \int |\ph|^p \Big) \1_{\{(\NN-\eta) L < M(\ph) \le \NN L \}}\Big],
\end{align*}
so we just need to prove that 
\begin{equation*}
\E\Big[ \exp\Big( \frac \beta {p L^\gamma} \int |\ph|^p \Big) \1_{\{(\NN-\eta) L < M(\ph) \le \NN L \}}\Big] = o(1)
\end{equation*}
as $L \to \infty$. We have that 
\begin{align*}
&~\E\Big[ \exp\Big( \frac \beta {p L^\gamma} \int |\ph|^p \Big) \1_{\{(\NN-\eta) L < M(\ph) \le \NN L \}}\Big] \\
= &~\E\Big[ \exp\Big(\1_{\{M(\ph) \le \NN L \}} \frac \beta {p L^\gamma} \int |\ph|^p \Big) \1_{\{M(\ph) > (\NN-\eta) L \}}\Big] +o(1) \\
\le &~ \E\Big[ \exp\Big(\1_{\{M(\ph) \le \NN L \}} \frac \beta {p L^\gamma} \int |\ph|^p - L\1_{\{M(\ph) \le (\NN-\eta) L \}}\Big)\Big] +o(1).
\end{align*}
So, from the Bou\'e-Dupuis formula \eqref{BD}, it is enough to prove that 
\begin{equation} \label{B1}
\begin{aligned}
&\sup_{V \in \mathbb H^1} \E\Big[ \1_{ \{ M(\ph(1)+V(1)) \leq \NN L  \}}\frac{\beta}{pL^\gamma} \int_{-L^\theta/2}^{L^\theta/2} |\ph(1) + V(1)|^p \\
&\phantom{\inf_{V \in \mathbb H^1_a} \E\Big[]}- \frac 12 \norm{V(s)}_{\dot H^1}^2 ds - \frac\alpha2  \norm{ V(s)}_{L^2}^2 ds\\
&\phantom{\inf_{V \in \mathbb H^1_a} \E\Big[]}+L\1_{\{M(\ph(1)+V(1)) \le (\NN-\eta) L\}} \Big] \to +\infty
\end{aligned}
\end{equation}
as $L \to \infty$.
We follow closely the proof of Lemma \ref{lemma:small_scale1}, define $\ph_\eps$ exactly as in \eqref{phi_eps}, and perform the change of variable $V = -\ph_\eps + W$. Going through the same computations as in Lemma \ref{lemma:small_scale1}, we obtain the analogous of \eqref{A1.1} 
\begin{align*}
 \eqref{B1}
\le &~ \E\left[C_\eps \frac{\beta}{pL^\gamma} \int |(\ph-\ph_\eps)(1)|^p + C_\eps \left(\frac 12  \norm{\ph_\eps}_{\dot H^1}^2 + \frac \alpha 2 \norm{ \ph_\eps}_{L^2}^2\right) \right] \\
& -(1-\eps) \inf_{W\in H^1} -  \1_{ \{ \norm{W}_{L^2}^2 \le L\sigma \}}\bigg(  \frac{(1+3\eps)\beta}{pL^\gamma} \int |W|^p \\
& \phantom{+ \inf_{W\in H^1}}+ \big(\frac 12  \norm{W}_{\dot H^1}^2 + \frac \alpha 2 \norm{ W}_{L^2}^2\big)\bigg)
 + L\1_{\left\{\|W\|_{L^2}^2 \le \left(\sqrt{\NN-\eta} - \sqrt{\frac{1+\eps}{2\sqrt\alpha}}\right)^2 L\right\}}
\end{align*}
Proceeding as in \eqref{A2}, we obtain 
\begin{equation}\label{B2}
\begin{aligned}
&~ \E\left[C_\eps \frac{\beta}{pL^\gamma} \int |(\ph-\ph_\eps)(1)|^p  + C_\eps \left(\frac 12  \norm{\ph_\eps}_{\dot H^1}^2 + \frac \alpha 2 \norm{ \ph_\eps}_{L^2}^2\right) \right] \\
\les &~L^{1 - \gamma} + \exp(-c \eps L^\frac 12) \eps^{-24} L^{13} \les L^{1-\gamma},
\end{aligned}
\end{equation}
and from \eqref{A3}, we have that 
\begin{equation} \label{B3}
\begin{aligned}
& \inf_{W\in H^1} - \1_{ \{ \norm{W}_{L^2}^2 \le L\sigma \}} (1+3\eps) \frac{\beta}{pL^\gamma} \int |W|^p \\
& \phantom{ \inf_{W\in H^1}}+ \big(\frac 12  \norm{W}_{\dot H^1}^2 + \frac \alpha 2 \norm{ W}_{L^2}^2\big)
 + L\1_{\left\{\|W\|_{L^2}^2 \le \left(\sqrt{\NN-\eta} - \sqrt{\frac{1+\eps}{2\sqrt\alpha}}\right)^2 L\right\}}\\
 \ge& \inf_{W\in H^1} \eps\norm{W}_{\dot H^1}^2 + \delta(\eps)  \norm{ W}_{L^2}^2 + L\1_{\left\{\|W\|_{L^2}^2 \le \left(\sqrt{\NN-\eta} - \sqrt{\frac{1+\eps}{2\sqrt\alpha}}\right)^2 L\right\}}
\gtrsim L,
\end{aligned}
\end{equation}
where 
$$\delta(\eps)  > 0$$
for $\eps$ small enough and $L$ big enough. Therefore, there exist two constants $C_1, C_2$ such that 
\begin{align*}
\eqref{B1} \le C_1 L^{1- \gamma} - C_2 L \to -\infty
\end{align*}
as $L \to \infty$.
\end{proof}

The following Lemma relies on the previous Lemma~\ref{lemma: theta_removal}. Here we show a one-sided inequality when in the mass cut-off we can replace $M$ by $M_\theta$, at the cost of changing the constraint by a small multiple of $L$.

\begin{lemma} \label{lemma: theta_removal2}
Let $\gamma \ge \frac p2-1$, let $0<\theta<\gamma$, and let $\NN > \frac 1 {2\sqrt\alpha}$. Let $M_\theta$ be as in \eqref{Ntheta}, let $\beta < \beta_0$ (defined in Lemma \ref{lemma:small_scale1}), and let $\eta < \NN - \frac 1 {2\sqrt\alpha}$. Then, for every $F:\D'(K) \to \R$ Borel and bounded, we have  
\begin{equation}
\begin{aligned}
&~\E\Big[ \exp\Big(F(\ph) + \frac \beta {p L^\gamma} \int |\ph|^p \1_{[-L^\theta/2,L^\theta/2]^c}\Big) \1_{\{M(\ph) \le \NN L\}} \Big] \\
\ge&~\E\Big[ \exp\Big(F(\ph) +  \frac \beta {p L^\gamma} \int |\ph|^p \1_{[-L^\theta/2,L^\theta/2]^c}\Big)\1_{\{M_\theta(\ph) \le (\NN-\eta) L\}}\Big] + o(1)
\end{aligned}
\end{equation}
\end{lemma}

\begin{proof}
We have that 
$$\1_{\{M(\ph) \le \NN L\}} \ge \1_{\{M_\theta(\ph) \le (\NN-\eta) L\}} - \1_{\left\{\int_{-L^\theta/2}^{L^\theta/2} |\ph|^2 > \eta L\right\}} \1_{\{M_\theta(\ph) \le (\NN-\eta) L\}}.$$
Therefore, we just need to prove that 
$$\E\Big[ \exp\Big( F(\ph) +  \frac \beta {p L^\gamma} \int |\ph|^p \1_{[-L^\theta/2,L^\theta/2]^c}\Big) \1_{\left\{\int_{-L^\theta/2}^{L^\theta/2} |\ph|^2 > \eta L\right\}} \1_{\{M_\theta(\ph) \le (\NN-\eta) L\}}\Big] = o(1).$$
By H\"older, for every $1<q<+\infty$, we have that 
\begin{align*}
&~\E\Big[ \exp\Big(F(\ph) +  \frac \beta {p L^\gamma} \int |\ph|^p \1_{[-L^\theta/2,L^\theta/2]^c}\Big) \1_{\left\{\int_{-L^\theta/2}^{L^\theta/2} |\ph|^2 > \eta L\right\}} \1_{\{M_\theta(\ph) \le (\NN-\eta) L\}}\Big] \\
\les&~\E\Big[ \exp\Big( \frac {q\beta} {p L^\gamma} \int |\ph|^p \1_{[-L^\theta/2,L^\theta/2]^c}\Big) \1_{\{M_\theta(\ph) \le (\NN-\eta) L\}}\Big]^{\frac 1 q}\\
&\times \P\left({\left\{\int_{-L^\theta/2}^{L^\theta/2} |\ph|^2 > \eta L\right\}}\right)^{1-\frac1q}.  
\end{align*}
By Lemma \ref{lemma: theta_removal}, we have that, as long as $q\beta < \beta_0$, 
$$\E\Big[ \frac {q\beta} {p L^\gamma} \int |\ph|^p \1_{[-L^\theta/2,L^\theta/2]^c}\Big) \1_{\{M_\theta(\ph) \le (\NN-\eta) L\}}\Big]^{\frac 1 q} \le \exp\Big(\frac Cq L^{1-\gamma}\Big).$$
Moreover, we have the large deviation estimate
$$
\P\left({\left\{\int_{-L^\theta/2}^{L^\theta/2} |\ph|^2 > \eta L\right\}}\right)\le \exp\Big(-c \frac{\eta  L}{L^\theta}\Big). 
 $$
Hence we have 
\begin{align*}
&~\E\Big[ \exp\Big(F(\ph) +  \frac \beta {p L^\gamma} \int |\ph|^p \1_{[-L^\theta/2,L^\theta/2]^c}\Big) \1_{\left\{\int_{-L^\theta/2}^{L^\theta/2} |\ph|^2 > \eta L\right\}} \1_{\{M_\theta(\ph) \le (\NN-\eta) L\}}\Big] \\
\les&~ \exp\Big(\frac Cq L^{1-\gamma}\Big) \exp\Big(-c\left(1-\frac 1 q\right)  \frac{\eta  L}{L^\theta}\Big) = o(1)
\end{align*}
as $L \to \infty$.
\end{proof}

\begin{proof}[Proof of Theorem \ref{thm:3}]

We first show that for every $F \in \D'(K)$, Borel and bounded, we have 
\begin{equation} \label{eqn: truncated_limit2}
\begin{aligned}
&\lim_{L \to \infty} \frac{ \E\Big[ \exp\Big(-F(\ph) + \frac \beta {p L^\gamma} \int |\ph|^p \Big) \1_{\{M(\ph) \le \NN L\}}\Big]}
{\E\Big[ \exp\Big( \frac \beta {p L^\gamma} \int |\ph|^p \Big) \1_{\{M(\ph) \le \NN L\}}\Big]}\\
=& \lim_{L\to \infty} \E\Big[ \exp(-F(\ph))\Big].
\end{aligned}
\end{equation}
The case $\gamma > 1$ follows immediately from \eqref{eqn:small_scale} with $\P(\{M(\ph) \le \NN L \}) \to 1$, by choosing $\ta = 1$. Therefore, we can assume $\gamma \le 1$.
By Lemma \ref{lemma: cutoff_removal}, we have that 
\begin{align}
&\frac{ \E\Big[ \exp\Big(F(\ph) + \frac \beta {p L^\gamma} \int |\ph|^p \Big) \1_{\{M(\ph) \le \NN L\}}\Big]}
{\E\Big[ \exp\Big( \frac \beta {p L^\gamma} \int |\ph|^p \Big) \1_{\{M(\ph) \le \NN L\}}\Big]} \notag \\
= &~\frac{ \E\Big[ \exp\Big(F(\ph) + \frac \beta {p L^\gamma} \int |\ph|^p \1_{[-L^\theta/2,L^\theta/2]^c}\Big) \1_{\{M(\ph) \le \NN L\}}\Big]}
{\E\Big[ \exp\Big( \frac \beta {p L^\gamma} \int |\ph|^p \1_{[-L^\theta/2,L^\theta/2]^c}\Big) \1_{\{M(\ph) \le \NN L\}}\Big]} (1+o(1)) \label{eqn: theta_added}
\end{align}
By Lemmas \ref{lemma: eps_removal} and  \ref{lemma: theta_removal2}, we have 
\begin{align*}
&~\frac{ \E\Big[ \exp\Big(F(\ph) + \frac \beta {p L^\gamma} \int |\ph|^p \1_{[-L^\theta/2,L^\theta/2]^c}\Big) \1_{\{M_\theta(\ph) \le \NN L\}}\Big]}
{\E\Big[ \exp\Big( \frac \beta {p L^\gamma} \int |\ph|^p \1_{[-L^\theta/2,L^\theta/2]^c}\Big) \1_{\{M_\theta(\ph) \le \NN L\}}\Big]} \\
\ge&~\frac{ \E\Big[ \exp\Big(F(\ph) + \frac \beta {p L^\gamma} \int |\ph|^p \1_{[-L^\theta/2,L^\theta/2]^c}\Big) \1_{\{M_\theta(\ph) \le \NN L\}}\Big]}
{\E\Big[ \exp\Big( \frac \beta {p L^\gamma} \int |\ph|^p \1_{[-L^\theta/2,L^\theta/2]^c}\Big) \1_{\{M(\ph) \le (\NN+\eta) L\}}\Big]}(1+o(1))\\
\ge &~\frac{ \E\Big[ \exp\Big(F(\ph) + \frac \beta {p L^\gamma} \int |\ph|^p \1_{[-L^\theta/2,L^\theta/2]^c}\Big) \1_{\{M_\theta(\ph) \le \NN L\}}\Big]}
{\E\Big[ \exp\Big( \frac \beta {p L^\gamma} \int |\ph|^p \1_{[-L^\theta/2,L^\theta/2]^c}\Big) \1_{\{M(\ph) \le \NN L\}}\Big]}(1+o(1))\\
\ge ~&\eqref{eqn: theta_added}(1+o(1)) \\
\ge &~\frac{ \E\Big[ \exp\Big(F(\ph) + \frac \beta {p L^\gamma} \int |\ph|^p \1_{[-L^\theta/2,L^\theta/2]^c}\Big) \1_{\{M(\ph) \le (\NN+\eta) L\}}\Big]}
{\E\Big[ \exp\Big( \frac \beta {p L^\gamma} \int |\ph|^p \1_{[-L^\theta/2,L^\theta/2]^c}\Big) \1_{\{M(\ph) \le \NN L\}}\Big]} \\
\ge &~\frac{ \E\Big[ \exp\Big(F(\ph) + \frac \beta {p L^\gamma} \int |\ph|^p \1_{[-L^\theta/2,L^\theta/2]^c}\Big) \1_{\{M_\theta(\ph) \le \NN L\}}\Big]}
{\E\Big[ \exp\Big( \frac \beta {p L^\gamma} \int |\ph|^p \1_{[-L^\theta/2,L^\theta/2]^c}\Big) \1_{\{M_\theta(\ph) \le \NN L\}}\Big]} (1+o(1)).
\end{align*}
Therefore, \eqref{eqn: truncated_limit2} is equivalent to
\begin{equation*}
\begin{aligned}
&~\lim_{L\to \infty}\frac{ \E\Big[ \exp\Big(-F(\ph) + \frac \beta {p L^\gamma} \int |\ph|^p \1_{[-L^\theta/2,L^\theta/2]^c}\Big) \1_{\{M_\theta(\ph) \le (\NN L\}}\Big]}
{\E\Big[ \exp\Big( \frac \beta {p L^\gamma} \int |\ph|^p \1_{[-L^\theta/2,L^\theta/2]^c}\Big) \1_{\{M_\theta(\ph) \le \NN L\}}\Big]}\\
=& \lim_{L\to \infty} \E\Big[ \exp(-F(\ph))\Big],
\end{aligned}
\end{equation*}
which is \eqref{eqn: truncated_limit}.

We now move to showing the weak convergence of $\rho_L \to \mu_{OU}$ with respect to the compact-open topology of $C(\R)$. 
Recall that the sets $\{ u \in C(\R) : \norm{u}_{C^\frac14([-n,n])} \le C_n\, \forall n \in \N\}$ are compact in $C(\R)$, for every choice of the sequence $C_n$. Let $n\in \N$, $M>0$. let $F_n^M(\ph) = \min(\norm{\ph}_{C^\frac14([-n,n])},M)$, and let $F_n(\ph) = \norm{\ph}_{C^\frac14([-n,n])}$. By \eqref{eqn: truncated_limit}, we have that 
\begin{align*}
 \lim_{L \to \infty} \int \exp(F_n^M(\ph)) d\rho_L(\ph) &= \E_{OU}[\exp(F_n^M(\ph))] \\
 & \le \E_{OU}[\exp(F_n(\ph))] =: K_n < +\infty,
\end{align*}
where $\E_{OU}$ is the expectation of a functional with respect to the Ornstein-Uhlenbeck measure. Therefore, for every $N>0$, 
\begin{align*}
\limsup_{L \to \infty} \rho_L(\{\norm{\ph}_{C^\frac14([-n,n])} > N\}) &\le \limsup \frac{1}{\exp(N)} \int \exp(F_n^N(\ph)) d\rho_L(\ph) \\
&\le K_n \exp(-N) \to 0
\end{align*}
as $N \to \infty$. Therefore, for every $\eps > 0$ and for every $n\in\N$, there exists $C_n = C_n(\eps)$ big enough such that 
$$\rho_L(\{\norm{\ph}_{C^\frac14([-n,n])} > C_n\}) \le 2^{-n-1}\eps. $$
Therefore, 
$$\rho_L(\{ u \in C(\R) : \norm{u}_{C^\frac14([-n,n])} \le C_n\, \forall n \in \N\}) \ge 1 - \eps,$$
so the family $\rho_L$ is tight in $C(\R)$. Therefore, by Skorohod's theorem, up to subsequences, $\rho_L$ has a weak limit as $L \to \infty$. Hence, let $\rho$ be any weak limit of $\rho_L$ as $L \to \infty$. To conclude the proof, we just need to prove that $\rho$ corresponds to the Ornstein-Uhlenbeck measure.
By \eqref{eqn: truncated_limit}, $\rho$ must satisfy 
$$\int \exp(F) d\rho = \E_{OU}[\exp(F)] $$
for every $F:\D'(K) \to \R$ Borel and bounded. 
For $G:\D'(K) \to \R,$ we can write 
$$G = \exp( \log( G + \norm{G}_{L^\infty} + 1)) - \exp( \log(\norm{G}_{L^\infty} + 1)).$$
Therefore, we have that 
\begin{equation} \label{rho=OU}
\int G d\rho = \E_{OU}[G] 
\end{equation}
for every $G \in \D'(K)$. 
Given $x_1,\dotsc,x_d \in \R$, and $F \subseteq \R^d$ Borel, denote 
$$E_{(x_1,\dots,x_d),F} = \{ \ph |( \ph(x_1),\dotsc,\ph(x_d)) \in F\}. $$
By \eqref{rho=OU}, we have that $\rho(E_{(x_1,\dots,x_d),F}) = \E_{OU}[\1_{E_{(x_1,\dots,x_d),F}}]$. Since the family of sets $\{E_{(x_1,\dots,x_d),F}\}$
is a Dynkin system for the Borel sets of $C(\R)$, we deduce that $\rho$ must be equal to the Ornstein-Uhlenbeck measure. 
\end{proof}

\appendix

\section{The Buo\'e-Dupuis formula in infinite dimension}

We now recall the Bou\'e-Dupuis formula from~\cite{BD, Ust}. The particular version we are going to use is an immediate corollary of \cite[Theorem 3.2]{Zhang}. 

\begin{proposition}[Bou\'e-Dupuis variational formula]\label{LEM:BD}
	Let $(W,H,\mu)$ be an abstract Wiener space, and let $(\ph(t), t \in [0,1])$ be a cylindrical Brownian motion on $H$ with $\ph \in C([0,1], W)$. Let $F:W \to \R$
	be a bounded measurable function. Then 
	\begin{align}
	\log \E\Big[e^{F( \ph(1))}\Big]
	= \sup_{\theta \in \mathbb H_a}
	\E\bigg[ F(\ph(1) + I(\theta)(1)) - \frac{1}{2} \int_0^1 \| \theta(t) \|_{H}^2 dt \bigg], 
	\label{P3}
	\end{align}	
\noindent
where $\mathbb H_a$ denotes the  set of $H$ valued stochastic processes $\theta$ for which 
$$\E\bigg[\int_0^1 \| \theta(t) \|_{H}^2\bigg] < \infty$$ 
and which are  progressively measurable w.r.t.\ the filtration induced by $\ph$. Furthermore, 
$I(\theta)$ is  defined by 
\begin{align*}
I(\theta)(t) = \int_0^t \theta(t') dt'.
\end{align*}
\end{proposition}
Throughout the paper, we will make use of the following. 
\begin{proposition} \label{prop: BDinf}
	Let $(W,H,\mu)$ be an abstract Wiener space, and let $(\ph(t), t \in [0,1])$ be a cylindrical Brownian motion on $H$ with $\ph \in C([0,1], W)$.  Let $F:W \to \R$
	be a measurable function bounded from below. Suppose that 
	$$\sup_{V \in \mathbb H^1} \E\bigg[ F(\ph(1) + V(1)) - \frac{1}{2} \int_0^1 \| \dot V(t) \|_{H}^2 dt \bigg] < \infty,$$
where $\mathbb H^1$ denotes the space of $H$-valued stochastic processes $V$ such that $V(0) = 0$ and 
$$\E\bigg[ \int_0^1 \| \dot V(t) \|_{H}^2 dt \bigg] < \infty. $$
Then 
\begin{equation}
\begin {aligned}
	\log \E\Big[e^{F( \ph(1))}\Big]
	&= \sup_{V \in \H}
	\E\bigg[ F(\ph(1) + V(1)) - \frac{1}{2} \int_0^1 \| \dot V(t) \|_{H}^2 dt \bigg]  \\
	&\le \sup_{V \in \mathbb H^1}
	\E\bigg[ F(\ph(1) + V(1)) - \frac{1}{2} \int_0^1 \| \dot V(t) \|_{H}^2 dt \bigg], 
	\label{BD}
\end{aligned}
\end{equation}
where $\H$ denotes the space of processes in $\mathbb H^1$ which are progressively measurable.
\end{proposition}
\begin{proof}
Notice that the last $\le$ in \eqref{BD} follows directly from the fact that $\mathbb H^1 \supseteq \H$, so we just need to show the first equality.
For $M >0$, let $F_M = F \wedge M$. Then, by \eqref{P3}, by the change of variables $V = I(\theta)$, we have 
\begin{align*}
\log \E\Big[e^{F_M( \ph(1))}\Big] &= 
\sup_{\theta \in \mathbb H_a} \E\bigg[ F_M(\ph(1) + I(\theta)(1)) - \frac{1}{2} \int_0^1 \| \theta(t) \|_{H}^2 dt \bigg] \\
&= \sup_{V \in \H} \E\bigg[ F_M(\ph(1) + V(1)) - \frac{1}{2} \int_0^1 \| \dot V(t) \|_{H}^2 dt \bigg] \\
&\le \sup_{V \in \H} \E\bigg[ F(\ph(1) + V(1)) - \frac{1}{2} \int_0^1 \| \dot V(t) \|_{H}^2 dt \bigg] \\
&\le \sup_{V \in \mathbb H^1} \E\bigg[ F(\ph(1) + V(1)) - \frac{1}{2} \int_0^1 \| \dot V(t) \|_{H}^2 dt \bigg].
\end{align*}
Therefore, taking the limit as $M\uparrow \infty$, by monotone convergence we obtain that
\begin {align} \label{BDpf1}
	\log \E\Big[e^{F( \ph(1))}\Big]
	&\le \sup_{V \in \H}
	\E\bigg[ F(\ph(1) + V(1)) - \frac{1}{2} \int_0^1 \| \dot V(t) \|_{H}^2 dt \bigg] < \infty.
\end{align}
In order to show the reverse inequality, fix $\eps > 0$, and let $V_\eps \in \H$ be such that 
\begin{align*}
&\E\bigg[ F(\ph(1) + V_\eps(1)) - \frac{1}{2} \int_0^1 \| \dot V_\eps(t) \|_{H}^2 dt \bigg]\\
\ge &\ - \eps + \sup_{V \in \H}
	\E\bigg[ F(\ph(1) + V(1)) - \frac{1}{2} \int_0^1 \| \dot V(t) \|_{H}^2 dt \bigg] 
\end{align*}
By \eqref{BDpf1}, $\E[ F(\ph(1) + V_\eps(1))] < \infty$. Therefore, by monotone convergence, there exists $M > 0$ such that 
$$\E[ F_M(\ph(1) + V_\eps(1))] \ge -\eps + \E[ F(\ph(1) + V_\eps(1))].$$ 
With this choice of $M$, by \eqref{P3} we have
\begin{align*}
&\sup_{V \in \H}
	\E\bigg[ F(\ph(1) + V(1)) - \frac{1}{2} \int_0^1 \| \dot V(t) \|_{H}^2 dt \bigg] \\
\le&\ \eps + \E\bigg[ F(\ph(1) + V_\eps(1)) - \frac{1}{2} \int_0^1 \| \dot V_\eps(t) \|_{H}^2 dt \bigg] \\
\le&\ 2\eps +  \E\bigg[ F_M(\ph(1) + V_\eps(1)) - \frac{1}{2} \int_0^1 \| \dot V_\eps(t) \|_{H}^2 dt \bigg] \\
\le&\ 2\eps +  \sup_{\theta \in \mathbb H_a}\E\bigg[ F_M(\ph(1) + I(\theta)(1)) - \frac{1}{2} \int_0^1 \| \theta(t) \|_{H}^2 dt \bigg] \\
=&\ 2\eps + \log \E[ e^{F_M(\ph(1))}]\\
\le&\ 2\eps + \log \E[ e^{F(\ph(1))}].
\end{align*}
By taking limits as $\eps \downarrow 0$, we obtain 
$$ \log \E[ e^{F(\ph(1))}] \ge \sup_{V \in \H}
	\E\bigg[ F(\ph(1) + V(1)) - \frac{1}{2} \int_0^1 \| \dot V(t) \|_{H}^2 dt \bigg]. $$
Together with \eqref{BDpf1}, this completes the proof.
\end{proof}
\section{Gaussian free field and Ornstein-Uhlenbeck measure}
First, we recall the following hypercontractivity property of Gaussian measure, which is an immediate corollary of \cite[Theorem I.22]{Simon74}. 
\begin{proposition}
Let $n\in \N$, and let $\nu$ be a Gaussian measure on $\R^n$. Let $p:\R^n \to \R$ be a polynomial of degree $d$. 
Let 
$$\sigma_p^2 := \int |p(u)|^2 d\nu(u).$$
Then there exists a constant $c_d$ depending only on $d$ such that 
\begin{equation}\label{hypercontractivity}
\int \exp\Big(c_d \big|\sigma_p^{-1}p(u)\big|^\frac 2 d\Big) d \nu(u) \le 2.
\end{equation}
\end{proposition}
Let $\mu_L$ the Gaussian free field defined in Section \ref{Sec:GaussianMeasure}.
\begin{proposition}
Let $\ph$ be distributed according to $\mu_L$. Then for every $x \in [-L/2,L/2]$, 
\begin{equation} \label{ph_pointwise}
\E[|\ph(x)|^2 ] = \frac 1 L \sum_{k \in \Z} \frac{1}{\alpha + 4\pi^2\big(\frac k L\big)^2} \to \int_{-\infty}^\infty \frac{dx}{\alpha + 4\pi^2 x^2} = \frac{1}{2\sqrt{\alpha}}.
\end{equation}
as $L \to \infty$. Moreover, for any $s < \frac 12$, we have that  
\begin{equation} \label{ph_Hs}
\E[\|\ph\|_{H^s}^2] = \sum_{k \in \Z} \frac{1}{\Big(\alpha + 4\pi^2\big(\frac k L\big)^2\Big)^{1-s}} \les L.
\end{equation}
\end{proposition}
\begin{proof}
Recall that according to \eqref{random_fourier} one has
$$\ph(x) = \sum_{k\in\Z} \frac{g_k}{\sqrt{\alpha + 4\pi^2\big(\frac{k}{L}\big)^2}} \frac{e^{i\frac{2\pi kx}{L}}}{\sqrt L}, $$
where $g_k$ are i.i.d.\ complex valued standard Gaussian random variables. The equalities in \eqref{ph_pointwise} and \eqref{ph_Hs} follow by direct computation. Similarly, by Riemann integral approximation, it is easy to see that for $s<\frac12$, 
$$ \lim_{L\to\infty} \frac 1 L\sum_{k \in \Z} \frac{1}{\Big(\alpha + 4\pi^2\big(\frac k L\big)^2\Big)^{1-s}} = \int_{-\infty}^\infty \frac{dx}{\Big(\alpha + 4\pi^2x^2\Big)^{1-s}} < \infty, $$
from which we immediately obtain both the convergence in \eqref{ph_pointwise} and the estimate in \eqref{ph_Hs}.
\end{proof}
\begin{corollary}
Let $s < \frac 12$, and let $0 < p < \infty$. Then  
\begin{equation}
\E[\|\ph\|_{H^s}^p] \les_{s,p} L^\frac p2. \label{ph_Hsp}
\end{equation}
\end{corollary}
\begin{proof}
Using Minkowski's inequality first and then \eqref{ph_Hs} we get
\begin{align*}
\Big( \E[\|\ph\|_{H^s}^p]  \Big)^{\frac{2}{p}}&=\E\Big[ \Big(\sum_{k \in \Z} \frac{|g_k|^2}{\Big(\alpha + 4\pi^2\big(\frac k L\big)^2\Big)^{1-s}}\Big)^\frac{p}{2}\Big]^{\frac{2}{p}} \\
&\leq  \sum_{k \in \Z} \frac{ \E\big[|g_k|^p]^{\frac{2}{p}}}{\Big(\alpha + 4\pi^2\big(\frac k L\big)^2\Big)^{1-s}} \\
&\lesssim   \sum_{k \in \Z} \frac{ 1}{\Big(\alpha + 4\pi^2\big(\frac k L\big)^2\Big)^{1-s}}  \lesssim L.
%& \E\Big[ \sum_{k,h \in \Z} \frac{|g_k|^2|g_h|^2}{\Big(\alpha + 4\pi^2\big(\frac k L\big)^2\Big)^{1-s}\Big(\alpha + 4\pi^2\big(\frac h L\big)^2\Big)^{1-s}}\Big] \\
%&\les \Big(\sum_{k \in \Z} \frac{1}{\Big(\alpha + 4\pi^2\big(\frac k L\big)^2\Big)^{1-s}}\Big)^2 \\
%&\les L^2.  \label{Hs4} \numberthis
\end{align*}
%It is easy to check that if 
%\begin{equation}
%\ph_N :=  \sum_{k\in\Z, |k|\le N} \frac{g_k}{\sqrt{\alpha + 4\pi^2\big(\frac{k}{L}\big)^2}} \frac{e^{i\frac{2\pi kx}{L}}}{\sqrt L}, \label{phN}
%\end{equation}
%then $\|\ph_N\|_{H^s} \le \|\ph\|_{H^s}$. Therefore, \eqref{Hs4} holds for $\ph_N$ as well. Moreover, since $\mathrm{Law}(\ph_N)$ is supported on a finite dimensional subspace of $H^s$, we can apply \eqref{hypercontractivity} and obtain that 
%\begin{align*}
%\E[\|\ph_N\|_{H^s}^p] &\les_p \E[\|\ph\|_{H^s}^4]^{\frac p4} \E\Big[\exp\Big(c_2\E[\|\ph\|_{H^s}^4]^{-\frac12} \|\ph_N\|_{H^s}^2\Big)\Big] \\
%&\le \E[\|\ph\|_{H^s}^4]^{\frac p4} \E\Big[\exp\Big(c_2\E[\|\ph_N\|_{H^s}^4]^{-\frac12} \|\ph_N\|_{H^s}^2\Big)\Big] \\
%&\le 2 \E[\|\ph\|_{H^s}^4]^{\frac p4} \\
%&\les L^\frac p2.
%\end{align*}
%The analogous estimate for $\ph$ follows by Fatou.
\end{proof}

\begin{proposition}\label{BD0}
Let $I \subseteq [-L/2,L/2]$ be an interval, and let $M > 0$. Then there exists a constant $c = c(\alpha) > 0$ such that 
\begin{equation}\label{LDE}
\P\Big(\Big\{ \Big|\int_I |\ph(x)|^2 dx - \frac{|I|}{2\sqrt{\alpha}}\Big| > M\Big\}\Big) \les \exp\Big( -c \frac{M}{|I|^\frac12}\Big).
\end{equation}
\end{proposition}
\begin{proof}

By \eqref{hypercontractivity} and a simple approximation argument 
%involving $\ph_N$ as defined in \eqref{phN}, 
it is enough to show that 
\begin{equation*}
\E\Big|\int_I |\ph(x)|^2 dx - \frac{|I|}{2\sqrt{\alpha}}\Big|^2 \les |I|. 
\end{equation*}
We write 
\begin{align}
&\E\Big|\int_I |\ph(x)|^2 dx - \frac{|I|}{2\sqrt{\alpha}}\Big|^2  \notag \\
=&\ \E\Big|\int_I |\ph(x)|^2  - \E[|\ph(x)|^2 dx \Big|^2 \label{LDEmain} \\
&\ + \Big|\int_I \E |\ph(x)|^2  - \frac{1}{2\sqrt{\alpha}} dx \Big|^2 \label{LDEerror}.
\end{align}
From \eqref{ph_pointwise} it follows easily that 
$$\eqref{LDEerror} \les \frac{|I|^2}{L^2} \les |I|, $$
so we focus on \eqref{LDEmain}. Define the function 
\begin{equation}\label{covariance}
K(z) = \E[\ph(x+z) \overline{\ph(x)}] = \frac{1}{\sqrt{L}} \sum_{k\in\Z} \frac{1}{\alpha +  4\pi^2\big(\frac{k}{L}\big)^2} \frac{e^{i \frac{2\pi kx}{L}}}{\sqrt L}.
\end{equation}
Then, exploiting the fact that $\ph$ is a Gaussian random variable, it is easy to check that 
$$\eqref{LDEmain} = \int_{I\times I} 2|K(x-y)|^2 dx dy. $$
Therefore, by Plancherel, we obtain that 
\begin{align*}
\eqref{LDEmain} \le 2 |I| \|K\|_{L^2}^2 = 2 |I| \cdot  \frac{1}{L} \sum_{k\in\Z} \frac{1}{\Big(\alpha +  4\pi^2\big(\frac{k}{L}\big)^2\Big)^2} \les |I|.
\end{align*}

\end{proof}
\begin{proposition}\label{BDmain}
The measure $\mu_L$ is supported on $C([-L/2,L/2])$. Moreover, as $L \to \infty$, we have that $\mu_L \to \mu_{OU}$ weakly as probability measures over $C(\R)$ equipped the compact-open topology.
\end{proposition}
\begin{proof}
Fix $K > 0$. Recall that by Sobolev embeddings, for any $s>0, p> 1$ with $s > \frac 1p$, we have that  
$$\| \ph \|_{C^{s - \frac{1}{p}}([-K,K])}^p \les |\ph(0)|^p + \int_{-K}^K\int_{-K}^K \frac{|\ph(x)-\ph(y)|^p}{|x-y|^{sp+1}}dxdy.$$
We have that 
$$ \E |\ph(x)-\ph(y)|^2 = \frac{1}{L} \sum_{k\in\Z} \frac{1}{\alpha + 4\pi^2\big(\frac k L\big)^2} |e^{i\frac{2\pi k |x-y|}{L}} - 1|^2 =: H(x-y).$$
It easy to check that $H \ge 0$,   and that for every $0\le \sigma < 1$, 
$$\||z|^{-\sigma} H(z)\|_{L^\infty} \les 1.$$ 
Therefore, by picking $p \gg 1$, $s > \frac 1p $ in such a way that $s + \frac 1p < \frac 12$, we obtain that 
\begin{align*}
\E[\| \ph \|_{C^{s - \frac{1}{p}}([-K,K])}^p] &\les \E |\ph(0)|^p +  \E\Big[\int_{-K}^K\int_{-K}^K \frac{|\ph(x)-\ph(y)|^p}{|x-y|^{sp+1}}dxdy\Big] \\
&=   \E |\ph(0)|^p + \int_{-K}^K\int_{-K}^K \frac{\E|\ph(x)-\ph(y)|^p}{|x-y|^{sp+1}}dxdy \\
&\les 1 +  \int_{-K}^K\int_{-K}^K \frac{|H(z)|^\frac p2}{|x-y|^{sp+1}}dxdy \\
&\les 1 + K^2.
\end{align*}
This shows that the sequence $\{\mu_L\}$ is tight in $C(\R)$, so by Skorohod's theorem, up to subsequences it has a weak limit $\mu_\infty$. 
In order to conclude the proof, it is enough to show that $\mu_\infty = \mu_{OU}$ (irrespective of the subsequence chosen).
Since $\mu_\infty$ is a limit of centered Gaussian measure, it is also a centered Gaussian measure, so it is enough to check its covariance structure. But by \eqref{covariance}, we have that 
$$ \E[\ph(x+z) \overline{\ph(x)}] = K(z) \to \int_{-\infty}^{\infty} \frac{e^{i\xi x}}{\alpha + 4\pi^2|\xi|^2}d\xi, $$
which is exactly the covariance of $\mu_{OU}$.
\end{proof}

\subsection*{Conflict of interest statement:} The authors have no conflict of interest.

\subsection*{Data availability statement:} No datasets were generated or analysed.

\bibliographystyle{abbrv}
\bibliography{phip}

\end{document}